\documentclass{amsart}

\usepackage{amsmath,amssymb,amsthm}

\usepackage{booktabs}
\usepackage[colorlinks=true,linkcolor=blue,citecolor=blue,urlcolor=blue]{hyperref}

\usepackage[backrefs,msc-links]{amsrefs}

\usepackage[capitalise,noabbrev,nameinlink]{cleveref}

\providecommand{\Pic}{\operatorname{Pic}}
\providecommand{\Sym}{\operatorname{Sym}}
 
\providecommand{\F}{\mathbb F}
\providecommand{\Z}{\mathbb Z}

\providecommand{\PP}{\mathbb P}
\providecommand{\A}{\mathbb A}
\providecommand{\Aut}{\operatorname{Aut}}
\providecommand{\PAut}{\operatorname{PAut}}

\providecommand{\Hull}{\operatorname{Hull}}
\providecommand{\Supp}{\operatorname{Supp}}
\providecommand{\Proj}{\operatorname{Proj}}
\providecommand{\Res}{\operatorname{Res}}
\providecommand{\ev}{\operatorname{ev}}
\providecommand{\wt}{\operatorname{wt}}

\providecommand{\dv}{\operatorname{div}}
\providecommand{\Ram}{\operatorname{Ram}}
\providecommand{\XX}{\mathcal X}
\providecommand{\CC}{\mathcal C}
\providecommand{\YY}{\mathcal Y}

\providecommand{\bw}{\mathbf w}

\newcommand{\Pw}{\PP^2_{\bw}}

\theoremstyle{plain}
\newtheorem{thm}{Theorem}[section]
\newtheorem{prop}[thm]{Proposition}
\newtheorem{lem}[thm]{Lemma}
\newtheorem{cor}[thm]{Corollary}
\newtheorem{conj}[thm]{Conjecture}

\theoremstyle{definition}
\newtheorem{defn}[thm]{Definition}
\newtheorem{prob}[thm]{Problem}
\newtheorem{exa}[thm]{Example}
\newtheorem{rem}[thm]{Remark}

\crefname{thm}{Thm.}{Theorems}
\crefname{prop}{Prop.}{Propositions}
\crefname{lem}{Lem.}{Lemmas}
\crefname{cor}{Cor.}{Corollaries}
\crefname{conj}{Conj.}{Conjectures}
\crefname{defn}{Def.}{Definitions}
\crefname{exa}{Exa.}{Examples}
\crefname{rem}{Rem.}{Remarks}
\Crefname{rem}{Remark}{Remarks}
\crefname{equation}{Eq.}{Eqs.}

 \usepackage[scale=.8]{geometry}

\begin{document}

%   front 

\title{Hulls, linear equivalence, and weighted superelliptic codes}

\author{J. Mezinaj}
\address{Department of Mathematics and Statistics, Oakland University, Rochester, MI 48309}
\email{mezinaj@oakland.edu}

\author{T. Shaska}
\address{Department of Computer Science \& Engineering, Oakland University, Rochester, MI 48309}
\email{shaska@oakland.edu} 



\begin{abstract}
The containment of the code of the meet $G \wedge A$ in the hull, and the identity $G \vee A - D = K - G \wedge A$ exchanging meet and join, are known, and imposing that $G \wedge A$ be principal is how algebraic geometry codes with one-dimensional hull are built. We turn that construction into a measurement. For arbitrary divisors $G$ and $A$ we compute $C_L (D,G) \cap C_L (D,A)$ exactly: it is the code of the meet together with an excess $\varepsilon (G,A)$, canonically the quotient of the two. In the dual situation $A = D-G+K$ the excess is a subquotient of $H^1 ( \XX, \mathcal O_\XX ( G \wedge A ) )^\vee$, so it vanishes when the meet is non-special, and where it does not vanish $K - G \wedge A$ is linearly equivalent to an effective divisor. On superelliptic curves $y^n = f(x)$, whose weighted plane models in $\PP^2_{(1, n/c, d/c)}$, $c = \gcd (n,d)$, turn every Riemann-Roch dimension into a lattice count, this divisor-class condition holds automatically whenever its degree is nonnegative, and the excess detects instead the finer condition that a section be decomposable along $G$ and $A$. The meet and the join are invariant under $s \mapsto M-s$, so every asymmetry of the hull profile is excess: on a totally split genus-two curve over $\F_7$ the reflected degrees $s=1$ and $s=4$ present the same class and have different excess. On the range where the excess is blind the hull is $c \, \mu (s) - n \delta + 1 - g_\XX$, $\mu (s) = \min \{ s, M-s \}$, and depends on the curve only through its number of affine rational points. Over a prime field, under one explicit inequality on $(n,d,q)$, $\ell \left( \lfloor M/2 \rfloor D_\infty \right)$ bounds the hull dimension throughout the family, with equality exactly on the totally split locus.
\end{abstract}

\maketitle

\setcounter{tocdepth}{1}
%\tableofcontents

%--------------------------------------------------------------

%-------------------------------------------------------------------------------------------------------------------
\section{Introduction}\label{sec-intro}

A hull is usually something one arranges. To build an algebraic geometry code with a prescribed hull one chooses the divisor so that the answer is forced in advance: a non-special divisor of small degree makes the hull the code of a meet \cite{ca-lo-ma, me-ta-qi}, a principal meet makes it one-dimensional \cite{sok}. The geometry is spent before the computation begins, and what comes out was put in. This paper reverses the arrangement. We compute the hull of an evaluation code under no hypotheses at all, and the part of the answer that is not forced by degrees turns out to be a linear equivalence on the curve.

The mechanism is one identity. Let $C = C_L (D,G)$ have dual $C_L (D,A)$, where $A = D - G + K$ for the canonical divisor attached to the duality. Then the meet and the join of the pair satisfy
\(
G \vee A - D \; = \; K - ( G \wedge A ) ,
\)
an equality of divisors and not of classes. Now the intersection of two evaluation codes always contains the code of their meet, and \cref{thm-meet} computes by how much it exceeds it: by a space of functions with poles bounded by the join, vanishing on $D$, and decomposable along $G$ and $A$. We call its dimension the excess. The identity puts that space inside the Serre dual $H^1 ( \XX, \mathcal O_\XX ( G \wedge A ) )^\vee$, and two complementary consequences follow. Where the meet is non-special the excess vanishes, so the hull is a Riemann-Roch quantity, fixed by a degree, and it knows nothing further. Where the meet is special the excess may survive, and if it does it forces $K - G \wedge A$ to be linearly equivalent to an effective divisor; for a meet of degree $2 g_\XX - 2$ this is the vanishing of a single class in $\Pic ( \XX )$. The hull is blind on the non-special range where the classical constructions operate; outside that range the excess may carry additional information.

To see it one needs a family in which both sides can be computed, and superelliptic curves $\XX : y^n = f(x)$ are such a family. Their weighted plane models identify the codes $\CC_s$ cut out by weighted forms of degree $s$, which are the weighted projective Reed-Muller codes of \cite{au-etal} restricted to $\XX$, with the algebraic geometry codes of $s D_\infty$, and turn every Riemann-Roch dimension in sight into a count of lattice points in a triangle. The four regimes then become explicit degree ranges; when the window is nonempty, the detecting range consists of the two adjacent intervals described below. On the blind one the hull is $c \, \mu (s) - n \delta + 1 - g_\XX$, a tent function of $s$ through $\mu (s) = \min \{ s, M-s \}$, symmetric about the midpoint $M/2$ determined by $(n,d,q)$, and depending on the curve through the number of its affine rational points and through nothing else. This is at once the most concrete formula in the paper and the statement that on that range the method has nothing to say.

Outside it, what the hull sees is finer than a divisor class. The meet and the join are unchanged by $s \mapsto M - s$, so the Riemann-Roch side of the count is symmetric about the midpoint and every asymmetry of the profile $s \mapsto \dim \Hull ( \CC_s )$ is excess, readable from a rank computation with no geometry at all. Yet the two symmetric degrees present the same class to be tested and need not return the same verdict, because the excess depends on the two Riemann-Roch spaces separately and not on their meet and join. On a totally split curve of genus two over $\F_7$ the degrees $s = 1$ and $s = 4$ both present $\Ram + 4 D_\infty - D$, of degree zero; the excess is $0$ at the first and $1$ at the second, and the second value certifies
\(
\Ram + 4 D_\infty \; \sim \; D .
\)
On this family that equivalence is automatic, so what the two degrees separate is not the class but the decomposability of a generator of $L ( K - G \wedge A )$ along the two divisors. A companion curve over $\F_{11}$, optimal and equally split, has the same phenomenon available at the corresponding reflected pair of degrees, but its hull profile is symmetric. This comparison isolates the role of decomposability; a separate example below shows that even curves with identical numerical data $(n,d,q,\delta)$ can have different excess profiles.

Two further statements concern the size of the hull rather than what it detects, and they bound the phenomenon from opposite sides. Clifford's theorem applied to the meet and to its complement gives $\dim \Hull \leq g_\XX + 1$ whenever the meet has degree between $0$ and $2 g_\XX - 2$, so the hull is uniformly bounded throughout the detecting range. The two mechanisms that do make it large are disjoint. One is the window, widest on the totally split locus, where the hull grows linearly in the degree; the other is a stratum on which a finite abelian group acts regularly on the evaluation points, so that the code is an ideal of the group algebra and the hull is measured by the failure of a defining set to be symmetric, a computation with no geometry in it. The window of such a stratum is empty. Over a prime field, under one explicit inequality between $\lfloor M/2 \rfloor$ and $(n-1)d$, the first mechanism dominates and the maximum of $\dim \Hull ( \CC_s )$ over the whole family is $\ell \left( \lfloor M/2 \rfloor D_\infty \right)$, attained exactly on the totally split locus; without it the two conspire, and we exhibit a curve over $\F_{11}$ on which the second exceeds the first by eleven.

A last section removes the hypothesis that $f$ have no rational root, at the opposite extreme from the one in force throughout, where $f$ splits completely over $\F_q$. The ramification then migrates from the join to the meet, every count of \cref{sec-hull} survives with $n \delta$ replaced by $d + n \delta$, and the order relation between the two divisors reverses, so that the upper half of the degree range is dual-containing rather than the lower half self-orthogonal. The Hermitian curve and its cyclic quotients, which the standing hypotheses exclude, enter the family there.

What is missing is a converse. A nonvanishing excess certifies the equivalence, but the equivalence does not certify a nonvanishing excess: the generator must also be decomposable, and the excess is in general strictly smaller than the bound $\ell ( K - G \wedge A )$ that contains it. Which elements are decomposable is the question this paper opens and does not close, and every statement we leave conjectural is a form of it, including the assertion that the whole profile $s \mapsto \varepsilon ( G_s, A_s )$ determines the weighted moduli point of $\XX$. The dimension of the hull is also the parameter governing the support splitting algorithm \cite{se-ssa} and enters the entanglement-assisted quantum-code parameters attached to $C$ \cite{gu-ji-gu}; the end of \cref{sec-comput} records what the computations below say about both.

Throughout, $q$ is a power of a prime $p$, $\F_q$ is the field with $q$ elements, and $\overline{\F}_q$ is a fixed algebraic closure. All curves are smooth, projective and geometrically irreducible unless stated otherwise, and for a curve $\XX$ over $\F_q$ we write $\F_q (\XX)$ for its function field, $g_\XX$ for its genus, and $L(G) = \{ h \in \F_q (\XX)^* : \dv (h) + G \geq 0 \} \cup \{ 0 \}$ and $\ell (G) = \dim_{\F_q} L(G)$ for the Riemann-Roch space of a divisor $G$ and its dimension.

%--------------------------------------------------------------
\section{Preliminaries}\label{sec-prelim}

%\subsection{Weighted projective spaces}\label{ssec-wps}

Let $\bw = (w_0, \dots , w_m)$ be a tuple of positive integers and let $S = \F_q [ x_0, \dots , x_m ]$ be graded by $\wt (x_i) = w_i$, so that a monomial $x_0^{a_0} \cdots x_m^{a_m}$ has weighted degree $\sum_i a_i w_i$ and $S = \bigoplus_{s \geq 0} S_s$, where $S_s$ is the span of the monomials of weighted degree $s$. A polynomial $F \in S_s$ is called weighted homogeneous of degree $s$; equivalently
\[
F ( \lambda^{w_0} x_0, \dots , \lambda^{w_m} x_m ) = \lambda^s \, F ( x_0, \dots , x_m ), \qquad \lambda \in \overline{\F}_q^{\, *}.
\]

\begin{defn}\label{def-wps}
The weighted projective space with weights $\bw$ is $\PP^m_{\bw} = \Proj S$. Its $\overline{\F}_q$-points are the orbits of $\overline{\F}_q^{\, m+1} \setminus \{ 0 \}$ under the action
\[
\lambda \cdot ( x_0, \dots , x_m ) = ( \lambda^{w_0} x_0, \dots , \lambda^{w_m} x_m ), \qquad \lambda \in \overline{\F}_q^{\, *},
\]
and we write $[ x_0 : \cdots : x_m ]$ for the class of $(x_0, \dots , x_m)$. A point is $\F_q$-rational if it is fixed by the Frobenius, and $\PP^m_{\bw} ( \F_q )$ denotes the set of such points.
\end{defn}

The tuple $\bw$ is called well formed if $\gcd ( w_0, \dots , \widehat{w_i}, \dots , w_m ) = 1$ for every $i$. Every weighted projective space is isomorphic to one with well formed weights, and in the case $m = 2$ one has $\PP^2_{(1,a,b)} \cong \PP^2_{(1, ac, bc)}$ for every $c \geq 1$; we assume from now on that all weight systems are well formed. The space $\PP^m_{\bw}$ is the disjoint union of the locally closed strata
\[
U_i = \{ [x_0 : \cdots : x_m] : x_0 = \cdots = x_{i-1} = 0, \; x_i \neq 0 \}, \qquad 0 \leq i \leq m,
\]
and it is a normal projective variety whose singular locus is contained in the union of the strata $U_i$ with $w_i > 1$. We write $( \PP^m_{\bw} )^{\mathrm{sm}}$ for the smooth locus. For the arithmetic of weighted projective spaces over finite fields and for heights on them we refer to \cite{au-etal, sa-sh-h}.

\begin{defn}\label{def-wcurve}
A weighted plane curve of degree $e$ is the zero locus $\XX = V(F) \subseteq \PP^2_{\bw}$ of a weighted homogeneous polynomial $F \in S_e$, $S = \F_q [x_0, x_1, x_2]$, which is geometrically irreducible and reduced. Its homogeneous coordinate ring is $S ( \XX ) = S / (F)$, graded by weighted degree.
\end{defn}

%\subsection{Superelliptic curves}\label{ssec-superell}

\begin{defn}\label{def-superell}
Let $n \geq 2$ and $d \geq 3$ be integers with $p \nmid n$, and let $f \in \F_q [x]$ be a separable polynomial of degree $d$. The superelliptic curve of level $n$ attached to $f$ is the smooth projective model $\XX$ of the affine plane curve
\begin{equation}\label{eq-superell}
y^n = f(x).
\end{equation}
We write $\pi : \XX \to \PP^1$, $(x,y) \mapsto x$, for the degree $n$ covering determined by \cref{eq-superell}, and
\[
\tau : \XX \to \XX, \qquad \tau (x,y) = ( x, \zeta_n y ),
\]
for the geometric superelliptic automorphism, where $\zeta_n \in \overline{\F}_q$ is a fixed primitive $n$-th root of unity.
\end{defn}

Throughout this paper we denote
\(
c = \gcd (n, d)\) , \( \bw = \left( 1, \tfrac{n}{c}, \tfrac{d}{c} \right)\), and \( e = \tfrac{nd}{c}\).
The covering $\pi$ is totally ramified over each of the $d$ roots of $f$, and the fibre over $x = \infty$ consists of $c$ points, each with ramification index $n/c$. We denote by 
\(
\Ram = \sum_{f (a) = 0} Q_a\), and \( D_\infty = \sum_{i=1}^{c} P_i\), 
for the reduced ramification divisor over the roots of $f$ and the reduced divisor at infinity. 

\begin{lem}\label{lem-basic}
With the notation above the genus of $\XX$ is given by 
\begin{equation}\label{eq-genus}
g_\XX = \frac{(n-1)(d-1) + 1 - c}{2},
\end{equation}
and, in $\mathrm{Div} ( \XX )$,
\begin{equation}\label{eq-divisors}
\begin{split}
\dv (x - a) &= n \, Q_a - \tfrac{n}{c} D_\infty \quad \text{for } f(a) = 0, \\
\dv (y) &= \Ram - \tfrac{d}{c} D_\infty, \\
\dv (dx) &= (n-1) \Ram - \left( \tfrac{n}{c} + 1 \right) D_\infty.
\end{split}
\end{equation}
\end{lem}

\begin{proof}
After base change to $\overline{\F}_q$, the extension defined by \cref{eq-superell} is Kummer, and \cref{eq-genus} together with the ramification data recorded above is \cite[Prop.~3.7.3]{st}. The three identities in \cref{eq-divisors} follow from that data: the first because $Q_a$ is the unique point over $a$ and $(x)_\infty = \pi^* ( \infty ) = \tfrac{n}{c} D_\infty$, the second by taking divisors in \cref{eq-superell}, and the third from $\dv (dx) = \pi^* \dv_{\PP^1} (dx) + \mathfrak{d}$ with different $\mathfrak{d} = (n-1) \Ram + ( \tfrac{n}{c} - 1 ) D_\infty$.
\end{proof}

Let $\Aut ( \XX, \pi )$ denote the subgroup of $\Aut ( \XX )$ preserving the fibration $\pi$. The subgroup $\langle \tau \rangle \cong \Z / n \Z$ is normal in $\Aut ( \XX, \pi )$, and the reduced fibration-preserving automorphism group
\[
\overline{\Aut} ( \XX, \pi ) := \Aut ( \XX, \pi ) / \langle \tau \rangle
\]
embeds in $\mathrm{PGL}_2 ( \overline{\F}_q )$ as the group of fractional linear transformations preserving the branch locus of $\pi$. The groups occurring as $\Aut ( \XX )$, together with equations for the corresponding families, are classified in \cite{sa, sa-sh} in every characteristic, and the loci they determine in the moduli space are studied in \cite{mssv}.

%\subsection{The weighted model and the bridge lemma}\label{ssec-model}
Write $f(x) = \sum_{i=0}^{d} a_i x^i$ and let $S = \F_q [ z, x, y ]$ be graded by $\wt (z) = 1$, $\wt (x) = n/c$, $\wt (y) = d/c$. The polynomial
\begin{equation}\label{eq-weighted-model}
F ( z, x, y ) = y^n - \sum_{i=0}^{d} a_i \, x^i z^{\, n (d-i) / c}
\end{equation}
is weighted homogeneous of degree $e = nd/c$, and $\XX \cong V(F) \subseteq \Pw$; the affine equation \cref{eq-superell} is recovered by setting $z = 1$, and the points of $V(F)$ with $z = 0$ are the points of $D_\infty$. We refer to $V(F) \subseteq \Pw$ as the weighted model of $\XX$. For $n = 2$ and $d = 2g+2$ one has $c = 2$ and $\bw = (1,1,g+1)$, the classical weighted model of a hyperelliptic curve.

The following lemma is the mechanism behind everything that follows. It identifies the weighted grading with the filtration by pole order at infinity, and simultaneously with the eigenspace decomposition under $\tau$.

\begin{lem}\label{lem-bridge}
Let $A = \F_q [x,y] / ( y^n - f (x) )$ and let $s \geq 0$. Then:
\begin{enumerate}
\item $A$ is the integral closure of $\F_q [x]$ in $\F_q (\XX)$, and $A = \bigcup_{s \geq 0} L ( s D_\infty )$;
\item the set
\(
B_s = \left\{ \, x^i y^j \; : \; i \geq 0, \; 0 \leq j \leq n-1, \; \tfrac{n}{c} i + \tfrac{d}{c} j \leq s \, \right\}
\)
is an $\F_q$-basis of $L ( s D_\infty )$, and dehomogenisation at $z = 1$ induces an isomorphism of $\F_q$-vector spaces
\(
S ( \XX )_s \; \xrightarrow{\ \sim \ } \; L ( s D_\infty );
\)
\item consequently
\begin{equation}\label{eq-ell}
\ell ( s D_\infty ) = \sum_{j=0}^{n-1} \max \left\{ 0, \; \left\lfloor \frac{cs - dj}{n} \right\rfloor + 1 \right\}.
\end{equation}
\end{enumerate}
\end{lem}

\begin{proof}
Part (1) and the fact that the monomials $x^i y^j$ with $0 \leq j \leq n-1$ form a basis of $A$ as an $\F_q [x]$-module are the standard description of the integral closure in a Kummer extension, \cite[Prop.~3.7.3]{st}. By \cref{lem-basic} the function $x^i y^j$ is regular on the affine part and has polar divisor $( \tfrac{n}{c} i + \tfrac{d}{c} j ) D_\infty$, so $B_s \subseteq L ( s D_\infty )$. Conversely let $h = \sum_{j=0}^{n-1} h_j (x) y^j \in L ( s D_\infty )$ with $h_j \in \F_q [x]$; since $\tau$ permutes the points of $D_\infty$ the space $L ( s D_\infty )$ is $\tau$-stable, and over $\F_q ( \zeta_n )$ the term $h_j (x) y^j$ is the $\zeta_n^j$-eigencomponent of $h$, so it lies in $L ( s D_\infty )$ separately and $\tfrac{n}{c} \deg h_j + \tfrac{d}{c} j \leq s$. This gives (2) for $L ( s D_\infty )$, and the isomorphism because a monomial $z^k x^i y^j$ of weighted degree $s$ has $k = s - \tfrac{n}{c} i - \tfrac{d}{c} j \geq 0$ and is sent to $x^i y^j$, while $F$ reduces the exponent of $y$ modulo $n$. Part (3) is the count of $B_s$.
\end{proof}

Formula \cref{eq-ell} is a piecewise quasi-polynomial in $s$, and it is eventually linear: it agrees with $cs + 1 - g_\XX$ once $cs > 2 g_\XX - 2$, as it must by Riemann-Roch. The identification $S ( \XX )_s \cong L ( s D_\infty )$ is what makes $\{ \CC_s \}_s$ a family of algebraic geometry codes with a single, uniformly described sequence of divisors rather than a collection of unrelated codes. For $c = 1$ the divisor $D_\infty$ is a single rational point, and \cref{eq-ell} shows that its Weierstrass semigroup is the numerical semigroup $\langle n, d \rangle$ generated by the two weights, of Frobenius number $nd - n - d = 2 g_\XX - 1$; the curve is then a $C_{a,b}$ curve in the sense of \cite{sh-wa}.

%\subsection{Evaluation codes and hulls}\label{ssec-codes}

Let $\XX$ be a smooth projective geometrically irreducible curve over $\F_q$, let $D = P_1 + \cdots + P_N$ be a reduced divisor of distinct $\F_q$-rational points, and let $G$ be a divisor with $\Supp G \cap \Supp D = \emptyset$. The evaluation map
\[
\ev : L (G) \longrightarrow \F_q^N , \qquad \ev (h) = \left( h (P_1), \dots , h (P_N) \right)
\]
is then defined, and $C_L (D,G) = \ev \, L (G)$ is the algebraic geometry code attached to $D$ and $G$, of dimension $\ell (G) - \ell (G-D)$. In the superelliptic case we take the following evaluation set.

Let $\YY = \{ P_1, \dots , P_N \} \subseteq \XX ( \F_q ) \cap ( \Pw )^{\mathrm{sm}}$ be a set of $\F_q$-rational points of $\XX$ with $\YY \cap \Supp D_\infty = \emptyset$, and put $D = \sum_{P \in \YY} P$.

\begin{defn}\label{def-code}
For $s \geq 1$ the \textbf{weighted superelliptic code} of degree $s$ is the image $\CC_s = \CC_s ( \XX, \YY, \bw )$ of the evaluation map
\[
\ev_s : S ( \XX )_s \longrightarrow \F_q^N, \qquad \ev_s ( F ) = \left( \frac{F (P_1)}{z (P_1)^s}, \dots , \frac{F(P_N)}{z(P_N)^s} \right).
\]
\end{defn}

The normalization by $z^s$ is the one used for weighted projective Reed-Muller codes in \cite{au-etal}; it is well defined because $z$ has weight one and does not vanish on $\YY$. By \cref{lem-bridge} we have, for every $s$,
\begin{equation}\label{eq-identification}
\CC_s = C_L ( D, s D_\infty ),
\end{equation}
the algebraic geometry code attached to the divisors $D$ and $s D_\infty$, and $\dim \CC_s = \ell ( s D_\infty ) - \ell ( s D_\infty - D )$.

\begin{defn}\label{def-hull}
The \textbf{hull of a linear code} $C \subseteq \F_q^N$ is $\Hull (C) = C \cap C^\perp$, where $C^\perp$ is the dual with respect to the standard bilinear form. The code $C$ is \textbf{self-orthogonal} if $\Hull (C) = C$, \textbf{self-dual} if $C = C^\perp$, and \textbf{linear complementary dual}, abbreviated LCD, if $\Hull (C) = 0$.
\end{defn}

We shall use the classical duality for algebraic geometry codes in the following form, for which we refer to \cite[Ch.~2]{st}. If $\eta$ is a differential with $v_P ( \eta ) = -1$ and $\Res_P ( \eta ) = 1$ for every $P \in \Supp D$, then
\begin{equation}\label{eq-duality}
C_L ( D, G )^\perp = C_L \left( D, \, D - G + \dv ( \eta ) \right).
\end{equation}
Consequently, writing $G_1 \wedge G_2$ and $G_1 \vee G_2$ for the divisors whose coefficient at each place is respectively the minimum and the maximum of those of $G_1$ and $G_2$, so that
\[
L ( G_1 \wedge G_2 ) = L (G_1) \cap L (G_2), \qquad G_1 \wedge G_2 + G_1 \vee G_2 = G_1 + G_2 ,
\]
one has
\begin{equation}\label{eq-hull-bound}
C_L \left( D, \, G \wedge ( D - G + \dv ( \eta ) ) \right) \subseteq \Hull \left( C_L ( D, G ) \right).
\end{equation}
Finally we recall that $\dim \Hull (C)$ is invariant under permutation equivalence of codes, but not under monomial equivalence when $q \geq 4$: by \cite{ca-me-ta-qi-pe} every linear code over $\F_q$ with $q \geq 4$ is monomially equivalent to an LCD code.

%------------------------------------------------------------------------
%------------------------------------------------------------------------
\section{The hull of an evaluation code}\label{sec-window}

The hull of an evaluation code is an intersection of two Riemann-Roch images, and an intersection of images is not necessarily  the image of an intersection. 

Let $\XX$ be a smooth projective geometrically irreducible curve of genus $g_\XX$ over $\F_q$, let $D = P_1 + \cdots + P_N$ be a reduced divisor of distinct $\F_q$-rational points, and let $G$ and $A$ be divisors with $\Supp G \cap \Supp D = \Supp A \cap \Supp D = \emptyset$. Denote by
\[
\begin{split}
R ( G, A ) 			&	:= \big( L(G) + L(A) \big) \cap L ( G \vee A - D ), \\
R_0 ( G, A ) 		&	:= L ( G - D ) + L ( A - D ), \\
\varepsilon ( G, A ) 	&	:= \dim_{\F_q} R (G,A) / R_0 (G,A) .
\end{split}
\]

\begin{thm}\label{thm-meet}
Let $\XX$, $D$, $G$ and $A$ be as above. Then the following hold.
\begin{enumerate}
\item $R_0 (G,A) \subseteq R (G,A)$; in particular $\varepsilon (G,A) \geq 0$.
\item $C_L ( D, G \wedge A ) \; \subseteq \; C_L ( D, G ) \cap C_L ( D, A )$.
\item There is a canonical isomorphism of $\F_q$-vector spaces
\begin{equation}\label{eq-meet-quotient}
\big( C_L ( D, G ) \cap C_L ( D, A ) \big) \big/ C_L ( D, G \wedge A ) \; \cong \; R (G,A) / R_0 (G,A) .
\end{equation}
\item Consequently,
\begin{equation}\label{eq-meet-exact}
\dim \big( C_L ( D, G ) \cap C_L ( D, A ) \big) \; = \; \ell ( G \wedge A ) - \ell ( G \wedge A - D ) + \varepsilon ( G, A ) .
\end{equation}
\end{enumerate}
\end{thm}

\begin{proof}
Evaluation at the points of $\Supp D$ is defined on $L(G)$ and on $L(A)$ by the disjointness of supports. Set
\[
\Gamma = \left\{ ( u, v ) \in L(G) \oplus L(A) \; : \; \ev (u) = \ev (v) \right\}
\]
and consider the two $\F_q$-linear maps
\[
\begin{split}
\varphi &: \Gamma \longrightarrow \F_q^N, \qquad \varphi (u,v) = \ev (u), \\
\psi &: \Gamma \longrightarrow \F_q ( \XX ), \qquad \psi (u,v) = u - v .
\end{split}
\]

The image of $\varphi$ is $C_L (D,G) \cap C_L (D,A)$, since every element of the intersection is of the form $\ev (u) = \ev (v)$ with $u \in L(G)$ and $v \in L(A)$, and its kernel is $L(G-D) \oplus L(A-D)$.

The image of $\psi$ is $R (G,A)$. Indeed $L(G)$ and $L(A)$ are contained in $L ( G \vee A )$, so $u - v \in L ( G \vee A )$, and $\ev (u-v) = 0$ places $u-v$ in $L ( G \vee A - D )$; as $u - v \in L(G) + L(A)$ this gives $\psi ( \Gamma ) \subseteq R (G,A)$. 

Conversely any $w \in R (G,A)$ may be written $w = u - v$ with $u \in L(G)$ and $v \in L(A)$, and $\ev (u) - \ev (v) = \ev (w) = 0$, so $(u,v) \in \Gamma$ and $\psi (u,v) = w$. The kernel of $\psi$ is the diagonal copy of $L(G) \cap L(A) = L ( G \wedge A )$.

Next, $\varphi ( \ker \psi ) = \ev \, L ( G \wedge A ) = C_L ( D, G \wedge A )$ and $\psi ( \ker \varphi ) = L(G-D) + L(A-D) = R_0 (G,A)$. Since the image of a kernel is contained in the image of $\Gamma$, this proves (1) and (2): $R_0 (G,A) = \psi ( \ker \varphi ) \subseteq \psi ( \Gamma ) = R (G,A)$, and $C_L ( D, G \wedge A ) = \varphi ( \ker \psi ) \subseteq \varphi ( \Gamma ) = C_L (D,G) \cap C_L (D,A)$.

For (3) we claim that
\[
\varphi^{-1} \left( C_L ( D, G \wedge A ) \right) \; = \; \ker \varphi + \ker \psi \; = \; \psi^{-1} \left( R_0 (G,A) \right) .
\]
In both equalities the containment of the right hand side in the left is the previous paragraph. For the first, let $(u,v) \in \Gamma$ with $\varphi (u,v) = \ev (w)$ for some $w \in L ( G \wedge A )$. Then $\ev (u-w) = 0$ and $\ev (v-w) = \ev (u) - \ev (w) = 0$, so $u - w \in L(G-D)$ and $v-w \in L(A-D)$, and $(u,v) = (w,w) + (u-w, v-w)$ lies in $\ker \psi + \ker \varphi$. For the second, let $(u,v) \in \Gamma$ with $u - v = a + b$, $a \in L(G-D)$ and $b \in L(A-D)$, and put $w = u - a = v + b$. Then $w \in L(G)$ and $w \in L(A)$, hence $w \in L ( G \wedge A )$, and $(u,v) = (w,w) + (a, -b)$ lies in $\ker \psi + \ker \varphi$.

Both $\varphi$ and $\psi$ therefore induce isomorphisms of $\Gamma / ( \ker \varphi + \ker \psi )$ onto
\[
\left( C_L (D,G) \cap C_L (D,A) \right) / C_L ( D, G \wedge A ) \qquad \text{and onto} \qquad R (G,A) / R_0 (G,A)
\]
respectively, which is \cref{eq-meet-quotient} and proves (3). Finally, taking dimensions in \cref{eq-meet-quotient} and using $\dim C_L ( D, G \wedge A ) = \ell ( G \wedge A ) - \ell ( G \wedge A - D )$ gives (4).
\end{proof}

We call $\varepsilon (G,A)$ the \textbf{excess} of the pair. It is not a Riemann-Roch quantity of $G$ or of $A$: the space $R (G,A)$ consists of the functions with poles bounded by the join which vanish on $D$ and are decomposable as a difference along $G$ and $A$, and $R_0 (G,A)$ of those decomposable along $G - D$ and $A - D$, that is, with each term vanishing on $D$ separately. Since $L(G-D) \cap L(A-D) = L ( G \wedge A - D )$ one has
\[
\dim R_0 (G,A) = \ell (G-D) + \ell (A-D) - \ell ( G \wedge A - D ),
\]
so the excess is determined by $\dim R (G,A)$ together with Riemann-Roch data.

%\subsection{Speciality of the meet}\label{ssec-speciality}

We now specialise the second divisor to $A = D - G + K$, where $K = \dv ( \eta )$ for a differential $\eta$ as in \cref{eq-duality}. By the duality theorem this makes $C_L (D,A)$ the dual of $C_L (D,G)$, so the intersection of \cref{thm-meet} becomes the hull. The gain is an identity of divisors rather than a linear equivalence, and it converts the excess into a cohomological quantity attached to the meet alone.

\begin{thm}\label{thm-canonical}
Let $\XX$, $D$ and $G$ be as above, and suppose that $A = D - G + K$, where $K = \dv ( \eta )$ for a differential $\eta$ as in \cref{eq-duality}, so that $C_L (D,A) = C_L (D,G)^\perp$ and $C_L (D,G) \cap C_L (D,A) = \Hull \left( C_L (D,G) \right)$. Then the following hold.
\begin{enumerate}
\item As divisors,
\begin{equation}\label{eq-join-canonical}
G \vee A - D \; = \; K - ( G \wedge A ) ,
\end{equation}
so that $L ( G \vee A - D ) = L ( K - G \wedge A )$, the Serre dual of $H^1 \left( \XX, \mathcal O_\XX ( G \wedge A ) \right)$.
\item One has $R (G,A) \subseteq L ( K - G \wedge A )$, so \cref{eq-meet-quotient} is an isomorphism of $\Hull \left( C_L (D,G) \right) \big/ C_L ( D, G \wedge A )$ onto the subquotient $R (G,A) / R_0 (G,A)$ of $H^1 \left( \XX, \mathcal O_\XX ( G \wedge A ) \right)^{\vee}$, and
\begin{equation}\label{eq-eps-h1}
0 \; \leq \; \varepsilon (G,A) \; \leq \; \ell ( K - G \wedge A ) \; = \; \dim_{\F_q} H^1 \left( \XX, \mathcal O_\XX ( G \wedge A ) \right) .
\end{equation}
\item If $G \wedge A$ is non-special, then $\varepsilon (G,A) = 0$ and
\[
\Hull \left( C_L (D,G) \right) \; = \; C_L ( D, G \wedge A ) .
\]
\end{enumerate}
\end{thm}

\begin{proof}
For (1), from $A = D - G + K$ we get $G + A = D + K$, and $G \wedge A + G \vee A = G + A$ gives \cref{eq-join-canonical}; equal divisors have equal Riemann-Roch spaces, and by duality and Riemann-Roch, \cite[Ch.~1]{st}, the space $L ( K - G \wedge A )$ is dual to $H^1 \left( \XX, \mathcal O_\XX ( G \wedge A ) \right)$, of dimension $\ell ( G \wedge A ) - \deg ( G \wedge A ) - 1 + g_\XX$, the index of speciality of $G \wedge A$.

For (2), by definition $R (G,A) \subseteq L ( G \vee A - D )$, which is $L ( K - G \wedge A )$ by (1); hence $R (G,A) / R_0 (G,A)$ is a subquotient of $H^1 \left( \XX, \mathcal O_\XX ( G \wedge A ) \right)^{\vee}$, \cref{eq-meet-quotient} is an isomorphism onto it, and \cref{eq-eps-h1} follows.

For (3), if $G \wedge A$ is non-special then $L ( K - G \wedge A ) = 0$, hence $R (G,A) = 0$ and $\varepsilon (G,A) = 0$; the two codes then have the same dimension by \cref{eq-meet-exact}, and the containment of \cref{thm-meet}(2) makes them equal.
\end{proof}

\begin{cor}\label{cor-window}
Let $\XX$, $D$, $G$ and $A = D - G + K$ be as in \cref{thm-canonical}. Then the following hold.
\begin{enumerate}
\item As divisors,
\[
\deg ( G \wedge A ) + \deg ( G \vee A ) \; = \; N + 2 g_\XX - 2 ;
\]
in particular the conditions \( \deg ( G \vee A ) < N \) and \( \deg ( G \wedge A ) > 2 g_\XX - 2 \) are equivalent.
\item If $\deg ( G \wedge A ) > 2 g_\XX - 2$, then
\begin{equation}\label{eq-window-hull}
\Hull \big( C_L (D,G) \big) = C_L ( D, G \wedge A ) \quad \text{ and } \quad \dim \Hull \big( C_L (D,G) \big) = \deg ( G \wedge A ) + 1 - g_\XX .
\end{equation}
\end{enumerate}
\end{cor}

\begin{proof}
For (1), the degree identity is \cref{eq-join-canonical} together with $\deg K = 2 g_\XX - 2$, and the equivalence of the two conditions follows.

For (2), a divisor of degree greater than $2 g_\XX - 2$ is non-special, so \cref{thm-canonical}(3) gives $\Hull ( C_L (D,G) ) = C_L ( D, G \wedge A )$, and Riemann-Roch gives $\ell ( G \wedge A ) = \deg ( G \wedge A ) + 1 - g_\XX$. Finally $G \wedge A \leq G \vee A$ gives $\deg ( G \wedge A ) \leq \deg ( G \vee A ) < N$, so $\ell ( G \wedge A - D ) = 0$ and the code has the dimension of the space.
\end{proof}

%\subsection{Detection}\label{ssec-detect}

%\cref{cor-window} says when the excess is blind. The complementary statement is that when it is not blind, what it sees is a linear equivalence.

For the pair of \cref{thm-canonical}, denote by
\(
t (G) := \deg ( G \vee A - D ) = 2 g_\XX - 2 - \deg ( G \wedge A ) ,
\)
the equality by \cref{eq-join-canonical}.

\begin{cor}\label{cor-detect}
Let $\XX$, $D$, $G$ and $A = D - G + K$ be as in \cref{thm-canonical}. If $\varepsilon (G,A) \neq 0$, then $K - G \wedge A$ is linearly equivalent over $\F_q$ to an effective divisor of degree $t (G)$, that is, its class lies in the image $W_{t (G)} ( \XX )$ of $\Sym^{t (G)} \XX \to \Pic^{t (G)} ( \XX )$. In particular:
\begin{enumerate}
\item if $t (G) < 0$, then $\varepsilon (G,A) = 0$;
\item if $t (G) = 0$, then $\varepsilon (G,A) \in \{ 0, 1 \}$, and $\varepsilon (G,A) = 1$ forces
\[
G \vee A \sim D , \qquad \text{equivalently} \qquad G \wedge A \sim K ;
\]
that is, the class of $K - G \wedge A$, one of the $| \Pic^0 ( \XX ) ( \F_q ) |$ classes of degree zero, is trivial;
\item if $0 < t (G) < g_\XX$, then a nonzero excess forces the class of $K - G \wedge A$ into $W_{t (G)} ( \XX )$, a proper subvariety of $\Pic^{t (G)} ( \XX )$, since by \cite[Ch.~I]{acgh} it has dimension $t (G) < g_\XX$.
\end{enumerate}
\end{cor}

\begin{proof}
By \cref{eq-eps-h1} a nonzero excess forces $\ell ( K - G \wedge A ) > 0$, that is, the existence of $h \in \F_q ( \XX )^*$ with $\dv (h) + K - G \wedge A \geq 0$; the divisor $\dv (h) + K - G \wedge A$ is then effective, linearly equivalent to $K - G \wedge A$, and of degree $\deg ( K - G \wedge A ) = t (G)$. This is the main assertion.

For (1), a divisor of negative degree has no nonzero global sections, so $\ell ( K - G \wedge A ) = 0$ and $\varepsilon (G,A) = 0$ by \cref{eq-eps-h1}.

For (2), a divisor of degree zero has $\ell \leq 1$, with equality if and only if it is principal; so $\varepsilon (G,A) \leq \ell ( K - G \wedge A ) \leq 1$ by \cref{eq-eps-h1}, and $\varepsilon (G,A) = 1$ forces $K - G \wedge A \sim 0$, that is $G \wedge A \sim K$, which by \cref{eq-join-canonical} is $G \vee A \sim D$.

For (3), the main assertion places the class in $W_{t (G)} ( \XX )$, and the quoted dimension makes that a proper subvariety.
\end{proof}

\begin{cor}\label{cor-clifford}
Let $\XX$, $D$, $G$ and $A = D - G + K$ be as in \cref{thm-canonical}. Then the following hold.
\begin{enumerate}
\item $\dim \Hull \left( C_L (D,G) \right) \leq \ell ( G \wedge A ) + \ell ( K - G \wedge A )$.
\item If $0 \leq \deg ( G \wedge A ) \leq 2 g_\XX - 2$, equivalently $0 \leq t (G) \leq 2 g_\XX - 2$, then
\[
\dim \Hull \left( C_L (D,G) \right) \leq g_\XX + 1 .
\]
\end{enumerate}
\end{cor}

\begin{proof}
For (1), the bound is \cref{eq-meet-exact} together with $\ell ( G \wedge A - D ) \geq 0$ and \cref{eq-eps-h1}.

For (2), $\deg ( K - G \wedge A ) = t (G) = 2 g_\XX - 2 - \deg ( G \wedge A )$ by \cref{eq-join-canonical}, so both $G \wedge A$ and $K - G \wedge A$ have degree in $[ 0, 2 g_\XX - 2 ]$ and Clifford's theorem \cite[Thm.~1.6.13]{st} applies to each:
\[
\begin{split}
\ell ( G \wedge A ) &\leq \tfrac{1}{2} \deg ( G \wedge A ) + 1 , \\
\ell ( K - G \wedge A ) &\leq \tfrac{1}{2} \left( 2 g_\XX - 2 - \deg ( G \wedge A ) \right) + 1 ,
\end{split}
\]
and the two right hand sides add to $g_\XX + 1$.
\end{proof}

\cref{cor-window} and \cref{cor-clifford} divide the range of $\deg ( G \wedge A )$ between them: above $2 g_\XX - 2$ the hull is the Riemann-Roch quantity of \cref{cor-window} and grows with the degree, while on $[ 0, 2 g_\XX - 2 ]$ it is at most $g_\XX + 1$ whatever the excess does, by a bound involving neither $D$ nor $G$. Thus throughout the intermediate range $0 \leq \deg ( G \wedge A ) \leq 2 g_\XX - 2$ the hull is uniformly bounded, even when the excess is nonzero.

The excess also vanishes for a reason having nothing to do with speciality.

\begin{prop}\label{prop-degenerate}
Let $\XX$, $D$, $G$ and $A$ be as above. If $G \leq A$, then $G \wedge A = G$, $C_L (D,G) \subseteq C_L (D,A)$ and $\varepsilon (G,A) = 0$; since $\varepsilon (A,G) = \varepsilon (G,A)$, the excess vanishes also when $A \leq G$. If $A = D - G + K$ as in \cref{thm-canonical}, then $G \leq A$ says that $C_L (D,G)$ is self-orthogonal, and then $\Hull ( C_L (D,G) ) = C_L (D,G)$.
\end{prop}

\begin{proof}
If $G \leq A$ then $G \wedge A = G$ and $G \vee A = A$, so $L(G) + L(A) = L(A)$ and $R (G,A) = L(A) \cap L(A-D) = L(A-D)$, while $R_0 (G,A) = L(G-D) + L(A-D) = L(A-D)$ as well; hence $\varepsilon (G,A) = 0$. The three spaces $R (G,A)$, $R_0 (G,A)$ and $L ( G \wedge A )$ are symmetric in $G$ and $A$, whence $\varepsilon (A,G) = \varepsilon (G,A)$. The containment of codes is clear, and under the hypothesis of \cref{thm-canonical} it reads $C_L (D,G) \subseteq C_L (D,G)^\perp$.
\end{proof}

The classical route from an algebraic geometry code to a quantum code arranges exactly such an order relation: in \cite{el-sh} a divisor $G$ invariant under an involution is chosen so that $D - G + \dv ( \eta ) \leq G$, and it is this relation, making the code contain its dual, that feeds the stabiliser construction. By \cref{prop-degenerate} such a choice forces $\varepsilon = 0$, so the construction is carried out entirely inside the zone where the hull is determined by the order relation between the two divisors and, by \cref{cor-detect}, detects nothing about the curve.

\begin{defn}\label{defn-zones}
Let $G$ and $A = D - G + K$ be as in \cref{thm-canonical}. The pair is called
\begin{enumerate}
\item[(Z0)] \textbf{degenerate} if $G \leq A$ or $A \leq G$;
\item[(Z1)] \textbf{blind} if $t (G) < 0$;
\item[(Z2)] \textbf{detecting} if $0 \leq t (G) \leq g_\XX - 1$;
\item[(Z3)] \textbf{saturated} if $t (G) \geq g_\XX$.
\end{enumerate}
\end{defn}

These are the four regimes of the excess, and only one of them carries information about the curve. In (Z0) the meet is one of the two divisors and $\varepsilon = 0$ by \cref{prop-degenerate}, for reasons of divisor order alone; in (Z1) the meet is non-special and $\varepsilon = 0$ by \cref{cor-window}, so the hull is determined by $\deg ( G \wedge A )$ and by nothing else. The two zones may overlap and either may be empty. In (Z2) the vanishing of $\varepsilon$ is no longer automatic, and nonvanishing imposes the proper closed condition that the class of $K - G \wedge A$ lie in $W_{t(G)} ( \XX )$. This is sharpest at $t (G) = 0$, where by \cref{cor-detect} a nonzero excess forces a single class in $\Pic^0 ( \XX )$ to vanish, among at least $( \sqrt q - 1 )^{2 g_\XX}$ classes by the Weil bounds. In (Z3) one has $W_{t (G)} ( \XX ) = \Pic^{t (G)} ( \XX )$, the condition is vacuous, and what the pair carries is the value of $\varepsilon$ rather than its vanishing.

\cref{cor-detect} asserts no converse: the linear equivalence is necessary for $\varepsilon \neq 0$ but not sufficient, since the generator of $L ( K - G \wedge A )$ must in addition be decomposable along $G$ and $A$. The excess is not a function of the pair $( G \wedge A, G \vee A )$, and the failure of the converse is exactly the extra information it carries.

%----------------------------------------------------------------
%----------------------------------------------------------------
\section{Superelliptic curves: the computable case}\label{sec-hull}

\cref{sec-window} leaves two quantities to evaluate: the Riemann-Roch dimension of the meet and the excess. On a superelliptic curve both are computable in closed form. \cref{lem-meetjoin} determines the meet and the join of the pair $( G_s, A_s )$, and \cref{prop-window} turns the four zones of \cref{defn-zones} into intervals of degrees. \cref{thm-quasi} reduces the hull to a lattice count plus the excess. On the window the excess vanishes and the hull depends on the curve only through the integer $\delta$ (\cref{thm-hull}). The meet and the join are invariant under $s \mapsto M - s$, so every asymmetry of the hull profile is excess (\cref{prop-asym}).

%\subsection{Standing hypotheses}\label{ssec-standing}
It makes the discussion easier if we denote the following conditions as: 

\begin{enumerate}
\item[(H1)] $p \nmid n$, the polynomial $f \in \F_q [x]$ is separable of degree $d \geq 3$, and $f$ has no root in $\F_q$.
\item[(H2)] $\YY$ is the set of all $\F_q$-rational points of $\XX$ lying over $\A^1 ( \F_q )$, and $D = \sum_{P \in \YY} P$, $N = \deg D = | \YY |$.
\item[(H3)] $E = \sum_{a \in \F_q} \pi^* (a) - D$ is the effective divisor collecting the points of $\pi^{-1} ( \A^1 ( \F_q ) )$ that are not $\F_q$-rational.
\item[(H4)] $n \mid q-1$, so that $\zeta_n \in \F_q$.
\end{enumerate}

\begin{lem}\label{lem-standing}
Assume \textup{(H1)}--\textup{(H3)}. Then $\Ram$, $E$ and $D_\infty$ are reduced and have pairwise disjoint supports, all three are disjoint from $\Supp D$, and
\(
\deg \Ram = d\), \(\deg D_\infty = c\), and \( \deg E = qn - N\). 
Moreover $E = 0$ if and only if $N = qn$, and $g_\XX \geq 1$.
\end{lem}

\begin{proof}
By (H1) no root of $f$ lies in $\F_q$, so $\pi$ is unramified over $\A^1 ( \F_q )$ and $\pi^* (a)$ is reduced of degree $n$ for every $a \in \F_q$. Hence $D$ and $E$ are reduced with disjoint supports, $\deg D + \deg E = qn$, and $E = 0$ if and only if $N = qn$. By definition $\Ram$ is the sum of the $d$ points over the roots of $f$, none of which lies in $\F_q$, and $D_\infty$ is the sum of the $c$ points over $x = \infty$. Hence $\deg \Ram = d$, $\deg D_\infty = c$, and the three supports are pairwise disjoint and disjoint from $\Supp D$. It remains to show $g_\XX \geq 1$, which by \cref{eq-genus} is the inequality $(n-1)(d-1) \geq c + 1$. Since $c \mid n$ and $c \mid d$ we have $c \leq \min \{ n, d \}$. For $c \leq 2$ the inequality follows from $n \geq 2$ and $d \geq 3$, and for $c \geq 3$ from $(n-1)(d-1) \geq (c-1)^2 \geq c+1$.
\end{proof}

Throughout we denote by 
\begin{equation}
\label{def-delta}
\Delta = \left\{ \, a \in \F_q \; : \; f(a) \notin ( \F_q^* )^n \, \right\}, \quad \delta = | \Delta |, \quad
e_\Delta (x) = \prod_{a \in \Delta} ( x - a ) \; \in \; \F_q [x] .
\end{equation}

\begin{lem}\label{lem-pullback}
Assume \textup{(H1)}--\textup{(H4)}. Then $E = \pi^* ( \Delta )$, with geometric points the $n \delta$ pairs $(a,b)$ where $a \in \Delta$ and $b^n = f(a)$. Moreover
\[
N = n \, ( q - \delta ), \quad \deg E = n \delta , \quad \dv \left( e_\Delta (x) \right) = E - \frac{n \delta}{c} \, D_\infty ,
\]
and $e_\Delta$ vanishes at no point of $\YY$.
\end{lem}

\begin{proof}
Let $a \in \F_q$. By (H1) the fibre $\pi^* (a)$ is reduced of degree $n$, and its points are the pairs $(a,b)$ with $b^n = f(a) \neq 0$. If one such $b$ lies in $\F_q$ then all of them do, since they are the $b \zeta_n^k$ and $\zeta_n \in \F_q$ by (H4). Hence either $\pi^* (a)$ consists of $n$ rational points, and then $a \notin \Delta$ and $\pi^* (a) \leq D$, or it contains no rational point, and then $a \in \Delta$ and $\pi^* (a) \leq E$. This proves $E = \pi^* ( \Delta )$ and the description of the geometric points. It follows that $\deg E = n \delta$ and $N = qn - n \delta$. For the divisor of $e_\Delta$, each $a \in \Delta$ is not a root of $f$. Hence $\dv ( x - a ) = \pi^* (a) - \tfrac{n}{c} D_\infty$, because $(x)_\infty = \tfrac{n}{c} D_\infty$ by \cref{lem-basic}. Summing over $a \in \Delta$ gives the assertion. Finally every point of $\YY$ lies over some $a \notin \Delta$, so $e_\Delta$ does not vanish on $\YY$.
\end{proof}

Throughout we denote  \(M = \frac{(q-1)n}{c} - 1\) and \(\mu (s) = \min \{ s, \, M - s \}\), 
and we attach to each $s \geq 1$ the divisor $G_s = s D_\infty$, so that $\CC_s = C_L ( D, G_s )$ by \cref{eq-identification}.
%
%\subsection{The dual family}\label{ssec-dualfamily}
%
The differential below is the standard device producing residues equal to one at the evaluation points. See \cite{st}, and \cite{el-sh} for its use on superelliptic curves.

\begin{prop}\label{prop-dual}
Assume \textup{(H1)}--\textup{(H3)}. Then the differential $\eta = - dx / ( x^q - x )$ satisfies $v_P ( \eta ) = -1$ and $\Res_P ( \eta ) = 1$ for every $P \in \Supp D$, and
\begin{equation}\label{eq-eta}
K \; = \; \dv ( \eta ) \; = \; (n-1) \Ram - D - E + \left( \frac{(q-1) n}{c} - 1 \right) D_\infty .
\end{equation}
Consequently, for every $s \geq 1$,
\begin{equation}\label{eq-dual}
\CC_s^\perp = C_L ( D, A_s ), \qquad A_s = (n-1) \Ram - E + ( M - s ) D_\infty .
\end{equation}
\end{prop}

\begin{proof}
Let $P \in \Supp D$ lie over $a \in \F_q$. Since $P$ is unramified, $x - a$ is a uniformiser at $P$, and $x^q - x = \prod_{b \in \F_q} (x - b)$ vanishes to order one at $P$. Hence $v_P ( \eta ) = -1$ and
\[
\Res_P ( \eta ) = - \frac{1}{\prod_{b \neq a} (a - b)} = - \frac{1}{(x^q - x)' (a)} = 1,
\]
because $(x^q - x)' = -1$. For \cref{eq-eta}, by \cref{lem-basic} and the definition of $E$,
\[
\begin{split}
\dv ( x^q - x ) &= \sum_{a \in \F_q} \pi^* (a) - \frac{qn}{c} D_\infty = D + E - \frac{qn}{c} D_\infty, \\
\dv ( \eta ) &= \dv (dx) - \dv ( x^q - x ) = (n-1) \Ram - \left( \tfrac{n}{c} + 1 \right) D_\infty - D - E + \tfrac{qn}{c} D_\infty,
\end{split}
\]
which is \cref{eq-eta}. Now \cref{eq-dual} follows from \cref{eq-duality} with $G = G_s$, since $D - G_s + \dv ( \eta ) = A_s$.
\end{proof}

\begin{lem}\label{lem-meetjoin}
Assume \textup{(H1)}--\textup{(H4)}, let $s \geq 1$, and let $A_s$ be as in \cref{eq-dual}. Then
\begin{equation}\label{eq-meetjoin}
\begin{split}
G_s \wedge A_s &= \mu (s) \, D_\infty - E, \\
G_s \vee A_s &= (n-1) \Ram + \max \{ s, \, M-s \} \, D_\infty .
\end{split}
\end{equation}
In particular,
\begin{equation}\label{eq-degrees}
\begin{split}
\deg ( G_s \wedge A_s ) &= c \, \mu (s) - n \delta , \\
\deg ( G_s \vee A_s ) &= c \max \{ s, \, M-s \} + (n-1) d .
\end{split}
\end{equation}
\end{lem}

\begin{proof}
By \cref{lem-standing} the divisors $\Ram$, $E$ and $D_\infty$ have pairwise disjoint supports. The coefficientwise minimum and maximum of $G_s$ and $A_s$ may therefore be computed one support at a time. By \cref{eq-dual}, along $\Supp \Ram$ the coefficients are $0$ and $n-1$, along $\Supp E$ they are $0$ and $-1$, along $\Supp D_\infty$ they are $s$ and $M-s$, and elsewhere both vanish. This gives \cref{eq-meetjoin}, and \cref{eq-degrees} follows from \cref{lem-standing,lem-pullback}.
\end{proof}

%\subsection{The four zones as intervals of degrees}\label{ssec-windowdef}

By \cref{lem-meetjoin} the quantity $t ( G_s )$ of \cref{cor-detect} is
\begin{equation}\label{eq-tofs}
t (s) \; := \; t ( G_s ) \; = \; 2 g_\XX - 2 - c \, \mu (s) + n \delta ,
\end{equation}
a function of $s$ alone once $(n,d,q,\delta)$ are fixed. The function $\mu$ is a tent function, unimodal with slopes $\pm 1$ and maximum $\lfloor M/2 \rfloor$. Hence $t$ is convex and piecewise affine with slopes $\mp c$, with minimum $2 g_\XX - 2 - c \lfloor M/2 \rfloor + n \delta$ at $s = \lfloor M/2 \rfloor$. The conditions (Z1)--(Z3) of \cref{defn-zones} depend on $s$ only through $t (s)$. Since $t$ is convex, each of them defines a union of at most two intervals of degrees.

\begin{defn}\label{defn-window}
The \textbf{window} and the \textbf{detection band} of $\XX$ are
\[
W ( \XX ) := \left\{ \, s \geq 1 \; : \; t (s) < 0 \, \right\} , \quad B ( \XX ) := \left\{ \, s \geq 1 \; : \; 0 \leq t (s) \leq g_\XX - 1 \, \right\} .
\]
\end{defn}

\begin{prop}\label{prop-window}
Assume \textup{(H1)}--\textup{(H4)} and set
\[
\mu_0 = \left\lfloor \frac{2 g_\XX - 2 + n \delta}{c} \right\rfloor + 1 , \qquad \mu_1 = \left\lceil \frac{g_\XX - 1 + n \delta}{c} \right\rceil .
\]
Then the following hold.
\begin{enumerate}
\item $W ( \XX ) = \left\{ \, s \; : \; \mu_0 \leq s \leq M - \mu_0 \, \right\}$, an interval of integers symmetric about $M/2$, and $\varepsilon ( G_s, A_s ) = 0$ for every $s \in W ( \XX )$.
\item $W ( \XX ) \neq \emptyset$ if and only if
\begin{equation}\label{eq-window-nonempty}
c \left\lfloor M/2 \right\rfloor - n \delta \; > \; 2 g_\XX - 2 .
\end{equation}
\item Put $u = \min \{ \mu_0 - 1, \lfloor M/2 \rfloor \}$. Then
\[
B ( \XX ) = \left\{ \, s \geq 1 \; : \; \mu_1 \leq \mu (s) \leq u \, \right\} .
\]
If $\mu_1 > u$ then $B ( \XX ) = \emptyset$. If $W ( \XX ) \neq \emptyset$, then $u = \mu_0 - 1$ and
\[
B ( \XX ) = \left\{ \, s \; : \; \mu_1 \leq s \leq \mu_0 - 1 \, \right\} \cup \left\{ \, s \; : \; M - \mu_0 + 1 \leq s \leq M - \mu_1 \, \right\},
\]
two intervals disjoint from $W ( \XX )$ and adjacent to it on either side. In this case, if $c = 1$ and $\mu_1 \geq 1$, each interval has exactly $g_\XX$ elements.
\end{enumerate}
\end{prop}

\begin{proof}
By \cref{eq-tofs} the condition $t (s) < 0$ reads $c \, \mu (s) > 2 g_\XX - 2 + n \delta$, that is $\mu (s) \geq \mu_0$, and the condition $0 \leq t (s) \leq g_\XX - 1$ reads $\mu_1 \leq \mu (s) \leq \mu_0 - 1$. Since $\mu (s) \leq \lfloor M/2 \rfloor$, the latter condition is equivalent to $\mu_1 \leq \mu (s) \leq u$, which gives (3). If the window is nonempty, then $\mu_0 \leq \lfloor M/2 \rfloor$, hence $u = \mu_0 - 1$; using $\mu (s) = s$ for $s \leq M/2$ and $\mu (s) = M - s$ for $s \geq M/2$ gives the two displayed intervals. The vanishing of the excess in (1) is \cref{cor-window}, whose hypothesis $\deg ( G_s \wedge A_s ) > 2 g_\XX - 2$ is $t (s) < 0$ by \cref{eq-tofs}. For (2), the window is nonempty precisely when $\mu$ attains $\mu_0$, that is when $\lfloor M/2 \rfloor \geq \mu_0$, which is \cref{eq-window-nonempty}. Finally, for $c = 1$ one has $\mu_0 = 2 g_\XX - 1 + n \delta$ and $\mu_1 = g_\XX - 1 + n \delta$, so $\mu_0 - \mu_1 = g_\XX$, and the hypothesis $\mu_1 \geq 1$ places the lower interval inside the range $s \geq 1$. The hypothesis fails only when $g_\XX = 1$ and $\delta = 0$, that is for a totally split curve $y^2 = f(x)$ with $\deg f = 3$, where the lower interval is empty.
\end{proof}

The degenerate zone (Z0) is also an interval, and it is nonempty exactly on the totally split locus.

\begin{defn}\label{def-split}
The curve $\XX$ is \textbf{totally split} over $\F_q$ if every fibre of $\pi$ over $\A^1 ( \F_q )$ consists of $n$ rational points, that is if $f(a) \in ( \F_q^* )^n$ for every $a \in \F_q$. Equivalently $\Delta = \emptyset$, or $\delta = 0$, or $E = 0$, or $N = qn$.
\end{defn}

\begin{prop}\label{prop-dichotomy}
Assume \textup{(H1)}--\textup{(H4)} and let $s \geq 1$. Then $A_s \not\leq G_s$ for every $s$, and $G_s \leq A_s$ if and only if $\XX$ is totally split and $s \leq \lfloor M/2 \rfloor$. In that case $\CC_s$ is self-orthogonal and $\varepsilon ( G_s, A_s ) = 0$.
\end{prop}

\begin{proof}
By \cref{prop-dual} we have $A_s - G_s = (n-1) \Ram - E + ( M - 2s ) D_\infty$. The coefficient of $A_s - G_s$ along $\Supp \Ram$ is $n - 1 > 0$, so $A_s \not\leq G_s$. If $E \neq 0$ then at any place of $\Supp E$ the coefficient of $A_s - G_s$ is $-1$, so $G_s \not\leq A_s$. If $E = 0$ then $A_s - G_s = (n-1) \Ram + ( M - 2s ) D_\infty$ is effective precisely when $2s \leq M$. The last assertion is \cref{prop-degenerate}.
\end{proof}

The dual computation also fixes the designed distances of the family. These are the parameters required by the applications.

\begin{cor}\label{cor-distances}
Assume \textup{(H1)}--\textup{(H4)}. Then for every $s \geq 1$
\begin{equation}\label{eq-degA}
\deg A_s = (n-1) d - n \delta + c ( M - s ) = N + 2 g_\XX - 2 - cs ,
\end{equation}
and, writing $d ( \, \cdot \, )$ for the minimum distance,
\[
\begin{split}
d ( \CC_s ) &\geq N - cs \qquad \text{if } cs < N , \\
d ( \CC_s^\perp ) &\geq cs - 2 g_\XX + 2 \qquad \text{if } cs > 2 g_\XX - 2 .
\end{split}
\]
Both hypotheses hold for every $s \in W ( \XX )$, and whenever both hold
\[
d ( \CC_s ) + d ( \CC_s^\perp ) \; \geq \; N - 2 g_\XX + 2 ,
\]
so the two designed distances trade off linearly in $s$ with constant sum.
\end{cor}

\begin{proof}
The first expression for $\deg A_s$ is \cref{eq-dual} together with $\deg \Ram = d$, $\deg E = n \delta$ and $\deg D_\infty = c$. The second follows from $G_s + A_s = D + K$ and $\deg K = 2 g_\XX - 2$. A nonzero $h \in L(G)$ has at most $\deg G$ zeros on $\Supp D$, so a nonzero codeword of $C_L (D,G)$ has weight at least $N - \deg G$. Applying this to $G_s$ and to $A_s$ and using \cref{eq-degA} gives the two bounds, and adding them gives the last one. For $s \in W ( \XX )$ \cref{prop-window} gives $c \mu (s) > 2 g_\XX - 2 + n \delta$, hence $cs \geq c \mu (s) > 2 g_\XX - 2$, while $c \max \{ s, M-s \} + (n-1) d = N + t (s) < N$ by \cref{eq-degrees,eq-tofs}.
\end{proof}

%\subsection{The Riemann-Roch side as a lattice count}\label{ssec-pullback}

The divisor $\mu (s) D_\infty - E$ has support at both ends of the curve and is not of the form treated by \cref{lem-bridge}. The next lemma removes $E$ at the cost of a uniform shift of the degree. The whole family is then again governed by lattice points in the weighted triangle. For $s \geq 1$ define
\begin{equation}\label{def-nu}
\nu (s) \; = \; \mu (s) - \frac{n \delta}{c} .
\end{equation}
This is an integer because $c \mid n$.

For $\nu \in \Z$ define
\begin{equation}\label{eq-Phi}
\Phi ( \nu ) \; = \; \sum_{j=0}^{n-1} \max \left\{ 0, \; \left\lfloor \frac{c \nu - dj}{n} \right\rfloor + 1 \right\} .
\end{equation}
For $\nu \geq 0$ this is \cref{eq-ell}, and for $\nu < 0$ every summand vanishes, so $\Phi ( \nu ) = \ell ( \nu D_\infty )$ for every $\nu \in \Z$.

\begin{lem}\label{lem-shift}
Assume \textup{(H1)}--\textup{(H4)}, let $\mu \in \Z$, and let $\nu = \mu - \tfrac{n\delta}{c}$. Then multiplication by $e_\Delta (x)$ is an isomorphism
\[
L \left( \nu D_\infty \right) \; \xrightarrow{\ \sim \ } \; L \left( \mu D_\infty - E \right),
\]
and the diagonal matrix $\operatorname{diag} \left( e_\Delta (P) \right)_{P \in \YY}$ carries $C_L ( D, \nu D_\infty )$ onto $C_L ( D, \mu D_\infty - E )$. In particular the two codes are monomially equivalent and
\begin{equation}\label{eq-shifted-count}
\ell \left( \mu D_\infty - E \right) \; = \; \Phi ( \nu ) .
\end{equation}
\end{lem}

\begin{proof}
For $h \in \F_q ( \XX )^*$ let $h' = h / e_\Delta (x)$. By \cref{lem-pullback},
\[
\dv (h') + \nu D_\infty \; = \; \dv (h) - E + \tfrac{n\delta}{c} D_\infty + \left( \mu - \tfrac{n\delta}{c} \right) D_\infty \; = \; \dv (h) + \mu D_\infty - E ,
\]
so $h \mapsto h'$ is an $\F_q$-linear bijection $L ( \mu D_\infty - E ) \to L ( \nu D_\infty )$. Since $h = h' e_\Delta (x)$ and $e_\Delta (P) \neq 0$ for every $P \in \YY$ by \cref{lem-pullback}, the diagonal matrix carries $C_L ( D, \nu D_\infty )$ onto $C_L ( D, \mu D_\infty - E )$. Finally the bijection gives $\ell ( \mu D_\infty - E ) = \ell ( \nu D_\infty ) = \Phi ( \nu )$, which is \cref{eq-shifted-count}.
\end{proof}

\begin{thm}\label{thm-quasi}
Assume \textup{(H1)}--\textup{(H4)} and let $s \geq 1$ satisfy $c \, \nu (s) < N$. Then
\begin{equation}\label{eq-quasi}
\dim \Hull ( \CC_s ) \; = \; \Phi \left( \nu (s) \right) + \varepsilon ( G_s, A_s ),
\end{equation}
where $\varepsilon ( G_s, A_s )$ is the excess of \cref{thm-meet}. Moreover the following hold.
\begin{enumerate}
\item $\Phi$ is a piecewise quasi-polynomial in $\nu$ and is eventually linear. More precisely, $\Phi ( \nu + n/c ) = \Phi ( \nu ) + n$ whenever $c \, \nu \geq (n-1)d$, and $\Phi ( \nu ) = c \, \nu + 1 - g_\XX$ whenever $c \, \nu > 2 g_\XX - 2$.
\item $\Phi ( \nu (s) ) = c \, \nu (s) + 1 - g_\XX$ if and only if $s \in W ( \XX )$.
\item $\varepsilon ( G_s, A_s ) = 0$ for $s \in W ( \XX )$, and in general
\[
0 \; \leq \; \varepsilon ( G_s, A_s ) \; \leq \; \ell \left( K - G_s \wedge A_s \right) , \qquad \deg \left( K - G_s \wedge A_s \right) = t (s) .
\]
\end{enumerate}
\end{thm}

\begin{proof}
By \cref{prop-dual}, \cref{thm-meet} and \cref{lem-meetjoin},
\[
\dim \Hull ( \CC_s ) = \ell \left( \mu (s) D_\infty - E \right) - \ell \left( \mu (s) D_\infty - E - D \right) + \varepsilon ( G_s, A_s ) .
\]
The second term vanishes because $\deg ( \mu (s) D_\infty - E - D ) = c \nu (s) - N < 0$, and the first is $\Phi ( \nu (s) )$ by \cref{lem-shift}. This proves \cref{eq-quasi}.

For (1), the formula \cref{eq-Phi} is a finite sum of truncated floor functions, hence is piecewise quasi-polynomial. Replacing $\nu$ by $\nu + n/c$ replaces $c \nu - dj$ by $c \nu - dj + n$ and increases each positive summand by exactly one. Hence $\Phi ( \nu + n/c ) = \Phi ( \nu ) + n$ whenever every summand is positive, that is for $c \nu \geq (n-1) d$. The eventual linear formula is Riemann--Roch applied to $\nu D_\infty$.

For (2), $\Phi ( \nu (s) ) = \ell ( \nu (s) D_\infty )$ by \cref{eq-Phi}, and Riemann-Roch gives $\ell ( \nu (s) D_\infty ) = c \, \nu (s) + 1 - g_\XX$ if and only if $\nu (s) D_\infty$ is non-special. Now $\Ram \sim \tfrac{d}{c} D_\infty$ by \cref{eq-divisors}, and $D + E \sim \tfrac{qn}{c} D_\infty$ by the proof of \cref{prop-dual}, so \cref{eq-eta} gives
\[
K \; \sim \; \left( \frac{(n-1)d - n}{c} - 1 \right) D_\infty \; = \; \frac{2 g_\XX - 2}{c} \, D_\infty
\]
by \cref{eq-genus}. Hence $K - \nu (s) D_\infty \sim \left( \tfrac{2 g_\XX - 2}{c} - \nu (s) \right) D_\infty$ is effective when $c \, \nu (s) \leq 2 g_\XX - 2$ and of negative degree otherwise. The divisor $\nu (s) D_\infty$ is therefore non-special precisely when $c \, \nu (s) > 2 g_\XX - 2$, that is, when $t (s) < 0$, which is $s \in W ( \XX )$ by \cref{defn-window}.

For (3), the vanishing on the window is \cref{prop-window}, and the bound and the degree are \cref{eq-eps-h1,eq-tofs}.
\end{proof}

\begin{cor}\label{thm-hull}
Assume \textup{(H1)}--\textup{(H4)} and let $s \in W ( \XX )$. Then
\begin{equation}\label{eq-main}
\Hull ( \CC_s ) \; = \; C_L \left( D, \, \mu (s) D_\infty - E \right) \; \cong \; C_L \left( D, \, \nu (s) D_\infty \right),
\end{equation}
where the isomorphism is the monomial equivalence of \cref{lem-shift}, and
\begin{equation}\label{eq-main-dim}
\dim \Hull ( \CC_s ) \; = \; c \, \nu (s) + 1 - g_\XX \; = \; c \, \mu (s) - n \delta + 1 - g_\XX .
\end{equation}
In particular $\dim \Hull ( \CC_s ) \geq g_\XX \geq 1$, so no $\CC_s$ with $s \in W ( \XX )$ is \textup{LCD}. The function $s \mapsto \dim \Hull ( \CC_s )$ is affine of slope $c$ on $W ( \XX ) \cap [ 1, M/2 ]$ and of slope $-c$ on $W ( \XX ) \cap [ M/2, \infty )$, with maximum $c \lfloor M/2 \rfloor - n \delta + 1 - g_\XX$. On the window the hull depends on $\XX$ only through the integer $\delta$, that is, through the number of affine rational points.
\end{cor}

\begin{proof}
Let $s \in W ( \XX )$. Then $t (s) < 0$, so the meet $G_s \wedge A_s = \mu (s) D_\infty - E$ has degree $c \, \nu (s) > 2 g_\XX - 2$, and \cref{cor-window} yields $\Hull ( \CC_s ) = C_L ( D, \mu (s) D_\infty - E )$ together with $\dim \Hull ( \CC_s ) = c \, \nu (s) + 1 - g_\XX$. This is \cref{eq-main,eq-main-dim}, the isomorphism in \cref{eq-main} being the monomial equivalence of \cref{lem-shift}. From $c \, \nu (s) > 2 g_\XX - 2$ we get $\dim \Hull ( \CC_s ) > g_\XX - 1$, and $g_\XX \geq 1$ by \cref{lem-standing}. On $[ 1, M/2 ]$ we have $\mu (s) = s$ and on $[ M/2, \infty )$ we have $\mu (s) = M - s$, which gives the two slopes and the maximum at $\mu (s) = \lfloor M/2 \rfloor$. Finally $\dim \Hull ( \CC_s )$ depends on $\XX$ only through $\delta$, and $N = n ( q - \delta )$, so $\delta$ is determined by the number of affine rational points.
\end{proof}

\begin{thm}\label{cor-thresholds}
Assume \textup{(H1)}--\textup{(H4)}, let $I = \{ s \geq 1 : \Hull ( \CC_s ) \neq 0 \}$, and let
\[
s^{+} \; = \; M + \left\lfloor \frac{(n-1) d - n \delta}{c} \right\rfloor + 1 .
\]
Then the following hold.
\begin{enumerate}
\item $W ( \XX ) \subseteq I \subseteq [ \, 1, s^{+} - 1 \, ]$.
\item For $s \geq s^{+}$ one has $\CC_s^\perp = 0$, and $\CC_s$ is \textup{LCD}.
\item For $s \notin W ( \XX )$ with $c \, \mu (s) < n \delta$ one has $\dim \Hull ( \CC_s ) = \varepsilon ( G_s, A_s )$.
\end{enumerate}
\end{thm}

\begin{proof}
For (2), $\deg A_s = (n-1) d - n \delta + c ( M - s )$ by \cref{eq-degA}, and $s \geq s^{+}$ gives $c ( s - M ) > (n-1) d - n \delta$. Hence $\deg A_s < 0$, so $L ( A_s ) = 0$ and $\CC_s^\perp = C_L (D, A_s) = 0$ by \cref{prop-dual}; in particular $\Hull ( \CC_s ) = 0$. For (1), the containment $W ( \XX ) \subseteq I$ holds because $\dim \Hull ( \CC_s ) \geq g_\XX \geq 1$ on the window by \cref{thm-hull}, and $I \subseteq [ \, 1, s^{+} - 1 \, ]$ follows from (2). For (3), by \cref{eq-quasi} it suffices that $\Phi ( \nu (s) ) = 0$, and this holds because $c \, \mu (s) < n \delta$ gives $\nu (s) < 0$.
\end{proof}

%--------------------------------------------------------------------------------
%\subsection{Asymmetry}\label{ssec-asym}
The meet and the join of the pair $( G_s, A_s )$ depend on $s$ only through $\mu (s)$ and $\max \{ s, M-s \}$. Both are invariant under $s \mapsto M - s$. The Riemann-Roch side of \cref{eq-quasi} is therefore symmetric about $M/2$, and the excess is the only term that can fail to be symmetric.

\begin{prop}\label{prop-asym}
Assume \textup{(H1)}--\textup{(H4)} and let $1 \leq s \leq M-1$. Then
\[
G_s \wedge A_s = G_{M-s} \wedge A_{M-s} , \qquad G_s \vee A_s = G_{M-s} \vee A_{M-s} , \qquad t (s) = t (M-s) ,
\]
and
\begin{equation}\label{eq-asym}
\dim \Hull ( \CC_s ) - \dim \Hull ( \CC_{M-s} ) \; = \; \varepsilon ( G_s, A_s ) - \varepsilon ( G_{M-s}, A_{M-s} ) .
\end{equation}
In particular, if $\varepsilon ( G_s, A_s ) = 0$ for every $s$, then $\dim \Hull ( \CC_s ) = \dim \Hull ( \CC_{M-s} )$ for every $1 \leq s \leq M-1$.
\end{prop}

\begin{proof}
Both $\mu (s) = \min \{ s, M-s \}$ and $\max \{ s, M-s \}$ are invariant under $s \mapsto M-s$, so \cref{eq-meetjoin} gives the three equalities. Hence $\ell ( G_s \wedge A_s )$ and $\ell ( G_s \wedge A_s - D )$ take the same value at $s$ and at $M-s$. Since $\CC_s^\perp = C_L ( D, A_s )$ by \cref{eq-dual}, \cref{eq-meet-exact} gives $\dim \Hull ( \CC_s ) = \ell ( G_s \wedge A_s ) - \ell ( G_s \wedge A_s - D ) + \varepsilon ( G_s, A_s )$, and subtracting this identity at $M-s$ from the one at $s$ gives \cref{eq-asym}.
\end{proof}

On the totally split locus one has $G_s \leq A_s$ for every $s \leq \lfloor M/2 \rfloor$ by \cref{prop-dichotomy}, so the excess vanishes on the lower half of the range. Then \cref{eq-asym} computes the excess on the upper half from hull dimensions alone.

\begin{cor}\label{cor-asym}
Assume \textup{(H1)}--\textup{(H4)} and that $\XX$ is totally split. Then for every $s$ with $1 \leq s \leq \lfloor M/2 \rfloor$,
\[
\varepsilon ( G_s, A_s ) = 0 , \qquad \varepsilon ( G_{M-s}, A_{M-s} ) \; = \; \dim \Hull ( \CC_{M-s} ) - \dim \Hull ( \CC_s ) .
\]
\end{cor}

\begin{proof}
The vanishing is \cref{prop-dichotomy}, and the formula is \cref{eq-asym}.
\end{proof}

Combined with \cref{cor-detect}, \cref{cor-asym} is the mechanism by which the family detects the moduli point. The two degrees $s$ and $M - s$ present the same meet, the same join, and hence the same class $[ K - G_s \wedge A_s ] \in \Pic^{t (s)} ( \XX )$ to be tested. A difference between $\dim \Hull ( \CC_s )$ and $\dim \Hull ( \CC_{M-s} )$ says that the class is effective and that the two degrees see it differently. This is possible only because the excess depends on $L ( G_s )$ and $L ( A_s )$ separately, and not on the pair $( G_s \wedge A_s, G_s \vee A_s )$. When $t (s) = 0$ the class tested is a single element of $\Pic^0 ( \XX )$, and by \cref{cor-detect} a nonzero value of \cref{eq-asym} forces
\[
(n-1) \Ram + \max \{ s, \, M-s \} \, D_\infty \; \sim \; D .
\]
%

%----------------------------------------------------------------------
\section{The excess in closed form: strata with a transitive action}\label{sec-equiv}

This section evaluates both sides of \cref{eq-quasi} on a second locus, disjoint from the totally split locus of \cref{sec-hull}. Under the criterion of \cref{prop-transitive-criterion} the abelian group $H_m \cong \Z_n \times \Z_m$ acts regularly on $\YY$. Then $\F_q^N$ is the regular representation of $H_m$, each $\CC_s$ is an ideal of $\F_q [ H_m ]$ with defining set $T_s \subseteq \Z_m \times \Z_n$, and \cref{thm-transitive-hull} computes $\Hull ( \CC_s )$ from $T_s$. By \cref{prop-window-empty} the window is empty and the curve is not totally split, so \cref{thm-hull} does not apply on this locus. Combining \cref{thm-transitive-hull} with \cref{eq-quasi} evaluates the excess in closed form (\cref{cor-epsilon-transitive}).

%\subsection{The group acting on the family}\label{ssec-stabiliser}

\begin{defn}\label{def-mf}
For $f = \sum_{i=0}^{d} a_i x^i$ satisfying \textup{(H1)} define
\[
m (f) \; = \; \gcd \left( \, q-1, \; \{ \, i \; : \; a_i \neq 0 \, \} \, \right) .
\]
\end{defn}

By construction $m(f)$ divides $q-1$ and, since $a_d \neq 0$, also $d$. Every exponent in the support of $f$ is a multiple of $m(f)$, so $f (x) = g \left( x^{m(f)} \right)$ for a unique $g \in \F_q [u]$, of degree $d / m(f)$.

\begin{prop}\label{prop-stabilizer}
Assume \textup{(H1)}--\textup{(H4)}, let $m = m(f)$, let $\lambda \in \F_q^*$ have order $m$, and let
\[
H \; = \; \left\{ \, \sigma \in \Aut ( \XX ) \; : \; \sigma \text{ is defined over } \F_q, \; \sigma \text{ preserves } \pi, \; \sigma ( \YY ) = \YY, \; \sigma ( D_\infty ) = D_\infty, \; \sigma \bigl( \pi^{-1} (0) \bigr) = \pi^{-1} (0) \, \right\} .
\]
Then
\[
H \; = \; \langle \tau \rangle \times \langle \sigma_\lambda \rangle \; \cong \; \Z_n \times \Z_m, \qquad \sigma_\lambda (x,y) = ( \lambda x, \, y ) ,
\]
and $H \leq \PAut ( \CC_s )$ for every $s \geq 1$. In particular $H$ is abelian and $|H| = nm$ is prime to $p$.
\end{prop}

\begin{proof}
Let $\sigma \in H$. Since $\sigma$ preserves $\pi$, it descends to an automorphism $\overline{\sigma}$ of $\PP^1$ preserving the branch locus. The conditions on $D_\infty$ and on $\pi^{-1} (0)$ say that $\overline{\sigma}$ fixes $\infty$ and $0$, respectively. Hence, being defined over $\F_q$, it has the form $x \mapsto \alpha x$ with $\alpha \in \F_q^*$. In particular $\overline{\sigma}$ permutes the roots of $f$, so $f ( \alpha x )$ and $f(x)$ have the same simple roots and $f ( \alpha x ) = \alpha^d f(x)$. Because $\langle \tau \rangle$ is normal in $\Aut ( \XX, \pi )$, one has $\sigma \tau \sigma^{-1} = \tau^k$ for some $k$ prime to $n$, so $\sigma^* y$ lies in the eigenspace of $\tau^*$ on $\overline{\F}_q ( \XX )$ on which $\tau^*$ acts by $\zeta_n^k$. That eigenspace is $\overline{\F}_q (x) \, y^k$, so $\sigma^* y = r(x) \, y^k$ with $r \in \overline{\F}_q (x)^*$, and applying $\sigma^*$ to \cref{eq-superell} gives $r(x)^n f(x)^k = f ( \alpha x )$. At a root $a$ of $f$ both $f(x)$ and $f ( \alpha x )$ vanish to order one, so comparing valuations at $x = a$ gives $n \, v_a (r) + k = 1$, and $k \equiv 1 \pmod n$. Replacing $r(x)$ by $r(x) f(x)^{(k-1)/n}$ we may take $k = 1$. Then $r(x)^n = f ( \alpha x ) / f(x) = \alpha^d$ is constant, so $r$ is a constant $\varepsilon \in \overline{\F}_q^{\, *}$ and $\sigma (x,y) = ( \alpha x, \varepsilon y )$. Substituting into \cref{eq-superell} gives $a_i \alpha^i = \varepsilon^n a_i$ for every $i$. Taking $i = d$ gives $\varepsilon^n = \alpha^d$. Taking $i = 0$, legitimate because $a_0 = f(0) \neq 0$ by \textup{(H1)}, gives $\varepsilon^n = 1$. Hence $\alpha^i = 1$ for every $i$ with $a_i \neq 0$, and $\alpha^{q-1} = 1$, so $\alpha \in \mu_m$. Moreover $\varepsilon \in \mu_n \subseteq \F_q^*$ by \textup{(H4)}, so $\sigma = \sigma_\alpha \tau^k$ for some $k$.

Conversely $\sigma_\lambda$ lies in $H$. It is defined over $\F_q$, it preserves $\pi$, $D_\infty$ and the fibre over $0$, and $f ( \lambda x ) = f(x)$ because $m$ divides every $i$ in the support of $f$; hence it maps the affine $\F_q$-rational point $(a,b)$ to $( \lambda a, b )$ and stabilises $\YY$. Also $\tau \in H$ by \textup{(H4)}. The two commute, and $\langle \tau \rangle \cap \langle \sigma_\lambda \rangle = 1$ because $\sigma_\lambda^j$ acts trivially on $y$ while $\tau^k$ acts trivially on $x$, so the product is direct. Finally $n \mid q-1$ and $m \mid q-1$ give $p \nmid nm$. Each $\sigma \in H$ fixes $D_\infty$ as a divisor, so $L ( s D_\infty )$ is stable under $h \mapsto h \circ \sigma^{-1}$, which under $\ev_s$ is the permutation of coordinates induced by $\sigma$ on $\YY$. Hence $H \leq \PAut ( \CC_s )$.
\end{proof}

\begin{rem}
The condition that the descended action fix $0$ is essential in \cref{prop-stabilizer}. A subgroup of $\mathrm{AGL}_1 ( \F_q )$ of order prime to $p$ need not consist of scalings about the origin; after translating its common affine fixed point to $0$, one obtains the normalized situation above. Thus \cref{prop-stabilizer} identifies the fibration-preserving subgroup fixing both $0$ and $\infty$, and makes no assertion that this is the full permutation automorphism group of $\CC_s$.
\end{rem}

%Automorphism groups of algebraic geometry codes on $C_{a,b}$ curves, which by \cref{sec-prelim} are the superelliptic curves with $c = 1$, are computed in \cite{sh-wa}. The group $H$ above is the subgroup of the geometric automorphism group that survives over $\F_q$ and fixes the point at infinity.

%\subsection{Strata with a regular action}\label{ssec-transitive}

\begin{prop}\label{prop-transitive-criterion}
Assume \textup{(H1)}--\textup{(H4)}, let $m \mid m(f)$, let $\lambda \in \F_q^*$ have order $m$, write $f (x) = g ( x^m )$ with $\deg g = d/m$, and assume that \(g (0) \; \notin \; ( \F_q^* )^n\). 
%
%\begin{equation}\label{prop-transitive-criterion}
%
%\end{equation}
%
Then $\sigma_\lambda : (x,y) \mapsto ( \lambda x, y )$ is an automorphism of $\XX$ fixing $D_\infty$, the group $H_m = \langle \tau, \sigma_\lambda \rangle \cong \Z_n \times \Z_m$ acts freely on $\YY$, and the following are equivalent:
\begin{enumerate}
\item $H_m$ acts transitively, hence regularly, on $\YY$;
\item $N = nm$;
\item exactly one of the $(q-1)/m$ values $g ( w^m )$, as $w$ runs through a set of representatives of $\F_q^* / \langle \lambda \rangle$, lies in $( \F_q^* )^n$.
\end{enumerate}
\end{prop}

\begin{proof}
Since $\lambda^m = 1$ we have $f ( \lambda x ) = g ( \lambda^m x^m ) = f(x)$, so $\sigma_\lambda$ is an automorphism fixing $x = \infty$ and hence $D_\infty$, and it commutes with $\tau$. Let $S = \F_q \setminus \Delta = \{ a \in \F_q : f(a) \in ( \F_q^* )^n \}$, so that $N = n | S |$ by \cref{lem-pullback}. By the above we have $0 \notin S$, so every point of $\YY$ has $x \neq 0$. If $\tau^i \sigma_\lambda^{\, j}$ fixes $(a,b) \in \YY$ then $\lambda^j a = a$ with $a \neq 0$ forces $m \mid j$, and then $\zeta_n^{\, i} b = b$ with $b \neq 0$ forces $n \mid i$. Hence the action is free, $|H_m| = nm$ divides $N$, and the action is transitive if and only if $N = nm$. This is the equivalence of (1) and (2). Finally $f$ is constant on each coset of $\langle \lambda \rangle$ in $\F_q^*$, because $x^m$ is, so $S$ is a union of such cosets. Since $0 \notin S$, condition (2) says precisely that $S$ is a single coset, which is (3).
\end{proof}

The hypothesis $m \mid m(f)$ is geometric: it says that the branch locus of $\pi$ is invariant under $x \mapsto \lambda x$, that is, that the reduced fibration-preserving group $\overline{\Aut} ( \XX, \pi )$ contains a cyclic group of order $m$ fixing $0$ and $\infty$, which places $\XX$ among the families with prescribed automorphism group of \cite{sa-sh}. The conditions of \cref{prop-transitive-criterion} are arithmetic and depend only on $q$ and $g$. The extreme case $m = d$ constrains the curve itself and is the case in which the maximal hull dimension admits a closed form.

\begin{rem}\label{rem-binomial}
Suppose $m = d$ in \cref{prop-transitive-criterion}. Then $\deg g = 1$, so $f$ is a binomial, $\XX : y^n = a x^d + b$ with $a, b \in \F_q^*$, and $d \mid q-1$ because $m \mid q-1$. Over $\overline{\F}_q$ the substitution $x = u x'$, $y = v y'$ with $a u^d = b$ and $v^n = b$ carries the equation to $y'^{\, n} = x'^{\, d} + 1$. Hence the locus $m = d$ is a single point of the moduli space for each type $(n,d)$, and its $\F_q$-members are the forms of $y^n = x^d + 1$. For $n = d$ one has $c = n$ and $\bw = (1,1,1)$, so the weighted model is a smooth plane curve of degree $n$ and $\XX$ is a form of the Fermat curve of degree $n$.
\end{rem}

\begin{prop}\label{prop-window-empty}
Under the hypotheses of \cref{prop-transitive-criterion}, if $H_m$ acts transitively on $\YY$ then $W ( \XX ) = \emptyset$. In particular $\XX$ is not totally split.
\end{prop}

\begin{proof}
By \cref{eq-degrees,eq-tofs} the condition $t (s) < 0$ reads $c \max \{ s, M-s \} + (n-1) d < N$, and $\max \{ s, M-s \}$ is minimised at $s = \lceil M/2 \rceil$. Hence the window is nonempty if and only if
\[
c \left\lceil M/2 \right\rceil + (n-1) d \; < \; N = nm .
\]
Assume this inequality. Since $c \lceil M/2 \rceil \geq c M / 2 = \left( (q-1) n - c \right) / 2$ and $d = m \deg g \geq m$, we get $\left( (q-1) n - c \right) / 2 + (n-1) m < nm$, that is $(q-1) n - c < 2m \leq 2 (q-1)$, whence
\[
(q-1) (n-2) \; < \; c \; \leq \; n .
\]
If $n \geq 3$ then $n \mid q-1$ gives $n \leq q-1$, so $(q-1)(n-2) \geq n (n-2) \geq n$, a contradiction. Hence $n = 2$, and $q$ is odd by (H4). Then $2 (q-1) - c < 2m$ with $c \leq 2$ gives $m > q - 2$, so $m = q-1$ and $d \geq m = q-1$. If $c = 2$ then $M = q-2$ is odd and $c \lceil M/2 \rceil = q-1$. If $c = 1$ then $M = 2(q-1) - 1$ and $c \lceil M/2 \rceil = q-1$. In both cases the displayed inequality reads $(q-1) + d < 2 (q-1)$, that is $d < q-1$, contradicting $d \geq q-1$. Hence $W ( \XX ) = \emptyset$. Finally $N = nm \leq n (q-1) < nq$, so $E \neq 0$.
\end{proof}

\begin{rem}\label{rem-hermitian}
By \cref{rem-binomial} the Fermat and Hermitian curves lie in the locus $m(f) = d$, but they satisfy neither \textup{(H1)} nor \cref{prop-transitive-criterion}. Let $q = q_0^2$, let $n = d = q_0 + 1$, and write the Hermitian curve as $y^n = f(x)$ with $f (x) = - \left( x^{q_0+1} + 1 \right)$, so that $m (f) = \gcd \left( q_0^2 - 1, \, q_0 + 1 \right) = d$ and $g_\XX = q_0 ( q_0 - 1 ) / 2$ by \cref{eq-genus}. The map $x \mapsto x^{q_0+1}$ is the norm $\F_{q_0^2}^* \to \F_{q_0}^*$ and is surjective. Hence the equation $x^{q_0+1} = -1$ has $d$ solutions in $\F_q$, so every root of $f$ is $\F_q$-rational and \textup{(H1)} fails. For the same reason $( \F_q^* )^n = \F_{q_0}^*$ contains $f (0) = -1$, so \cref{prop-transitive-criterion} fails as well. \cref{prop-transitive-criterion} selects the forms of $y^n = x^d + 1$ on which $f$ takes $n$-th power values on a single coset of $\langle \lambda \rangle$, hence with the fewest rational points the shape permits, whereas a maximal curve takes such values as often as possible. By \cref{prop-window-empty}, a transitive stratum is never totally split.
\end{rem}

%\subsection{The hull as a defining set}\label{ssec-defining}

\begin{lem}\label{lem-defining-set}
Assume the criterion of \cref{prop-transitive-criterion}. Then $\F_q^N$ is the regular representation of $H_m$, the code $\CC_s$ is an ideal of $\F_q [H_m]$ with defining set
\[
T_s \; = \; \left\{ \, ( i \bmod m, \; j \bmod n ) \; : \; x^i y^j \in B_s \, \right\} \; \subseteq \; \Z_m \times \Z_n ,
\]
and $\dim \CC_s = | T_s |$.
\end{lem}

\begin{proof}
The action of $H_m$ on $\YY$ is regular by \cref{prop-transitive-criterion}, so $\F_q^N$ is the regular representation and each of the $N = nm$ characters occurs in it with multiplicity one. The isotypic components are therefore lines. The monomial $x^i y^j$ has character $( i \bmod m, j \bmod n )$, so $\ev_s ( x^i y^j )$ lies in the corresponding line. Two monomials of $B_s$ with the same character have the same exponent of $y$, since $0 \leq j \leq n-1$, and exponents of $x$ differing by a multiple of $m$. As $\YY$ lies over a single coset of $\langle \lambda \rangle$, the function $x^m$ is constant on $\YY$, so their images are proportional. No image is zero, because $x$ and $y$ vanish nowhere on $\YY$ by \cref{prop-transitive-criterion} and (H1). Hence $\CC_s$ is the sum of the lines indexed by $T_s$.
\end{proof}

\begin{thm}\label{thm-transitive-hull}
Assume the criterion of \cref{prop-transitive-criterion} and write $-T_s = \{ ( -a, -b ) : (a,b) \in T_s \}$. Then
\begin{equation}\label{eq-transitive-hull}
\Hull ( \CC_s ) \; = \; \bigoplus_{\chi \, \in \, T_s \setminus ( - T_s )} \left( \F_q^N \right)_\chi , \qquad \dim \Hull ( \CC_s ) \; = \; | T_s | - | T_s \cap ( - T_s ) | .
\end{equation}
In particular $\CC_s$ is \textup{LCD} if and only if $T_s = - T_s$, and self-orthogonal if and only if $T_s \cap ( - T_s ) = \emptyset$.
\end{thm}

\begin{proof}
For characters $\chi, \psi$ of $H_m$ one has $\sum_{h \in H_m} \chi (h) \psi (h) = 0$ unless $\psi = \chi^{-1}$, in which case the sum equals $| H_m | = N$, which is nonzero in $\F_q$ because $nm \mid ( q-1 )^2$. Hence the standard bilinear form pairs the $\chi$-component with the $\chi^{-1}$-component and is zero on all other pairs, so $\CC_s^\perp$ is the sum of the lines indexed by the complement of $- T_s$. Intersecting with $\CC_s$ gives \cref{eq-transitive-hull}.
\end{proof}

\cref{eq-transitive-hull} is the exact analogue, for the abelian group $\Z_n \times \Z_m$, of the classical criterion for cyclic codes. If $C$ is the cyclic code of length $N$ with defining set $T \subseteq \Z_N$, then $\Hull (C) = 0$ if and only if $- T \subseteq T$. Skersys \cite{sk} proved that the average dimension of the hull of cyclic codes of length $N$ vanishes precisely when $N$ divides an integer of the form $q^i + 1$, and grows linearly in $N$ otherwise. That exceptional condition is a self-reciprocity condition: every $q$-cyclotomic coset modulo $N$ is stable under negation, equivalently every irreducible factor of $x^N - 1$ over $\F_q$ is self-reciprocal. It is the criterion studied in \cite{sh-si}. \cref{thm-transitive-hull} places the codes $\CC_s$ on the transitive strata in that setting: the hull vanishes precisely on the symmetric defining sets and is large precisely when $T_s$ is far from symmetric.

\cref{thm-transitive-hull} determines the hull for every $s$ by a finite computation with no geometry, and by \cref{prop-window-empty} this is exactly the range on which \cref{thm-hull} is silent. Comparing the two evaluates the excess.

\begin{cor}\label{cor-epsilon-transitive}
Assume the criterion of \cref{prop-transitive-criterion}. Then $\delta = q - m$ and, for every $s \geq 1$ with $c \, \nu (s) < N$,
\[
\varepsilon ( G_s, A_s ) \; = \; \left| T_s \right| - \left| T_s \cap ( - T_s ) \right| - \Phi \left( \nu (s) \right) ,
\]
where $\Phi$ and $\nu$ are as in \cref{eq-Phi,def-nu}.
\end{cor}

\begin{proof}
By \cref{lem-pullback}, $N = n ( q - \delta )$, and the criterion gives $N = nm$; hence $\delta = q - m$. The formula is \cref{eq-quasi} combined with \cref{eq-transitive-hull}.
\end{proof}

%\subsection{A bound and its sharpness}\label{ssec-torsion}

The $2$-torsion of $H_m$ constrains the symmetry of $T_s$. This bounds the hull, and hence the excess by \cref{cor-epsilon-transitive}.

\begin{prop}\label{prop-torsion-bound}
Assume the criterion of \cref{prop-transitive-criterion} and let
\[
t_2 \; = \; \left| \left( \Z_m \times \Z_n \right) [2] \right| \; = \; \gcd (2, m) \cdot \gcd (2, n)
\]
be the number of characters of $H_m$ of order at most two. Then $\dim \Hull ( \CC_s ) \leq ( nm - t_2 ) / 2$ for every $s \geq 1$.
\end{prop}

\begin{proof}
Let $U = T_s \setminus ( - T_s )$, so that $\dim \Hull ( \CC_s ) = |U|$ by \cref{eq-transitive-hull}. If $\chi \in U$ then $\chi \notin - T_s$, hence $- \chi \notin T_s$ and a fortiori $- \chi \notin U$. In particular $\chi \neq - \chi$, so $U$ contains no character of order at most two, and $U$ meets each of the $(nm - t_2)/2$ pairs $\{ \chi, - \chi \}$ with $\chi \neq - \chi$ in at most one element.
\end{proof}

\begin{exa}\label{exa-torsion}
Let $q = 11$, $n = 2$, $m = 2$, $d = 4$ and $f (x) = x^4 - x^2 + 7 = g ( x^2 )$ with $g (u) = u^2 - u + 7$. The discriminant of $g$ is $6$, a non-square modulo $11$, so $g$ is irreducible and $f$ has no root in $\F_{11}$; $f$ is separable and $m(f) = 2$. Moreover $g(0) = 7$ is not a square, and among the values $g(1) = 7$, $g(3) = 2$, $g(4) = 8$, $g(5) = 5$, $g(9) = 2$ taken on the squares of representatives of $\F_{11}^* / \{ \pm 1 \}$ exactly one, namely $g(5) = 5$, is a square. Hence \cref{prop-transitive-criterion} applies, $S = \{ 4, 7 \}$ and $N = 4 = nm$. But every character of $\Z_2 \times \Z_2$ satisfies $\chi = - \chi$, so $T_s = - T_s$ for every $s$ and $\Hull ( \CC_s ) = 0$ for every $s$, while $g_\XX = 1$. Thus $\max_s \dim \Hull ( \CC_s ) = 0 < g_\XX$, and \cref{prop-torsion-bound} is sharp here, giving $(nm - t_2)/2 = 0$.
\end{exa}

\cref{tab-transitive} lists strata satisfying \cref{prop-transitive-criterion}, with $f$ chosen as the first polynomial of the prescribed shape meeting condition (3), together with the two bounds and the largest hull dimension attained. In every case the computed hull dimensions agree with \cref{eq-transitive-hull} for all $s$, and in every line but the first the hull is nonzero over an entire interval of degrees.

\begin{table}[ht]
\caption{Strata with a regular action of $H_m = \Z_n \times \Z_m$, and their maximal hull dimensions.}\label{tab-transitive}
\begin{tabular}{rrrrlrrrrr}
\toprule
$q$ & $n$ & $m$ & $d$ & $f$ & $N = |H_m|$ & $t_2$ & $\frac{nm-t_2}{2}$ & $g_\XX$ & $h_{\max}$ \\
\midrule
$11$ & $2$ & $2$  & $4$  & $x^4 - x^2 + 7$ & $4$  & $4$ & $0$  & $1$ & $0$ \\
$7$  & $2$ & $6$  & $6$  & $x^6 + 3$    & $12$ & $4$ & $4$  & $2$ & $2$ \\
$7$  & $3$ & $6$  & $6$  & $x^6 + 5$    & $18$ & $2$ & $8$  & $4$ & $4$ \\
$11$ & $2$ & $5$  & $5$  & $x^5 + 6$    & $10$ & $2$ & $4$  & $2$ & $2$ \\
$11$ & $2$ & $10$ & $10$ & $x^{10} + 2$ & $20$ & $4$ & $8$  & $4$ & $4$ \\
$11$ & $5$ & $5$  & $5$  & $x^5 + 2$    & $25$ & $1$ & $12$ & $6$ & $6$ \\
$13$ & $2$ & $6$  & $6$  & $x^6 + 5$    & $12$ & $4$ & $4$  & $2$ & $2$ \\
$13$ & $3$ & $6$  & $6$  & $x^6 + 2$    & $18$ & $2$ & $8$  & $4$ & $4$ \\
$17$ & $2$ & $8$  & $8$  & $x^8 + 5$    & $16$ & $4$ & $6$  & $3$ & $3$ \\
$19$ & $2$ & $18$ & $18$ & $x^{18} + 3$ & $36$ & $4$ & $16$ & $8$ & $8$ \\
\bottomrule
\end{tabular}
\end{table}

In every line the last column equals the minimum of the two preceding ones, and it is tempting to record that as a theorem. It is not one. Nine of the ten lines have $m = d$, so by \cref{rem-binomial} nine of the ten curves are binomial, forms of $y^n = x^d + 1$ carrying no parameter and realising only seven distinct types $(n,d)$. The first line, the only one on which $f$ carries a parameter, has $m = 2 < 4 = d$, and there both expressions vanish. On a line with $m = d$ the minimum is the genus: the inequality $g_\XX \leq (nd - t_2)/2$ is equivalent to $t_2 \leq n + d + c - 2$, which holds because $t_2 \leq 4 \leq n + d + c - 2$ for all $n \geq 2$ and $d \geq 3$. The table therefore tests a single law, $h_{\max} = g_\XX$ on the locus $m = d$, and carries no information about the regime $m < d$. That law is a theorem with exactly one exceptional pair $(n,d)$, and it is \cref{thm-hmax}. The minimum itself fails in both directions, by \cref{exa-cex-above,exa-cex-below}. What replaces it is not an expression in $g_\XX$ and $t_2$ but a covering number.

%\subsection{The hull as a covering number}\label{ssec-covering}

\cref{thm-transitive-hull} determines the hull at each $s$ and says nothing about how the answer moves with $s$, whereas a maximum over $s$ is precisely a statement about that motion. The defining set is the sublevel set of a single function on the character group, and the two conditions $\chi \in T_s$ and $- \chi \notin T_s$ then become the two endpoints of an interval of degrees. The hull dimension at $s$ is the number of those intervals containing $s$, and the maximal hull dimension is the covering number of an explicit family of $nm$ intervals. Both depend on the stratum only through $(n,d,m)$.

\begin{defn}\label{def-weight}
Identify $\widehat{H}_m$ with $\Z_m \times \Z_n$ as in \cref{lem-defining-set}. The \emph{weight} of a character $\chi = ( \alpha, \beta )$, written with the least nonnegative representatives $0 \leq \alpha \leq m-1$ and $0 \leq \beta \leq n-1$, is the weighted degree
\[
\wt ( \chi ) \; = \; \wt \left( x^\alpha y^\beta \right) \; = \; \frac{n}{c} \, \alpha + \frac{d}{c} \, \beta .
\]
\end{defn}

\begin{lem}\label{lem-sublevel}
Assume the criterion of \cref{prop-transitive-criterion}. Then, for every $s \geq 0$,
\[
T_s \; = \; \left\{ \, \chi \in \widehat{H}_m \; : \; \wt ( \chi ) \leq s \, \right\} .
\]
\end{lem}

\begin{proof}
By \cref{lem-bridge} the monomial $x^i y^j$ with $i \geq 0$ and $0 \leq j \leq n-1$ lies in $B_s$ if and only if $\frac{n}{c} i + \frac{d}{c} j \leq s$, and its character is $( i \bmod m, j \bmod n )$ by \cref{lem-defining-set}. Fix $\chi = ( \alpha, \beta )$ with $0 \leq \alpha \leq m-1$ and $0 \leq \beta \leq n-1$. The monomials of that character with $0 \leq j \leq n-1$ are those with $j = \beta$ and $i \equiv \alpha \pmod m$, and among them $x^\alpha y^\beta$ has the smallest weighted degree, both weights $\frac{n}{c}$ and $\frac{d}{c}$ being positive. Hence $\chi \in T_s$ if and only if $x^\alpha y^\beta \in B_s$, which is $\wt ( \chi ) \leq s$.
\end{proof}

\begin{thm}\label{thm-covering}
Assume the criterion of \cref{prop-transitive-criterion} and, for $\chi \in \widehat{H}_m$, define
\[
I_\chi \; = \; \left\{ \, k \in \Z \; : \; \wt ( \chi ) \leq k < \wt ( - \chi ) \, \right\} .
\]
Then, for every $s \geq 0$,
\begin{equation}\label{eq-covering}
\Hull ( \CC_s ) \; = \; \bigoplus_{\chi \, : \, s \, \in \, I_\chi} \left( \F_q^N \right)_\chi , \qquad \dim \Hull ( \CC_s ) \; = \; \# \left\{ \, \chi \in \widehat{H}_m \; : \; s \in I_\chi \, \right\} .
\end{equation}
\end{thm}

\begin{proof}
By \cref{thm-transitive-hull} the hull is the sum of the isotypic lines indexed by $T_s \setminus ( - T_s )$. By \cref{lem-sublevel} the condition $\chi \in T_s$ reads $\wt ( \chi ) \leq s$, and the condition $\chi \notin - T_s$, that is $- \chi \notin T_s$, reads $\wt ( - \chi ) > s$.
\end{proof}

The intervals are governed by one reflection identity, the analogue on the character group of the identity $G \vee A - D = K - G \wedge A$ of \cref{thm-canonical}. Passing from $\chi$ to $- \chi$ replaces the weight by its complement, and the complement is taken with respect to a modulus that records only which of the two coordinates of $\chi$ vanish.

\begin{lem}\label{lem-reflection}
Define
\[
e_m \; = \; \frac{nm}{c} , \qquad e \; = \; \frac{nd}{c} , \qquad \Omega \; = \; e_m + e .
\]
Then $e_m \leq e$, with equality if and only if $m = d$, and for $\chi = ( \alpha, \beta ) \in \widehat{H}_m$ written with the least nonnegative representatives,
\[
\wt ( \chi ) + \wt ( - \chi ) \; = \;
\begin{cases}
0, & \alpha = 0, \; \beta = 0, \\
e_m, & \alpha \neq 0, \; \beta = 0, \\
e, & \alpha = 0, \; \beta \neq 0, \\
\Omega, & \alpha \neq 0, \; \beta \neq 0 .
\end{cases}
\]
\end{lem}

\begin{proof}
The least nonnegative representative of $- \alpha$ in $\Z_m$ is $m - \alpha$ if $\alpha \neq 0$ and $0$ if $\alpha = 0$, and likewise for $\beta$ in $\Z_n$; now add, using $\frac{n}{c} m = e_m$ and $\frac{d}{c} n = e$. The inequality $e_m \leq e$ is $m \leq d$, which holds because $m \mid m(f) \mid d$.
\end{proof}

Thus the $nm$ intervals fall into three concentric families, according to the type of $\chi$. Within each family the two endpoints of $I_\chi$ are reflections of one another in a fixed centre, so the family is totally ordered by inclusion, and the three centres $\frac{e_m - 1}{2}$, $\frac{e-1}{2}$ and $\frac{\Omega - 1}{2}$ occur in that order. Counting the three families separately gives the hull.

\begin{prop}\label{prop-abc}
Assume the criterion of \cref{prop-transitive-criterion}. For $s \geq 0$ define
\[
\begin{split}
h_1 (s) &= \# \left\{ \, \alpha \; : \; 1 \leq \alpha \leq m-1, \; \tfrac{n}{c} \alpha \leq \min \{ s, \, e_m - 1 - s \} \, \right\} , \\
h_2 (s) &= \# \left\{ \, \beta \; : \; 1 \leq \beta \leq n-1, \; \tfrac{d}{c} \beta \leq \min \{ s, \, e - 1 - s \} \, \right\} , \\
h_3 (s) &= \# \left\{ \, ( \alpha, \beta ) \in [1, m-1] \times [1, n-1] \; : \; \tfrac{n}{c} \alpha + \tfrac{d}{c} \beta \leq \min \{ s, \, \Omega - 1 - s \} \, \right\} .
\end{split}
\]
Then $\dim \Hull ( \CC_s ) = h_1 (s) + h_2 (s) + h_3 (s)$, and the first two counts are given in closed form by
\[
h_1 (s) = \max \left\{ 0, \; \left\lfloor \frac{c \min \{ s, e_m - 1 - s \}}{n} \right\rfloor \right\} , \qquad h_2 (s) = \max \left\{ 0, \; \left\lfloor \frac{c \min \{ s, e - 1 - s \}}{d} \right\rfloor \right\} .
\]
\end{prop}

\begin{proof}
Partition $\widehat{H}_m$ according to which of $\alpha$ and $\beta$ vanish, and apply \cref{thm-covering} with \cref{lem-reflection}. The trivial character has $I_\chi = \emptyset$ and contributes nothing. If $\alpha \neq 0$ and $\beta = 0$ then $\wt ( - \chi ) = e_m - \wt ( \chi )$, so $s \in I_\chi$ reads $\wt ( \chi ) \leq s$ together with $\wt ( \chi ) \leq e_m - 1 - s$, which is the condition defining $h_1 (s)$; here the constraint $\alpha \leq m-1$ is automatic, since $\frac{n}{c} \alpha \leq e_m - 1$ forces $\alpha < m$. The cases $\alpha = 0$, $\beta \neq 0$ and $\alpha, \beta \neq 0$ are identical, with $e_m$ replaced by $e$ and by $\Omega$, the constraint $\beta \leq n-1$ being automatic in the second case. The two displayed formulas count the multiples of $\frac{n}{c}$, respectively of $\frac{d}{c}$, in an initial segment of the positive integers.
\end{proof}

\begin{cor}\label{cor-hmax-computable}
Assume the criterion of \cref{prop-transitive-criterion} and let $h_{\max} = \max_{s \geq 1} \dim \Hull ( \CC_s )$. Then $\dim \Hull ( \CC_s )$, and hence $h_{\max}$, depends only on $(n,d,m)$ and on $s$, and not on $q$, on $f$ or on $\YY$. Moreover, if $h_{\max} > 0$ then the maximum is attained at $s = \wt ( \chi )$ for some $\chi \in \widehat{H}_m$, so that
\[
h_{\max} \; = \; \max_{\chi \, \in \, \widehat{H}_m} \; \dim \Hull \left( \CC_{\wt ( \chi )} \right) .
\]
The maximum runs over at most $nm$ degrees, all of them in the range $1 \leq s \leq \frac{n}{c} ( m-1 ) + \frac{d}{c} ( n-1 )$.
\end{cor}

\begin{proof}
The three counts of \cref{prop-abc} involve only $n$, $d$, $m$, $c$ and $s$, and $c = \gcd (n,d)$. By \cref{eq-covering} the function $s \mapsto \dim \Hull ( \CC_s )$ is a sum of indicator functions of intervals of integers, and such a function attains its maximum at the left endpoint of one of the intervals that are nonempty, that is at some $\wt ( \chi )$; the largest weight occurring is $\frac{n}{c} ( m-1 ) + \frac{d}{c} ( n-1 )$, and the trivial character contributes an empty interval.
\end{proof}

That $h_{\max}$ is a function of $(n,d,m)$ alone already accounts for a feature of \cref{tab-transitive}: the lines with $(n,m,d) = (2,6,6)$ over $\F_7$ and over $\F_{13}$ agree in the last column, as do the lines with $(n,m,d) = (3,6,6)$, although the curves and the fields differ.

%\subsection{Two counterexamples}\label{ssec-counterexamples}

The next two examples lie on transitive strata of the same type $n = m = 3$ over the same field. The maximal hull dimension exceeds $\min \{ g_\XX, (nm-t_2)/2 \}$ in the first and falls short of it in the second, so no such minimum can be the answer.

\begin{exa}\label{exa-cex-above}
Let $q = 7$, $n = 3$ and $f (x) = x^3 + 2$, so that $d = 3$, $c = 3$, $m(f) = 3$, $g_\XX = 1$ and $M = 5$. Take $m = 3$, $\lambda = 2$ of order $3$ in $\F_7^*$, and $g (u) = u + 2$. The cubes in $\F_7^*$ are $\{ 1, 6 \}$, so $-2 = 5$ is not a cube, $f$ is separable with no root in $\F_7$, and \textup{(H1)} holds; $n \mid q-1$ gives \textup{(H4)}. The value $g(0) = 2$ is not a cube, and of the two values $g (1) = 3$ and $g ( 3^3 ) = g (6) = 1$ taken on the representatives $1, 3$ of $\F_7^* / \langle \lambda \rangle$ exactly one is a cube. Hence \cref{prop-transitive-criterion} applies, $S = \{ 3, 5, 6 \}$ and $N = 9 = nm$. Here $t_2 = 1$, so
\[
\min \left\{ \, g_\XX , \; \frac{nm-t_2}{2} \, \right\} \; = \; \min \{ 1, 4 \} \; = \; 1 .
\]
But $\frac{n}{c} = \frac{d}{c} = 1$, so $B_1 = \{ 1, x, y \}$ and
\[
T_1 = \{ (0,0), (1,0), (0,1) \} , \qquad - T_1 = \{ (0,0), (2,0), (0,2) \} , \qquad T_1 \cap ( - T_1 ) = \{ (0,0) \} ,
\]
whence $\dim \Hull ( \CC_1 ) = 3 - 1 = 2$ by \cref{eq-transitive-hull}. The profiles are $\dim \CC_s = 3, 6, 8, 9$ and $\dim \Hull ( \CC_s ) = 2, 1, 1, 0$ for $s = 1, 2, 3, 4$, and $\Hull ( \CC_s ) = 0$ for $s \geq 4$. Thus $h_{\max} = 2 > 1$, attained at $s = 1 \leq M$.
\end{exa}

\begin{exa}\label{exa-cex-below}
Let $q = 7$, $n = 3$ and $f (x) = x^6 + 5 x^3 + 2 = g ( x^3 )$ with $g (u) = u^2 + 5u + 2$, so that $d = 6$, $c = 3$, $m(f) = 3$, $g_\XX = 4$ and $M = 5$. Again $f$ is separable with no root in $\F_7$, the value $g (0) = 2$ is not a cube, and of $g (1) = 1$ and $g (6) = 5$ exactly one is, so \cref{prop-transitive-criterion} applies with $m = 3$, $S = \{ 1, 2, 4 \}$ and $N = 9 = nm$. Here $t_2 = 1$, so
\[
\min \left\{ \, g_\XX , \; \frac{nm-t_2}{2} \, \right\} \; = \; \min \{ 4, 4 \} \; = \; 4 ,
\]
while $\frac{n}{c} = 1$, $\frac{d}{c} = 2$ and the profiles are $\dim \CC_s = 2, 4, 5, 7, 8, 9$ and $\dim \Hull ( \CC_s ) = 1, 1, 2, 2, 1, 0$ for $s = 1, \dots, 6$. Thus $h_{\max} = 2 < 4$, attained at $s = 3$ and $s = 4$, both at most $M$.
\end{exa}

The two examples show that $h_{\max}$ is neither $g_\XX$ nor $(nm-t_2)/2$ nor their minimum. In particular the genus is not an upper bound for the hull dimension on a transitive stratum: in \cref{exa-cex-above} it is exceeded. The bound of \cref{prop-torsion-bound} is unaffected, and both examples respect it. What is true is the case $m = d$, which is where the evidence of \cref{tab-transitive} lay, and \cref{exa-cex-above} is exactly its unique exception.

%\subsection{The maximal hull dimension when \texorpdfstring{$m = d$}{m = d}}\label{ssec-hmax}

The genus enters through a lattice count, in the same triangle that governs \cref{lem-bridge}.

\begin{lem}\label{lem-genus-lattice}
Let $n \geq 2$ and $d \geq 3$, let $c = \gcd (n,d)$ and $e = \frac{nd}{c}$. Then
\[
g_\XX \; = \; \# \left\{ \, ( \alpha, \beta ) \in \Z_{\geq 1}^2 \; : \; \tfrac{n}{c} \alpha + \tfrac{d}{c} \beta \leq e-1 \, \right\} .
\]
\end{lem}

\begin{proof}
Multiplying by $c$, the condition reads $n \alpha + d \beta \leq nd - c$ with $\alpha, \beta \geq 1$. Both $n \alpha + d \beta$ and $nd$ are multiples of $c$, so this is equivalent to $n \alpha + d \beta < nd$, that is, $( \alpha, \beta )$ is an interior lattice point of the triangle with vertices $(0,0)$, $(d,0)$, $(0,n)$. That triangle has area $\frac{nd}{2}$ and $d + n + c$ boundary lattice points, so by Pick's theorem the number of interior points is $\frac{nd}{2} - \frac{n+d+c}{2} + 1$, which is $\frac{(n-1)(d-1)+1-c}{2} = g_\XX$ by \cref{eq-genus}.
\end{proof}

The whole content of the next theorem is the following inequality. It says that the two linear families of \cref{prop-abc} never buy back more than the quadratic family gives up.

\begin{prop}\label{prop-key-inequality}
Let $n \geq 2$ and $d \geq 3$, let $c = \gcd (n,d)$, $n' = \frac{n}{c}$, $d' = \frac{d}{c}$ and $e = c n' d'$. For $0 \leq s \leq e-1$ let $r = e-1-s$ and define
\[
\begin{split}
& k_1 = \left\lfloor \frac{s}{n'} \right\rfloor , \quad k_2 = \left\lfloor \frac{s}{d'} \right\rfloor , \quad l_1 = \left\lfloor \frac{r}{n'} \right\rfloor , \quad l_2 = \left\lfloor \frac{r}{d'} \right\rfloor , \\
& \Theta \; = \; \# \left\{ \, ( \alpha, \beta ) \in \Z_{\geq 1}^2 \; : \; s < n' \alpha + d' \beta \leq e-1 \, \right\} .
\end{split}
\]
Then
\[
\Theta \; \geq \; \min \{ k_1, l_1 \} + \min \{ k_2, l_2 \} ,
\]
with the single exception $(n,d) = (3,3)$ and $s = r = 1$, where $\Theta = 1$ while the right hand side equals $2$.
\end{prop}

\begin{proof}
Let $\rho = \min \{ s, r \}$ and write $a = \lfloor \rho / n' \rfloor$, $b = \lfloor \rho / d' \rfloor$. Since the floor function is monotone, $a = \min \{ k_1, l_1 \}$ and $b = \min \{ k_2, l_2 \}$, so the claim reads $\Theta \geq a + b$, and $k_1, l_1 \geq a$, $k_2, l_2 \geq b$. Throughout, $e = n d' = d n'$.

We first prove the two bounds $\Theta \geq k_2 \, l_1$ and $\Theta \geq k_1 \, l_2$. Fix $\beta \geq 1$ with $d' \beta \leq s$; there are $k_2$ such $\beta$. The pairs $( \alpha, \beta )$ counted by $\Theta$ are those with $n' \alpha$ in the half open interval $( s - d' \beta , \, e-1-d' \beta ]$, of length $r$, and every such $\alpha$ is positive because $s - d' \beta \geq 0$. The interval contains at least $\lfloor r / n' \rfloor = l_1$ multiples of $n'$, so these $\beta$ contribute at least $k_2 \, l_1$ pairs to $\Theta$. Exchanging the roles of the two coordinates gives $\Theta \geq k_1 \, l_2$.

Suppose $a = 0$, so that the claim reads $\Theta \geq b$. If $k_1 = 0$ then $s < n'$, so $r \geq e - n' = n' ( d-1 )$ and $l_1 \geq d - 1 \geq 1$, whence $\Theta \geq k_2 l_1 \geq k_2 \geq b$. If $l_1 = 0$ then $r < n'$, so $s \geq e - n'$ and $k_1 \geq d - 1 \geq 1$, whence $\Theta \geq k_1 l_2 \geq l_2 \geq b$. If $b = 0$ the same argument applies with the two coordinates exchanged, using $e - d' = d' ( n-1 )$ and $n - 1 \geq 1$. From now on assume $a, b \geq 1$.

If $a \geq 2$ and $b \geq 2$, then $\Theta \geq k_2 l_1 \geq b \, a \geq a + b$. If $a = 1$ and $l_1 \geq 2$, then $\Theta \geq k_2 l_1 \geq 2 b \geq b + 1 = a + b$, and likewise $\Theta \geq k_1 l_2 \geq 2 b$ if $a = 1$ and $k_1 \geq 2$; the case $b = 1$ is symmetric. It remains to treat the cases $k_1 = l_1 = 1$ and $k_2 = l_2 = 1$.

\emph{Case $k_1 = l_1 = 1$.} Then $n' \leq s < 2n'$ and $n' \leq r < 2n'$, so $2n' \leq e - 1 \leq 4n' - 2$, and $e = d n'$ forces $d = 3$. Suppose first $3 \mid n$, so that $c = 3$, $d' = 1$, $b = \rho$, and $e = 3n'$. The pairs with $\alpha = 1$ are those with $\beta \in ( s - n' , \, 2n' - 1 ]$, and there are $3n' - 1 - s = r$ of them; the pairs with $\alpha = 2$ are those with $\beta \in [ 1, \, n' - 1 ]$, and there are $n' - 1$ of them; no $\alpha \geq 3$ occurs, since $\beta \geq 1$ forces $n' \alpha \leq 3n' - 2$. Hence $\Theta = r + n' - 1$, while $a + b = \rho + 1 \leq r + 1$, so the claim holds for $n' \geq 2$. For $n' = 1$ we are at $(n,d) = (3,3)$ with $s = r = 1$, where $\Theta = \# \{ (1,1) \} = 1$ and $a + b = 2$: the stated exception. Suppose now $3 \nmid n$, so that $c = 1$, $n' = n$ and $d' = 3$. For $n = 2$ one has $e = 6$ and $s, r \in \{ 2, 3 \}$, so $b = 0$, excluded above; hence $n \geq 4$. The pairs with $\alpha = 1$ number $\lfloor \frac{2n-1}{3} \rfloor - \lfloor \frac{s-n}{3} \rfloor \geq \lfloor \frac{3n-1-s}{3} \rfloor = l_2$, and those with $\alpha = 2$ number $\lfloor \frac{n-1}{3} \rfloor \geq 1$. Hence $\Theta \geq l_2 + 1 \geq b + 1 = a + b$.

\emph{Case $k_2 = l_2 = 1$.} Then $d' \leq s < 2d'$ and $d' \leq r < 2d'$, so $e = n d'$ forces $n = 3$, and $(n,d) = (3,3)$ was treated in the previous case. Suppose first $3 \mid d$, so that $c = 3$, $n' = 1$, $a = \rho$, and $e = 3d'$ with $d' = d/3 \geq 2$. Slicing on $\beta \in \{ 1, 2 \}$ as in the previous case gives $\Theta = r + d' - 1 \geq \rho + 1 = a + b$. Suppose now $3 \nmid d$, so that $c = 1$, $d' = d \geq 4$ and $n' = 3$. The pairs with $\beta = 1$ number $\lfloor \frac{2d-1}{3} \rfloor - \lfloor \frac{s-d}{3} \rfloor \geq \lfloor \frac{3d-1-s}{3} \rfloor = l_1$, and those with $\beta = 2$ number $\lfloor \frac{d-1}{3} \rfloor \geq 1$. Hence $\Theta \geq l_1 + 1 \geq a + b$.
\end{proof}

\begin{thm}\label{thm-hmax}
Assume the criterion of \cref{prop-transitive-criterion} and suppose $m = d$. Then $1 \leq e - 1 \leq M$,
\(
\dim \Hull \left( \CC_{e-1} \right) \; = \; \dim \Hull \left( \CC_{e} \right) \; = \; g_\XX ,
\)
and
\[
\max_{s \geq 1} \, \dim \Hull ( \CC_s ) \; = \;
\begin{cases}
g_\XX + 1 = 2 , & (n,d) = (3,3), \\
g_\XX , & \text{otherwise} .
\end{cases}
\]
\end{thm}

\begin{proof}
First, $e - 1 \geq 1$, and $d = m \mid q-1$ gives $e = \frac{nd}{c} \leq \frac{(q-1)n}{c} = M + 1$, so $e - 1 \leq M$. With $m = d$ one has $e_m = e$ and $\Omega = 2e$. Write $n' = \frac{n}{c}$ and $d' = \frac{d}{c}$, and define
\[
\Lambda (k) \; = \; \# \left\{ \, ( \alpha, \beta ) \in \Z_{\geq 1}^2 \; : \; n' \alpha + d' \beta \leq k \, \right\} ,
\]
so that $\Lambda ( e-1 ) = g_\XX$ by \cref{lem-genus-lattice}. In the count $h_3 (s)$ of \cref{prop-abc} the constraints $\alpha \leq m-1$ and $\beta \leq n-1$ are automatic whenever $n' \alpha + d' \beta \leq e-1$, since then $n' \alpha \leq n'd - 1$ and $d' \beta \leq d'n - 1$. As $\min \{ s, 2e-1-s \} \leq e-1$ for every $s$, we have $h_3 (s) = \Lambda ( \min \{ s, 2e-1-s \} )$ throughout.

Let $s \geq e$. Then $\min \{ s, e-1-s \} < 0$, so $h_1 (s) = h_2 (s) = 0$ and $\dim \Hull ( \CC_s ) = \Lambda ( 2e-1-s ) \leq \Lambda ( e-1 ) = g_\XX$, with equality at $s = e$.

Now let $0 \leq s \leq e-1$ and let $\rho = \min \{ s, e-1-s \}$. Then $\min \{ s, 2e-1-s \} = s$, so \cref{prop-abc} gives
\[
\dim \Hull ( \CC_s ) \; = \; \left\lfloor \frac{\rho}{n'} \right\rfloor + \left\lfloor \frac{\rho}{d'} \right\rfloor + \Lambda (s) .
\]
At $s = e-1$ this yields $\rho = 0$ and $\dim \Hull ( \CC_{e-1} ) = \Lambda ( e-1 ) = g_\XX$. Subtracting $\Lambda (s)$, the inequality $\dim \Hull ( \CC_s ) \leq g_\XX$ reads $\lfloor \rho / n' \rfloor + \lfloor \rho / d' \rfloor \leq g_\XX - \Lambda (s) = \Theta$, which is \cref{prop-key-inequality} with $r = e-1-s$. For $(n,d) \neq (3,3)$ that proposition supplies the inequality for every $s$, and the maximum is $g_\XX$. For $(n,d) = (3,3)$ one has $n' = d' = 1$, $e = 3$ and $g_\XX = 1$, and the displayed formula reads $\dim \Hull ( \CC_s ) = 2 \rho + \Lambda (s)$. This equals $2$ at $s = 1$ and is at most $1$ at $s \in \{ 0, 2 \}$, while $\dim \Hull ( \CC_s ) \leq g_\XX = 1$ for $s \geq 3$. Hence the maximum is $g_\XX + 1 = 2$.
\end{proof}

The bound $\dim \Hull ( \CC_s ) \leq (nm-t_2)/2$ of \cref{prop-torsion-bound} holds for every $s$ and is attained in \cref{exa-torsion}. Off the locus $m = d$ it is the only bound available: by \cref{cor-hmax-computable} the maximal hull dimension is the covering number of the family $\{ I_\chi \}$, and not a function of $g_\XX$ and $t_2$. On the locus $m = d$, \cref{thm-hmax} and \cref{cor-epsilon-transitive} bound the excess: for every $s$ with $c \, \nu (s) < N$,
\[
\varepsilon ( G_s, A_s ) \; = \; \dim \Hull ( \CC_s ) - \Phi \left( \nu (s) \right) \; \leq \; g_\XX - \Phi \left( \nu (s) \right) \; \leq \; g_\XX
\]
unless $(n,d) = (3,3)$, with equality throughout when $s \in \{ e-1, e \}$ and $\Phi ( \nu (s) ) = 0$. Like the problems left open in \cref{sec-window}, this bounds the dimension of the space of decomposable elements without identifying them.

%-----------------------------------------------------------------------------------------
%-----------------------------------------------------------------------------------------
\section{Computations: measuring the excess}\label{sec-comput}

By \cref{thm-quasi} the hull of $\CC_s$ is the lattice count $\Phi ( \nu (s) )$ plus the excess, and the first term depends on $\XX$ only through $(n,d,q)$ and the integer $\delta$. The excess is therefore computable from any table of hull dimensions, by \cref{eq-epsilon-comp}, and it is the only quantity in the family that is not determined by the numerical data. This section computes it along the stratification of the moduli space by automorphism group, using the equations of the strata given in \cite{sa-sh}.

Three things are tested: the exactness of \cref{eq-main-dim} inside the window, the size of the excess outside it, and what the excess sees outside it. For the last, \cref{prop-eps-moduli} exhibits over $\F_7$ a hull profile asymmetric about $M/2$; by \cref{cor-asym} that asymmetry is excess, at two degrees which present the same meet, the same join and the same class of $\Pic ( \XX )$. Which classes the family presents at all is settled first, and the answer restricts what \cref{cor-detect} can say here.

All computations are carried out over prime fields $\F_q$ with $n \mid q-1$, so that \textup{(H1)}--\textup{(H4)} are available. For a given $f$ we verify that $f$ is squarefree with no root in $\F_q$, and take $\YY$ to be the set of all affine $\F_q$-rational points, so that $\delta$, $N = n(q-\delta)$ and $\deg E = n \delta$ are as in \cref{lem-pullback} and $M$, $\mu$, $\nu$, $t$, $W ( \XX )$, $B ( \XX )$ and $\Phi$ are determined by $(n,d,q,\delta)$ alone.

The code $\CC_s$ is generated by the rows of the $| B_s | \times N$ matrix $\left( x(P)^i y(P)^j \right)_{(i,j) \in B_s, \, P \in \YY}$, where $B_s$ is the monomial basis of \cref{lem-bridge}; no normalisation is required, since $z = 1$ at every affine point. If $G_s$ is a row-reduced generator matrix of $\CC_s$ of rank $k_s$, then
\begin{equation}\label{eq-hull-rank}
\dim \Hull ( \CC_s ) = k_s - \operatorname{rank} \left( G_s \, G_s^{\, t} \right) ,
\end{equation}
all ranks being computed by Gaussian elimination over $\F_q$. The excess is then read off from \cref{eq-quasi} as
\begin{equation}\label{eq-epsilon-comp}
\varepsilon ( G_s, A_s ) \; = \; \dim \Hull ( \CC_s ) - \Phi \left( \nu (s) \right) ,
\end{equation}
which is the quantity of interest outside the window, where by \cref{thm-hull} it is the only thing left to compute.

The classes of $\Pic ( \XX )$ the family actually presents are far fewer than \cref{sec-window} allows. Every divisor occurring in \cref{eq-meetjoin} is supported on $\Ram$, $E$ and $D_\infty$, and the first two are already equivalent to multiples of the third.

\begin{lem}\label{lem-classes}
Assume \textup{(H1)}--\textup{(H4)} and put $\gamma = \frac{2 g_\XX - 2}{c}$, an integer because $2 g_\XX - 2 = nd - n - d - c$ by \cref{eq-genus} and $c$ divides both $n$ and $d$. Then, in $\Pic ( \XX )$,
\[
\begin{split}
\Ram \; &\sim \; \frac{d}{c} \, D_\infty , \qquad E \; \sim \; \frac{n \delta}{c} \, D_\infty , \qquad D \; \sim \; \frac{N}{c} \, D_\infty , \qquad K \; \sim \; \gamma \, D_\infty , \\
G_s \wedge A_s \; &\sim \; \nu (s) \, D_\infty , \qquad K - G_s \wedge A_s \; \sim \; \left( \gamma - \nu (s) \right) D_\infty \; = \; \frac{t (s)}{c} \, D_\infty ,
\end{split}
\]
the last two for every $s \geq 1$.
\end{lem}

\begin{proof}
The first equivalence is \cref{eq-divisors} and the second is \cref{lem-pullback}. For the third, the proof of \cref{prop-dual} gives $\dv ( x^q - x ) = D + E - \frac{qn}{c} D_\infty$, so that $D \sim \frac{qn}{c} D_\infty - E \sim \frac{n (q - \delta)}{c} D_\infty = \frac{N}{c} D_\infty$ by \cref{lem-pullback}. For the fourth, \cref{eq-eta} exhibits $K$ as $(n-1) \Ram - D - E + M D_\infty$, which is therefore equivalent to $\lambda D_\infty$ with
\[
\lambda \; = \; \frac{(n-1)d}{c} - \frac{n (q - \delta)}{c} - \frac{n \delta}{c} + \frac{(q-1)n}{c} - 1 \; = \; \frac{(n-1)d - n}{c} - 1 \; = \; \frac{nd - n - d - c}{c} \; = \; \gamma .
\]
By \cref{eq-meetjoin} and \cref{def-nu} we have $G_s \wedge A_s = \mu (s) D_\infty - E \sim \left( \mu (s) - \frac{n \delta}{c} \right) D_\infty = \nu (s) D_\infty$, and subtracting gives the last equivalence, of degree $2 g_\XX - 2 - c \nu (s) = t (s)$ by \cref{eq-tofs}.
\end{proof}

\begin{cor}\label{cor-vacuous}
Assume \textup{(H1)}--\textup{(H4)} and let $s$ satisfy $t (s) \geq 0$. Then $K - G_s \wedge A_s$ is linearly equivalent to the effective divisor $\frac{t(s)}{c} D_\infty$, and its class lies in the cyclic subgroup of $\Pic ( \XX )$ generated by $[ D_\infty ]$, where it is determined by $(n,d,q,\delta)$ and $s$ alone. In particular the conclusion of \cref{cor-detect} holds on every curve satisfying \textup{(H1)}--\textup{(H4)}, at every such $s$ and whatever the value of $\varepsilon$; and when $t (s) = 0$ one has
\[
G_s \vee A_s \; \sim \; D
\]
on every such curve.
\end{cor}

\begin{proof}
Immediate from \cref{lem-classes}, the divisor $\frac{t(s)}{c} D_\infty$ being effective when $t(s) \geq 0$. The last assertion is the case $\nu (s) = \gamma$, together with $G_s \vee A_s - D = K - G_s \wedge A_s$ from \cref{eq-join-canonical}.
\end{proof}

This is a limitation of the family, not of \cref{sec-window}: the condition that \cref{cor-detect} imposes on the class of $K - G \wedge A$ is met automatically here, because the canonical class and the class of $D$ are both multiples of $[ D_\infty ]$, so the effectivity that a nonvanishing excess certifies is certified already by $\dv (y)$ and $\dv ( x^q - x )$. The excess therefore carries information only through the failure of the converse: a generator of $L ( K - G \wedge A )$ must in addition be decomposable along $G$ and $A$. \cref{lem-classes} also identifies the containing space, making the bound \cref{eq-eps-h1} explicit.

\begin{cor}\label{cor-eps-bound}
Assume \textup{(H1)}--\textup{(H4)}. Then for every $s \geq 1$
\[
\varepsilon ( G_s, A_s ) \; \leq \; \ell \left( \left( \gamma - \nu (s) \right) D_\infty \right) \; = \; \Phi \left( \gamma - \nu (s) \right) ,
\]
and consequently, for every $s$ with $c \, \nu (s) < N$,
\[
\dim \Hull ( \CC_s ) \; \leq \; \Phi ( \nu (s) ) + \Phi \left( \gamma - \nu (s) \right) \; = \; 2 \, \Phi ( \nu (s) ) - c \, \nu (s) - 1 + g_\XX .
\]
\end{cor}

\begin{proof}
The first assertion is \cref{eq-eps-h1} together with \cref{lem-classes} and \cref{eq-ell}, the dimension of a Riemann-Roch space depending only on the class of the divisor. The second follows from \cref{eq-quasi}, and the closing identity is Riemann-Roch for $\nu (s) D_\infty$, whose index of speciality is $\ell ( K - \nu (s) D_\infty ) = \Phi ( \gamma - \nu (s) )$ by \cref{lem-classes}.
\end{proof}

\cref{prop-dichotomy} shows that the degenerate zone (Z0) is nonempty only on the totally split locus, and that off that locus the divisorial criterion for self-orthogonality fails at every $s$. The next theorem shows that off that locus self-orthogonality itself fails, and does so for every $s$ at once.

\begin{thm}\label{thm-so}
Assume \textup{(H1)}--\textup{(H4)} and let $1 \leq s \leq M$. If $\CC_s$ is self-orthogonal, then
\begin{equation}\label{eq-power-sums}
\sum_{a \in \Delta} a^{\, i} \; = \; 0 \quad \text{in } \F_q, \qquad \text{for every } 0 \leq i \leq \left\lfloor \frac{cs}{n} \right\rfloor .
\end{equation}
In particular $p \mid \delta$.
\end{thm}

\begin{proof}
Fix $i$ with $\tfrac{n}{c} i \leq s$ and put $v = x^i$, so that $v \in B_s \subseteq L ( s D_\infty )$ and $1 \in B_s$ as well. Self-orthogonality gives $\ev_s (1) \cdot \ev_s (v) = 0$, that is $\sum_{P \in \Supp D} v (P) = 0$. Consider $\omega = v \, \eta$ with $\eta$ as in \cref{prop-dual}. Since $\dv (v) \geq - s D_\infty$ and $\dv ( \eta )$ is given by \cref{eq-eta}, we get $\dv ( \omega ) \geq (n-1) \Ram - D - E + ( M - s ) D_\infty \geq - D - E$ using $s \leq M$. Hence $\omega$ has at worst simple poles, confined to $\Supp D \cup \Supp E$. If $P \in \Supp D \cup \Supp E$ lies over $a \in \F_q$ then $P$ is unramified by (H1), so $x-a$ is a uniformiser at $P$ and the computation in the proof of \cref{prop-dual} gives $\Res_P ( \eta ) = 1$, whence $\Res_P ( \omega ) = v(P)$. The residue theorem therefore reads
\[
\sum_{P \in \Supp D} v (P) \; + \; \sum_{P \in \Supp E} \operatorname{Tr}_{k(P) / \F_q} \left( v(P) \right) \; = \; 0 ,
\]
the first sum having no traces because those points are rational. The second sum is the sum of $v$ over the geometric points of $E$, which by \cref{lem-pullback} are the pairs $(a,b)$ with $a \in \Delta$ and $b^n = f(a)$; for $v = x^i$ it equals $n \sum_{a \in \Delta} a^i$. As the first sum vanishes and $p \nmid n$ by (H1), we obtain \cref{eq-power-sums}. Taking $i = 0$ gives $p \mid \delta$.
\end{proof}

\begin{cor}\label{cor-so-dichotomy}
Assume \textup{(H1)}--\textup{(H4)}, that $q = p$ is prime, and that $\YY \neq \emptyset$. Then $\CC_s$ is self-orthogonal for some $s \geq 1$ if and only if $\XX$ is totally split.
\end{cor}

\begin{proof}
Sufficiency is \cref{prop-dichotomy}, the range $s \leq \lfloor M/2 \rfloor$ being nonempty since $M \geq 1$. Conversely, suppose $\CC_s$ is self-orthogonal for some $s \geq 1$; we claim $\delta = 0$. As $\dim \CC_s$ is nondecreasing in $s$ and self-orthogonality is inherited downwards, we may assume $s \leq M$, and \cref{thm-so} gives $p \mid \delta$. If $\delta \geq 1$ then $\delta = q$, since $\delta \leq q = p$; hence $\Delta = \F_q$ and $N = 0$, contradicting $\YY \neq \emptyset$.
\end{proof}

Over non-prime fields \cref{eq-power-sums} is a genuine additional constraint: it says that $\sum_{a \in \Delta} g (a) = 0$ in $\F_q$ for every polynomial $g \in \F_q [x]$ of degree at most $\lfloor cs/n \rfloor$.

The first computations are across the strata of a genus two locus. Let $n = 2$ and $d = 6$, so that $c = 2$, $\bw = (1,1,3)$, $g_\XX = 2$ and $M = q-2$. Representatives of the strata of the corresponding locus are
\begin{equation}\label{eq-strata}
\begin{split}
\Aut ( \XX ) &= \Z_2 : \quad y^2 = f(x) \quad \text{generic}, \\
\Aut ( \XX ) &= V_4 : \quad y^2 = x^6 + a x^4 + b x^2 + 1, \qquad x \mapsto -x, \\
\Aut ( \XX ) &= D_8 : \quad y^2 = x^6 + a x^4 + a x^2 + 1, \qquad x \mapsto -x, \; x \mapsto 1/x, \\
\Aut ( \XX ) &= D_{12} : \quad y^2 = x^6 + a x^3 + 1, \qquad x \mapsto \zeta_3 x, \; x \mapsto 1/x, \\
\Aut ( \XX ) &= 2 D_{12} : \quad y^2 = x^6 + 1 .
\end{split}
\end{equation}
The $D_8$ stratum is the locus on which $f$ is self-reciprocal, $x^6 f (1/x) = f(x)$, studied from the point of view of self-inversive polynomials in \cite{sh-si}; the $D_{12}$ stratum is the locus on which $f$ is a polynomial in $x^3$ up to the reciprocal normalisation.

For $n = 2$, $d = 6$ the criterion \cref{eq-window-nonempty} reads $2 \lfloor M/2 \rfloor - 2 \delta > 2$, that is
\begin{equation}\label{eq-window-genus2}
\delta \; \leq \; \left\lfloor \frac{q-2}{2} \right\rfloor - 2 ,
\end{equation}
and by \cref{prop-window} the detection band is $\{ s : \mu (s) = \delta + 1 \}$, nonempty precisely when $\delta \leq \lfloor M/2 \rfloor - 1$. \cref{tab-strata} records, for $q = 11$ and $q = 19$ and one curve in each stratum, the invariants $N$, $\delta$ and $m(f)$, the largest value of the lattice count, and the largest hull dimension $h_{\max} = \max_s \dim \Hull ( \CC_s )$.

\begin{table}[ht]
\caption{Hull data across the strata of the genus two locus, $n=2$, $d=6$. The window and the detection band are empty in every line, so every degree lies in the saturated zone \textup{(Z3)}.}\label{tab-strata}
\begin{tabular}{llrrrrr}
\toprule
$\Aut ( \XX )$ & equation & $N$ & $\delta$ & $m(f)$ & $\max_s \Phi ( \nu (s) )$ & $h_{\max}$ \\
\midrule
\multicolumn{7}{l}{$q = 11$, \; $M = 9$, \; window nonempty iff $\delta \leq 2$, \; band nonempty iff $\delta \leq 3$} \\
$\Z_2$    & $y^2 = x^6 + 2x^3 + x + 1$    & $12$ & $5$ & $1$ & $0$ & $1$ \\
$V_4$     & $y^2 = x^6 + x^2 + 1$         & $14$ & $4$ & $2$ & $1$ & $1$ \\
$D_8$     & $y^2 = x^6 + x^4 + x^2 + 1$   & $6$  & $8$ & $2$ & $0$ & $0$ \\
$D_{12}$  & $y^2 = x^6 + x^3 + 1$         & $10$ & $6$ & $1$ & $0$ & $1$ \\
$2D_{12}$ & $y^2 = x^6 + 1$               & $10$ & $6$ & $2$ & $0$ & $1$ \\
\midrule
\multicolumn{7}{l}{$q = 19$, \; $M = 17$, \; window nonempty iff $\delta \leq 6$, \; band nonempty iff $\delta \leq 7$} \\
$\Z_2$    & $y^2 = x^6 + x + 1$           & $16$ & $11$ & $1$ & $0$ & $1$ \\
$V_4$     & $y^2 = x^6 + x^2 + 1$         & $18$ & $10$ & $2$ & $0$ & $1$ \\
$D_8$     & $y^2 = x^6 + x^4 + x^2 + 1$   & $22$ & $8$  & $2$ & $1$ & $1$ \\
$D_{12}$  & $y^2 = x^6 + 3x^3 + 1$        & $8$  & $15$ & $3$ & $0$ & $1$ \\
$2D_{12}$ & $y^2 = x^6 + 1$               & $2$  & $18$ & $6$ & $0$ & $0$ \\
\bottomrule
\end{tabular}
\end{table}

Three features of \cref{tab-strata} are accounted for by the theory. First, every degree lies in the saturated zone: since $\mu (s) \leq \lfloor M/2 \rfloor$ and $\delta \geq \lfloor M/2 \rfloor$ in all ten lines, \cref{eq-tofs} gives $t (s) = 2 - 2 \mu (s) + 2 \delta \geq 2 = g_\XX$ throughout. By \cref{cor-vacuous} the effectivity condition of \cref{cor-detect} is therefore vacuous, and what the excess carries on these strata is its value, not its vanishing. On the eight lines with $\delta > \lfloor M/2 \rfloor$ one has $\nu (s) < 0$ for every $s$, so $\Phi \circ \nu$ vanishes identically and $\dim \Hull ( \CC_s ) = \varepsilon ( G_s, A_s )$ for every $s$: the entire hull is excess, which is why the values are small and supported at isolated degrees. On the two remaining lines, $V_4$ at $q=11$ and $D_8$ at $q=19$, one has $\delta = \lfloor M/2 \rfloor$, hence $\nu (s) = 0$ at the peak and $\Phi = 1$ there, and $h_{\max} = 1$ is already accounted for by the lattice count.

Second, none of these codes is self-orthogonal for any $s \geq 1$. All the computations are over prime fields and none of the curves is totally split, so \cref{cor-so-dichotomy} applies in every line, on every stratum and for every $q$. In the ten computed examples of \cref{tab-strata}, the codes are \textup{LCD} for all but at most four values of $s$. This is a feature of those examples, not a consequence of \cref{cor-so-dichotomy} alone.

Third, the column $m(f)$ measures the cyclic scaling subgroup of the reduced fibration-preserving automorphism group that survives over $\F_q$ and fixes both $0$ and $\infty$. By \cref{prop-stabilizer} this subgroup acts on every code $\CC_s$; no assertion is made there that it is the full permutation automorphism group. The geometric reduced group of the $D_{12}$ and $2 D_{12}$ strata contains $\zeta_3$ and $\zeta_6$ respectively, but $3 \nmid 10$, so over $\F_{11}$ these do not descend and $m(f) = 1$ and $2$; over $\F_{19}$ one has $6 \mid 18$ and $m(f) = 3$ and $6$. Thus $H = \Z_2 \times \Z_{m(f)}$ jumps from order $4$ to order $12$ between the two rows for $y^2 = x^6+1$ without any change in the curve, purely arithmetically. In the line $2 D_{12}$ at $q=19$ one has $|H| = 12$ and $N = 2$, so $|H| \nmid N$ and the action of $H$ on $\YY$ is not free; this is the failure of the hypothesis $g (0) \notin ( \F_q^* )^n$ of \cref{prop-transitive-criterion}, since $f(0) = 1$ is a square in $\F_{19}$, so the two points of $\YY$ lie over $x=0$ and $\sigma_\lambda$ fixes them both.

The hull profile is not a function of the numerical data.

\begin{exa}\label{exa-sameN}
Over $\F_{11}$ the curves $y^2 = x^6 + x^3 + 1$ and $y^2 = x^6 + 1$, lying in the strata with automorphism groups $D_{12}$ and $2 D_{12}$, both have $N = 10$, $\delta = 6$ and identical dimension sequences $\dim \CC_s = 2,3,5,7,8,9,10, \dots$ for $s = 1, 2, \dots$. Their hull sequences are
\[
\begin{split}
y^2 = x^6 + x^3 + 1 : \quad & 0,\, 1,\, 0,\, 0,\, 0,\, 0,\, 0, \dots \\
y^2 = x^6 + 1 : \quad & 0,\, 0,\, 0,\, 0,\, 1,\, 1,\, 0, \dots
\end{split}
\]
so that $\dim \Hull ( \CC_s )$ is determined neither by $(n,d,q,N)$ nor by $(n,d,q,\delta)$. Since $\delta = 6 > \lfloor M/2 \rfloor = 4$ both sequences are pure excess, taking the value $1$ at $s=2$ for the first curve and at $s = 5,6$ for the second. Both profiles are asymmetric about $M/2 = 4.5$, in agreement with \cref{prop-asym}: the unequal values at reflected degrees are exactly unequal excess values. The two curves are distinguished by $m(f) = 1$ and $m(f) = 2$, so the example does not exclude the finer invariant $(n,d,q,\delta,m(f))$; what it shows is that the excess, and with it the moduli point, is not a function of the numerical data entering \cref{thm-hull}.
\end{exa}

We now examine the totally split curves of \cref{def-split}, for which $\delta = 0$ and $\nu (s) = \mu (s)$. Such curves exist but are rare. Over $\F_7$ with $n=2$, $d=6$ an exhaustive search over the $3^7$ admissible value vectors produces several, for instance
\[
\XX_7 : \quad y^2 = 6x^6 + x^5 + 6x^4 + x^3 + 6x^2 + x + 1, \qquad N = 14 = 7 \cdot 2 ,
\]
while over $\F_{11}$ a random search over value vectors produced
\begin{equation}\label{eq-x11}
\XX_{11} : \quad y^2 = 5x^6 + 9x^5 + 6x^4 + 4x^3 + 7x^2 + 4x + 3, \qquad N = 22 = 11 \cdot 2 .
\end{equation}
Both have $m(f) = 1$, since each involves $x$ to the first power, so $H = \langle \tau \rangle$. The leading coefficient in \cref{eq-x11} is a square in $\F_{11}$, so both points of $D_\infty$ are rational and $\# \XX_{11} ( \F_{11} ) = 24$, which is the Serre bound $11 + 1 + 2 \lfloor 2 \sqrt{11} \rfloor$ for a genus two curve over $\F_{11}$; the curve is optimal. A totally split curve has the largest possible number of affine rational points, so the totally split locus lies inside the locus of curves with many rational points, and for small $q$ it consists of curves at or near the Serre bound.

For both curves $t (s) = 2 - 2 \mu (s)$, so the window is $\{ s : \mu (s) \geq 2 \}$ and the detection band is $\{ s : \mu (s) = 1 \}$. Thus $W ( \XX_7 ) = \{ 2, 3 \}$ and $B ( \XX_7 ) = \{ 1, 4 \}$, while $W ( \XX_{11} ) = \{ 2, \dots , 7 \}$ and $B ( \XX_{11} ) = \{ 1, 8 \}$.

\begin{table}[ht]
\caption{The totally split curves $\XX_7$ and $\XX_{11}$: computed hull dimensions against \cref{thm-quasi}. Degrees in the detection band are those with $t(s) = 0$.}\label{tab-split}
\begin{tabular}{lrrrrrrrrrrr}
\toprule
\multicolumn{12}{l}{$\XX_7$, \; $M = 5$, \; $W = \{ 2,3 \}$, \; $B = \{ 1,4 \}$} \\
$s$ & $1$ & $2$ & $3$ & $4$ & $5$ & $6$ & $7$ & $8$ & & & \\
\midrule
$\mu (s) = \nu (s)$        & $1$ & $2$ & $2$ & $1$ & $0$  & $-1$ & $-2$ & $-3$ & & & \\
$t (s)$                    & $0$ & $-2$ & $-2$ & $0$ & $2$ & $4$ & $6$ & $8$ & & & \\
$\Phi ( \nu (s) )$         & $2$ & $3$ & $3$ & $2$ & $1$ & $0$ & $0$ & $0$ & & & \\
$\dim \Hull ( \CC_s )$     & $2$ & $3$ & $3$ & $3$ & $3$ & $2$ & $1$ & $0$ & & & \\
$\varepsilon ( G_s, A_s )$ & $0$ & $0$ & $0$ & $1$ & $2$ & $2$ & $1$ & $0$ & & & \\
\midrule
\multicolumn{12}{l}{$\XX_{11}$, \; $M = 9$, \; $W = \{ 2, \dots , 7 \}$, \; $B = \{ 1,8 \}$} \\
$s$ & $1$ & $2$ & $3$ & $4$ & $5$ & $6$ & $7$ & $8$ & $9$ & $10$ & $11$ \\
\midrule
$\mu (s) = \nu (s)$        & $1$ & $2$ & $3$ & $4$ & $4$ & $3$ & $2$ & $1$ & $0$ & $-1$ & $-2$ \\
$t (s)$                    & $0$ & $-2$ & $-4$ & $-6$ & $-6$ & $-4$ & $-2$ & $0$ & $2$ & $4$ & $6$ \\
$\Phi ( \nu (s) )$         & $2$ & $3$ & $5$ & $7$ & $7$ & $5$ & $3$ & $2$ & $1$ & $0$ & $0$ \\
$\dim \Hull ( \CC_s )$     & $2$ & $3$ & $5$ & $7$ & $7$ & $5$ & $3$ & $2$ & $1$ & $0$ & $0$ \\
$\varepsilon ( G_s, A_s )$ & $0$ & $0$ & $0$ & $0$ & $0$ & $0$ & $0$ & $0$ & $0$ & $0$ & $0$ \\
\bottomrule
\end{tabular}
\end{table}

\begin{prop}\label{prop-eps-moduli}
For the curve $\XX_7$ the degrees $s = 1$ and $s = 4$ satisfy
\[
G_1 \wedge A_1 = G_4 \wedge A_4 = D_\infty , \qquad G_1 \vee A_1 = G_4 \vee A_4 = \Ram + 4 D_\infty , \qquad t (1) = t (4) = 0 ,
\]
while $\varepsilon ( G_1, A_1 ) = 0$ and $\varepsilon ( G_4, A_4 ) = 1$. Hence the excess is a function neither of the pair $( G_s \wedge A_s, \, G_s \vee A_s )$ nor of the class $[ K - G_s \wedge A_s ] \in \Pic^0 ( \XX_7 )$, both of which the two degrees present identically.
\end{prop}

\begin{proof}
The three equalities are \cref{prop-asym} with $M = 5$, together with \cref{eq-meetjoin} and $\mu (1) = \mu (4) = 1$, $\max \{ 1, 4 \} = 4$ and $E = 0$. Since $\XX_7$ is totally split and $1 \leq \lfloor M/2 \rfloor = 2$, \cref{prop-dichotomy} gives $\varepsilon ( G_1, A_1 ) = 0$, and \cref{cor-asym} then gives $\varepsilon ( G_4, A_4 ) = \dim \Hull ( \CC_4 ) - \dim \Hull ( \CC_1 ) = 3 - 2 = 1$, in agreement with \cref{tab-split}. The common class is trivial by \cref{cor-vacuous}, since $t (1) = t (4) = 0$.
\end{proof}

\cref{cor-detect} attaches to the nonvanishing $\varepsilon ( G_4, A_4 ) = 1$ the equivalence $\Ram + 4 D_\infty \sim D$ in $\Pic ( \XX_7 )$. The equivalence holds, but it is not detected: by \cref{cor-vacuous} it holds on every curve satisfying \textup{(H1)}--\textup{(H4)} at a degree with $t (s) = 0$, being for $\XX_7$ the combination $\Ram \sim 3 D_\infty$, $D \sim 7 D_\infty$ of $\dv (y)$ and $\dv ( x^7 - x )$. On $\XX_{11}$ the corresponding degrees are $s = 1$ and $s = 8$, the class they present is trivial by \cref{lem-classes}, and the excess vanishes at both; the profile of \cref{tab-split} is symmetric about $M/2 = 4.5$, as \cref{prop-asym} requires whenever the excess vanishes identically. The asymmetry of the profile of $\XX_7$ is therefore a certificate not of a linear equivalence but of the decomposability of a generator of $L ( K - G_s \wedge A_s )$ along the two divisors, obtainable from \cref{eq-hull-rank} alone; a symmetric profile certifies nothing.

The remaining degrees of $\XX_7$ lie in the saturated zone and show how far the excess is from its bound \cref{eq-eps-h1}, which by \cref{cor-eps-bound} is here $\Phi ( \gamma - \nu (s) )$ with $\gamma = 1$. At $s = 5$ one has $G_5 \wedge A_5 = 0$, so $K - G_5 \wedge A_5 = K$ and $\ell (K) = g_\XX = 2$; the computed $\varepsilon = 2$ says that the whole canonical space is decomposable. At $s = 7$ one has $K - G_7 \wedge A_7 = K + 2 D_\infty$ of degree $6$, so $\ell = 5$, while $\varepsilon = 1$: only a line of a five-dimensional space is decomposable. Determining which is the content of the problem left open in \cref{sec-window}.

Where the totally split curves sit inside the family is shown by an exhaustive computation. The substitution $f(x) \mapsto c^n f(ax+b)$, with $a, c \in \F_q^*$ and $b \in \F_q$, corresponds to the isomorphism $(x,y) \mapsto (ax+b, y/c)$ from the curve $y^n = c^n f(ax+b)$ to the curve $y^n = f(x)$. It carries the points at infinity to the points at infinity, so it identifies the spaces $L ( s D_\infty )$ of the two curves for every $s$, and it maps affine rational points bijectively; hence it induces a permutation equivalence of every code $\CC_s$ and preserves $\dim \Hull ( \CC_s )$. For $q = 11$, $n = 2$, $d = 6$ we computed $\dim \Hull ( \CC_s )$ over the full range $1 \leq s < s^+$ for every squarefree polynomial $f$ of degree $6$ with no root in $\F_{11}$, one from each of the $11372$ classes under this substitution, four of which have $N = 0$; \cref{tab-scan} records the results. By \cref{eq-window-genus2} the window is nonempty exactly for $\delta \leq 2$, that is for $N \geq 18$, and there \cref{thm-hull} gives
\begin{equation}\label{eq-hmax-genus2}
\max_{s \in W ( \XX )} \dim \Hull ( \CC_s ) \; = \; c \left\lfloor M/2 \right\rfloor - n \delta + 1 - g_\XX \; = \; 7 - 2 \delta .
\end{equation}

\begin{table}[ht]
\caption{Maximal hull dimension against the number of affine rational points, $q = 11$, $n=2$, $d=6$. The gap between the last two rows is a lower bound for $\max_s \varepsilon ( G_s, A_s )$.}\label{tab-scan}
\begin{tabular}{lrrrrrrrrrrr}
\toprule
$N$ & $2$ & $4$ & $6$ & $8$ & $10$ & $12$ & $14$ & $16$ & $18$ & $20$ & $22$ \\
$\delta$ & $10$ & $9$ & $8$ & $7$ & $6$ & $5$ & $4$ & $3$ & $2$ & $1$ & $0$ \\
\midrule
classes of $f$ & $56$ & $328$ & $931$ & $1802$ & $2565$ & $2565$ & $1802$ & $931$ & $328$ & $56$ & $4$ \\
$\max_s \Phi ( \nu (s) )$ & $0$ & $0$ & $0$ & $0$ & $0$ & $0$ & $1$ & $2$ & $3$ & $5$ & $7$ \\
$h_{\max}$ & $0$ & $0$ & $1$ & $2$ & $2$ & $2$ & $3$ & $3$ & $3$ & $5$ & $7$ \\
\bottomrule
\end{tabular}
\end{table}

For $N \geq 18$ the window is nonempty and the two lower rows agree, the computed maximum $3, 5, 7$ being exactly \cref{eq-hmax-genus2}; the last column consists of the four classes of totally split curves, among them that of $\XX_{11}$ of \cref{eq-x11}, which is consistent with the rarity of such curves. For $N \leq 16$ the window is empty and the two rows separate, the gap being carried by the excess. The maximal hull dimension is non-decreasing in $N$ over the whole range and grows sharply as $\delta$ approaches zero, from $2$ at the generic value $\delta = 5$ to $7$ at $\delta = 0$. The lower bound implicit in \cref{eq-hmax-genus2} is a theorem.

\begin{prop}\label{prop-monotone-bound}
Fix $(n,d,q)$ and let $\XX$ satisfy \textup{(H1)}--\textup{(H4)}. If $W ( \XX ) \neq \emptyset$ then
\[
h_{\max} ( \XX ) \; \geq \; c \left\lfloor M/2 \right\rfloor - n \delta + 1 - g_\XX ,
\]
and the right hand side is strictly decreasing in $\delta$, hence strictly increasing in $N$. The window is nonempty for every $\XX$ with $n \delta < c \lfloor M/2 \rfloor - 2 g_\XX + 2$, in particular for every totally split curve with $n (q-1) > 2 (n-1)(d-1)$.
\end{prop}

\begin{proof}
The inequality is \cref{thm-hull} evaluated at $s = \lfloor M/2 \rfloor$, which lies in the window whenever the window is nonempty by \cref{prop-window}, together with $N = n (q - \delta)$. The criterion for nonemptiness is \cref{eq-window-nonempty}. For the last assertion put $\delta = 0$: one has $c \lfloor M/2 \rfloor \geq c ( M-1 )/2 = ( n (q-1) - 2c )/2$ and $2 g_\XX - 2 = (n-1)(d-1) - 1 - c$, so the criterion is implied by $n(q-1)/2 - c > (n-1)(d-1) - 1 - c$, which is the stated inequality.
\end{proof}

That $h_{\max} \geq c \lfloor M/2 \rfloor + 1 - g_\XX$ on the totally split locus is \cref{prop-monotone-bound}. The reverse inequality is the assertion that no curve with $\delta > 0$ reaches that value through the excess, and it is not automatic: \cref{thm-quasi} bounds the excess only by $\ell ( K - G_s \wedge A_s )$, which is not small in the saturated zone, and \cref{sec-equiv} exhibits a second mechanism, operating precisely where the window is empty, which produces hulls of dimension $g_\XX$ on curves with $\delta > 0$. It is a theorem under one explicit inequality separating the two, and \cref{exa-star-sharp,exa-star-fails} show that some such hypothesis is needed.

Write
\[
h^{*} \; = \; c \left\lfloor M/2 \right\rfloor + 1 - g_\XX
\]
for the value in question. Two bounds carry the proof, one on each side of the midpoint, and only the second uses the hypothesis: below $M/2$ the hull is bounded by the code itself, hence by a Riemann-Roch space monotone in $s$ which saturates at $s = \lfloor M/2 \rfloor$; above $M/2$ it is bounded by the dual code, whose divisor loses degree as $s$ grows.

\begin{lem}\label{lem-half-degree}
Assume \textup{(H1)}--\textup{(H4)} and $c \lfloor M/2 \rfloor > 2 g_\XX - 2$. Then $\ell ( \lfloor M/2 \rfloor D_\infty ) = h^{*}$, and for every $s$ with $1 \leq s \leq \lfloor M/2 \rfloor$,
\[
\dim \Hull ( \CC_s ) \; \leq \; \dim \CC_s \; \leq \; \ell ( s D_\infty ) \; \leq \; h^{*} .
\]
If $\dim \Hull ( \CC_s ) = h^{*}$ for such an $s$, then $\CC_s$ is self-orthogonal.
\end{lem}

\begin{proof}
The divisor $\lfloor M/2 \rfloor D_\infty$ has degree $c \lfloor M/2 \rfloor > 2 g_\XX - 2$, so it is non-special and $\ell ( \lfloor M/2 \rfloor D_\infty ) = h^{*}$. The first inequality is $\Hull ( \CC_s ) \subseteq \CC_s$, the second is $\dim \CC_s = \ell ( s D_\infty ) - \ell ( s D_\infty - D )$, and the third is $s D_\infty \leq \lfloor M/2 \rfloor D_\infty$. If the outer terms agree then all three are equalities, so $\dim \Hull ( \CC_s ) = \dim \CC_s$ and the containment $\Hull ( \CC_s ) \subseteq \CC_s$ is an equality, that is $\CC_s \subseteq \CC_s^\perp$.
\end{proof}

\begin{thm}\label{tthm-split-max}
Assume \textup{(H1)}--\textup{(H4)}, that $q = p$ is prime, that $\YY \neq \emptyset$, and that
\begin{equation}\label{eq-star}
c \left\lfloor M/2 \right\rfloor \; \geq \; (n-1) d .
\end{equation}
Then
\[
h_{\max} ( \XX ) \; \leq \; h^{*} \; = \; c \left\lfloor M/2 \right\rfloor + 1 - g_\XX \; = \; \ell \left( \left\lfloor M/2 \right\rfloor D_\infty \right) ,
\]
with equality if and only if $\XX$ is totally split, in which case the value is attained at $s = \lfloor M/2 \rfloor$ and at $s = M - \lfloor M/2 \rfloor$; and if $\XX$ is not totally split then $h_{\max} ( \XX ) \leq h^{*} - 1$. For fixed $(n,d,q)$ subject to \cref{eq-star}, equality in this upper bound occurs exactly on the totally split locus. In particular, if that locus is nonempty, it is precisely the maximizing locus and the maximum equals $\ell ( \lfloor M/2 \rfloor D_\infty )$.
\end{thm}

\begin{proof}
Since $2 g_\XX - 2 = nd - n - d - c$ by \cref{eq-genus}, the hypothesis \cref{eq-star} reads $c \lfloor M/2 \rfloor \geq 2 g_\XX + n + c - 2$, that is
\[
h^{*} \; \geq \; g_\XX + n + c - 1 ,
\]
and in particular $c \lfloor M/2 \rfloor > 2 g_\XX - 2$ and $\lfloor M/2 \rfloor \geq 1$, so \cref{lem-half-degree} is available. Write $h (s) = \dim \Hull ( \CC_s )$.

Let first $1 \leq s \leq \lfloor M/2 \rfloor$. \cref{lem-half-degree} gives $h (s) \leq h^{*}$. If $\delta \geq 1$ then equality is impossible: it would make $\CC_s$ self-orthogonal, and \cref{thm-so}, applicable because $s \leq M$, would give $p \mid \delta$, whereas $1 \leq \delta \leq p-1$ since $N = n (q - \delta) \geq 1$. Hence $h (s) \leq h^{*} - 1$ for every such $s$ when $\delta \geq 1$.

Let now $s > \lfloor M/2 \rfloor$, so that $\mu (s) = M - s$ and, by \cref{eq-degA},
\[
\deg A_s \; = \; (n-1) d - n \delta + c \, \mu (s) .
\]
We distinguish three cases according to $\deg ( G_s \wedge A_s ) = c \, \mu (s) - n \delta$, given by \cref{eq-degrees}.

If $c \, \mu (s) - n \delta > 2 g_\XX - 2$ then $t (s) < 0$ by \cref{eq-tofs}, so $s \in W ( \XX )$ and \cref{thm-hull} gives $h (s) = c \, \mu (s) - n \delta + 1 - g_\XX \leq h^{*} - n \delta$, because $\mu (s) \leq \lfloor M/2 \rfloor$.

If $0 \leq c \, \mu (s) - n \delta \leq 2 g_\XX - 2$ then \cref{cor-clifford} gives $h (s) \leq g_\XX + 1 \leq h^{*} - ( n + c - 2 ) \leq h^{*} - 1$, since $n \geq 2$ and $c \geq 1$.

If $c \, \mu (s) - n \delta < 0$ then the display above gives $\deg A_s \leq (n-1) d - 1$, and therefore
\[
\ell ( A_s ) \; \leq \; \max \left\{ \, g_\XX, \; (n-1) d - g_\XX \, \right\} \; = \; (n-1) d - g_\XX :
\]
for $\deg A_s < 0$ the space vanishes, for $0 \leq \deg A_s \leq 2 g_\XX - 2$ Clifford's theorem gives $\ell ( A_s ) \leq \frac{1}{2} \deg A_s + 1 \leq g_\XX$, and for $\deg A_s > 2 g_\XX - 2$ Riemann-Roch gives $\ell ( A_s ) = \deg A_s + 1 - g_\XX \leq (n-1) d - g_\XX$, while $(n-1) d \geq 2 g_\XX$ is the inequality $n + c \geq 2$. Since $\Hull ( \CC_s ) \subseteq \CC_s^\perp = C_L ( D, A_s )$ by \cref{prop-dual}, and since $(n-1) d - g_\XX = g_\XX + n + c - 2$ by \cref{eq-genus}, we get
\[
h (s) \; \leq \; \ell ( A_s ) - \ell ( A_s - D ) \; \leq \; g_\XX + n + c - 2 \; \leq \; h^{*} - 1 .
\]

Collecting the four cases, $h (s) \leq h^{*}$ for every $s \geq 1$, and $h (s) \leq h^{*} - 1$ for every $s \geq 1$ as soon as $\delta \geq 1$. Finally let $\delta = 0$. Then $c \lfloor M/2 \rfloor - n \delta > 2 g_\XX - 2$, so $W ( \XX ) \neq \emptyset$ by \cref{eq-window-nonempty} and $\lfloor M/2 \rfloor \in W ( \XX )$ by \cref{prop-window}, as does $M - \lfloor M/2 \rfloor$, both having $\mu (s) = \lfloor M/2 \rfloor$; \cref{thm-hull} evaluates the hull there as $c \lfloor M/2 \rfloor + 1 - g_\XX = h^{*}$.
\end{proof}

\begin{rem}\label{rem-star}
By \cref{eq-genus} the hypothesis \cref{eq-star} has the three equivalent forms $c \lfloor M/2 \rfloor \geq (n-1) d$, $c \lfloor M/2 \rfloor \geq 2 g_\XX + n + c - 2$ and $h^{*} \geq g_\XX + n + c - 1$. Since $c \lfloor M/2 \rfloor \geq ( cM - c )/2 = \left( n (q-1) - 2c \right)/2$, it is implied by
\[
n (q-1) \; \geq \; 2 \left( (n-1) d + c \right) ,
\]
so it holds for every $q$ large with respect to $d$, and it is only slightly stronger than the criterion $n (q-1) > 2 (n-1)(d-1)$ of \cref{prop-monotone-bound} for the window of a totally split curve to be nonempty. For $n = 2$, $d = 6$, where $q$ is odd by \textup{(H4)}, it reads $q \geq 9$ for every admissible $q$, so it holds for the fields of \cref{tab-strata,tab-scan} and for $\XX_{11}$, and fails for $\XX_7$; for $n = 3$, $d = 4$ it reads $q \geq 7$, so it holds in \cref{tab-picard}. The third form says what \cref{eq-star} excludes. On a stratum with a transitive action the window is empty by \cref{prop-window-empty}, and when $m = d$ the hull reaches $g_\XX$ by \cref{thm-hmax}; \cref{eq-star} forces $h^{*} > g_\XX$, so that mechanism cannot overtake the one of \cref{sec-hull}.
\end{rem}

\begin{exa}\label{exa-star-sharp}
Let $q = 7$, $n = 2$ and $f (x) = x^6 + x^2 + 6$, so that $d = 6$, $c = 2$, $g_\XX = 2$, $M = 5$ and $h^{*} = 3$. For $a \neq 0$ one has $a^6 = 1$ and hence $f (a) = a^2$, a nonzero square, while $f (0) = 6$ is not a square; so $f$ is separable with no root in $\F_7$, \textup{(H1)}--\textup{(H4)} hold, $\Delta = \{ 0 \}$, $\delta = 1$ and $N = 12$. The profile computed from \cref{eq-hull-rank} is
\[
\dim \Hull ( \CC_s ) \; = \; 1, \, 2, \, 3, \, 3, \, 2, \, 1, \, 0, \, 0, \, \dots \qquad s = 1, 2, \dots ,
\]
so $h_{\max} = 3 = h^{*}$ although $\delta = 1$. The window is empty, $\nu (s) = \mu (s) - 1$, and \cref{eq-epsilon-comp} gives $\varepsilon = 1$ at $s = 3$ and $\varepsilon = 2$ at $s = 4$, where $\Phi ( \nu (s) ) = 2$ and $1$: the maximum is reached entirely through the excess, and at both degrees the excess equals its bound $\Phi ( \gamma - \nu (s) )$ of \cref{cor-eps-bound}. Here $c \lfloor M/2 \rfloor = 4 < 6 = (n-1) d$, so \cref{eq-star} fails, and so does the second half of \cref{tthm-split-max}: the maximum is attained off the totally split locus as well as on it. It is not exceeded, and since $t (3) = 0$ the degree $s = 3$ lies in the detection band, where by \cref{cor-vacuous} the equivalence certified is one that holds on every curve of the family.
\end{exa}

\begin{exa}\label{exa-star-fails}
Let $q = 11$, $n = 5$ and $f (x) = x^{10} + 9$, so that $d = 10$, $c = 5$, $g_\XX = 16$, $M = 9$ and $h^{*} = 5$. The fifth powers in $\F_{11}^*$ are $\{ 1, 10 \}$; for $a \neq 0$ one has $a^{10} = 1$ and $f (a) = 10$, a fifth power, while $f (0) = 9$ is not, so \textup{(H1)}--\textup{(H4)} hold with $\Delta = \{ 0 \}$, $\delta = 1$ and $N = 50$. Here $m (f) = 10 = d$ and the criterion of \cref{prop-transitive-criterion} is met with $m = 10$, so \cref{thm-hmax} applies and gives $h_{\max} = g_\XX = 16$, attained at $s = e-1 = 9$ and $s = e = 10$. Thus $h_{\max} = 16 > 5 = h^{*}$ and the inequality of \cref{tthm-split-max} itself fails, \cref{eq-star} failing by a wide margin since $c \lfloor M/2 \rfloor = 20 < 40 = (n-1) d$. The window is empty, as \cref{prop-window-empty} requires of a transitive stratum, and the large hull is produced by an unsymmetric defining set and not by split fibres: it is the second mechanism, on a curve on which the first is switched off.
\end{exa}

The same behaviour persists off the hyperelliptic locus. For $n = 3$ and $d = 4$ one has $c = 1$, $\bw = (1,3,4)$, $g_\XX = 3$, and $\XX$ is a Picard curve; the divisor at infinity is a single rational point and $M = 3(q-1) - 1$. Over $\F_{13}$ one has $M = 35$, $\delta = 13 - N/3$, the window is nonempty precisely when $\delta \leq 4$, that is $N \geq 27$, and by \cref{prop-window} the detection band consists of $g_\XX = 3$ degrees on each side of the window, nonempty precisely when $\delta \leq 5$. \cref{tab-picard} gives the largest hull dimension over all squarefree $f$ of degree $4$ with no root in $\F_{13}$, computed as for \cref{tab-scan} over the full range $1 \leq s < s^+$, one $f$ from each of the $222$ classes under $f(x) \mapsto c^3 f(ax+b)$, three of which have $N = 0$.

\begin{table}[ht]
\caption{Maximal hull dimension for Picard curves over $\F_{13}$, $n=3$, $d=4$, over all curves and all degrees.}\label{tab-picard}
\begin{tabular}{lrrrrrrrrr}
\toprule
$N$ & $3$ & $6$ & $9$ & $12$ & $15$ & $18$ & $21$ & $24$ & $27$ \\
$\delta$ & $12$ & $11$ & $10$ & $9$ & $8$ & $7$ & $6$ & $5$ & $4$ \\
classes of $f$ & $9$ & $20$ & $40$ & $54$ & $40$ & $30$ & $14$ & $9$ & $3$ \\
$\max_{s} \Phi ( \nu (s) )$ & $0$ & $0$ & $0$ & $0$ & $0$ & $0$ & $0$ & $1$ & $3$ \\
$h_{\max}$ & $1$ & $2$ & $2$ & $3$ & $3$ & $4$ & $3$ & $5$ & $4$ \\
\bottomrule
\end{tabular}
\end{table}

Only the last column has a nonempty window, namely $W ( \XX ) = \{ 17, 18 \}$, where \cref{prop-monotone-bound} guarantees $\lfloor M/2 \rfloor - 3 \delta + 1 - g_\XX = 3$; the computed value $4$ is attained on every curve with $N = 27$, at $s = 16$ and $s = 19$, outside the window, so the excess is nonzero there. In the remaining columns the window is empty, \cref{eq-main-dim} predicts nothing, and the gap between the last two rows is pure excess. Unlike \cref{tab-scan}, the maximum is not monotone in $N$: it drops from $4$ to $3$ between $N = 18$ and $N = 21$, and from $5$ to $4$ between $N = 24$ and $N = 27$.

The section closes with what the computations say for two applications, the equivalence problem and the codes built from the family. The cost of the support splitting algorithm on a code with hull of dimension $h$ grows like $q^h$, and the algorithm separates coordinates only up to the orbits of $\PAut$ \cite{se-ssa, se-hull}, so the pair $( \dim \Hull ( \CC_s ), \Aut ( \XX ) )$ controls the difficulty of deciding equivalence in this family. Two distinct mechanisms can make this parameter large. On the totally split locus the window is nonempty and by \cref{thm-hull} the hull is nonzero on an interval of degrees with dimension growing linearly in $\mu (s)$, reaching $c \lfloor M/2 \rfloor + 1 - g_\XX$; for $\XX_{11}$ at $s = 4$ one has $h = 7$ and $q^h > 1.9 \cdot 10^7$, against $q^h \leq 19$ for every curve of \cref{tab-strata}. On a stratum satisfying \cref{prop-transitive-criterion} the window is empty and the hull is instead given exactly by \cref{thm-transitive-hull} and bounded by $(nm-t_2)/2$. Such a stratum is never totally split, so these two mechanisms occur on disjoint loci. In the particular genus-two examples of \cref{tab-strata}, none of the codes is self-orthogonal and the codes are \textup{LCD} for all but at most four degrees; no such uniform bound is asserted for arbitrary curves with empty window.

If $C$ is an $[N,k]_q$ code with $\dim \Hull (C) = h$, the two standard hull-based entanglement-assisted constructions give
\[
[[ N,\, k-h,\, d(C);\, N-k-h ]]_q
\qquad\text{and}\qquad
[[ N,\, N-k-h,\, d(C^\perp);\, k-h ]]_q
\]
\cite{gu-ji-gu}. We use the second form because \cref{cor-distances} gives a designed distance for $\CC_s^\perp$. For $s$ in the window, $s D_\infty$ is non-special and $cs < N$, so
\[
k = cs + 1 - g_\XX, \qquad h = c\,\mu(s) - n\delta + 1 - g_\XX ,
\]
and the construction returns
\[
\left[ \left[ \, N, \;
N - cs - c\,\mu(s) + n\delta + 2g_\XX - 2, \;
\geq cs - 2g_\XX + 2 \, ; \;
cs - c\,\mu(s) + n\delta
\, \right] \right]_q .
\]
Thus the entanglement consumption is $k-h = cs-c\,\mu(s)+n\delta$. On the lower half of the window it equals $n\delta$, while on the upper half it increases by $2c$ when $s$ increases by one. In particular, on a totally split curve the lower half has zero entanglement consumption; there $\CC_s$ is self-orthogonal by \cref{prop-dichotomy}, and the CSS construction returns $[[ N, N - 2k, \geq cs - 2 g_\XX + 2 ]]_q$. This is the regime in which quantum codes from superelliptic curves have been built \cite{el-sh}. The entanglement-assisted construction remains available when the hull is nonzero and, on the window, all of its displayed parameters are explicit. For $\XX_{11}$ the degrees $s = 4$ and $s = 6$ give respectively the CSS code $[[22, 8, \geq 6]]_{11}$ and the entanglement-assisted code $[[22, 6, \geq 10; 6]]_{11}$. On a transitive stratum one has $k=|T_s|$ and $h=|T_s|-|T_s\cap(-T_s)|$, so the entanglement consumption in the second construction is $|T_s\cap(-T_s)|$; since $T_s$ depends only on $(n,d,m,s)$, so does this consumption.

The \textup{LCD} members are the $\CC_s$ with $\Hull ( \CC_s ) = 0$, and by \cref{cor-thresholds} these are all $s \geq s^+$ together with the degrees outside the window at which the excess vanishes. Those with $s \geq s^+$ are degenerate, since there $\CC_s^\perp = 0$ and $d ( \CC_s ) = 1$; the useful ones are those with $s < s^+$. In the genus-two examples of \cref{tab-strata}, all but at most four degrees are \textup{LCD}; for any nondegenerate degree to which both designed-distance bounds apply, \cref{cor-distances} supplies $d ( \CC_s ) \geq N - cs$ and $d ( \CC_s^\perp ) \geq cs - 2 g_\XX + 2$. As these have constant sum $N - 2 g_\XX + 2$, a curve of small genus with many rational points is the favourable case and $s$ selects the trade-off. \textup{LCD} codes over $\F_q$ are of interest as countermeasures to side channel and fault injection attacks, where the relevant parameter is exactly this pair \cite{ca-gu}.

Although the hull profile is not a monomial invariant \cite{ca-me-ta-qi-pe}, it retains arithmetic information about the evaluation realization $\CC_s ( \XX, \YY )$ of \cref{def-code}: on a common window, a single hull dimension recovers the number of affine rational points.

\begin{prop}\label{prop-perm-invariant}
Let $\XX, \XX'$ be superelliptic curves of the same type $(n,d)$ over the same $\F_q$, both satisfying \textup{(H1)}--\textup{(H4)} and both with nonempty window. If, for some degree $s \in W ( \XX ) \cap W ( \XX' )$,
\[
\dim \Hull \bigl( \CC_s ( \XX, \YY ) \bigr)
=
\dim \Hull \bigl( \CC_s ( \XX', \YY' ) \bigr),
\]
then $\delta = \delta'$; equivalently the two curves have the same number $N=n(q-\delta)$ of affine evaluation points.
\end{prop}

\begin{proof}
The integers $M$, and hence $\mu (s)$, and the genus depend only on $(n,d,q)$, by \cref{eq-genus}. Applying \cref{eq-main-dim} to the two curves gives
\[
c \mu (s) - n \delta + 1 - g_\XX
=
c \mu (s) - n \delta' + 1 - g_\XX ,
\]
whence $\delta = \delta'$, and therefore $N=n(q-\delta)=n(q-\delta')$.
\end{proof}

Since $\Aut ( \XX )$ acts monomially on the weighted forms, the natural equivalence question for this family is linear rather than permutation equivalence, and \cref{prop-perm-invariant} says nothing about it: $\dim \Hull$ is not preserved by monomial equivalence when $q \geq 4$ \cite{ca-me-ta-qi-pe}.

\begin{conj}\label{tthm-moduli-invariant}
Let $\XX, \XX'$ be as in \cref{prop-perm-invariant}. If $\CC_s ( \XX, \YY )$ and $\CC_s ( \XX', \YY' )$ are monomially equivalent for every $s$, then $\XX$ and $\XX'$ have the same weighted moduli point, that is $\XX \cong \XX'$ over $\overline{\F}_q$.
\end{conj}

The numerical invariant $\delta$, and hence $N$, is recovered from any hull dimension at a common window degree by \cref{prop-perm-invariant}. By \cref{cor-vacuous}, the divisor-class condition supplied by \cref{cor-detect} is automatic in this family; any finer information visible to the hull profile must therefore come through the excess. The examples support this distinction: \cref{prop-eps-moduli} separates two degrees presenting the same meet, the same join and the same class, while \cref{exa-sameN} separates curves agreeing in all the numerical data. What is not proved is that the whole profile $s \mapsto \varepsilon ( G_s, A_s )$ suffices to reconstruct the weighted moduli point. The computed examples are consistent with the conjecture, but we have not tested monomial equivalence directly, and hull profiles cannot do so: by \cite{ca-me-ta-qi-pe} every code over $\F_q$ with $q \geq 4$ is monomially equivalent to an \textup{LCD} one, so the conjecture cannot be reduced to an assertion about $\dim \Hull$ alone.

%-----------------------------------------------------------------------------------------
%-----------------------------------------------------------------------------------------
\section{Split ramification and maximal curves}\label{sec-splitram}

Hypothesis \textup{(H1)} asks that $f$ have no root in $\F_q$. It is used twice, to make $\pi$ unramified over $\A^1 ( \F_q )$ and to keep $\Ram$ away from $\Supp D$, and by \cref{rem-hermitian} it is exactly what excludes the Fermat and Hermitian curves from \cref{sec-hull}. Those are the curves on which the window is widest and \cref{thm-hull} sharpest, so the exclusion is expensive. This section removes it at the opposite extreme, where $f$ splits completely over $\F_q$ and the ramification is as rational as it can be, and the whole of \cref{sec-hull} survives under a single substitution.

The reason is one sign. When no root of $f$ is rational the ramification enters the canonical divisor \cref{eq-eta} with coefficient $n-1$, and $\Ram$ contributes to the join. When every root is rational the fibre over each of them is $n \, Q_a$ rather than a reduced divisor, the extra $n \, Q_a$ cancels most of the $(n-1) \Ram$ coming from $\dv (dx)$, and $\Ram$ enters with coefficient $-1$ instead. It therefore migrates from the join to the meet, where it behaves exactly as $E$ does, and the effect on every count in \cref{sec-hull} is to replace $n \delta$ by $d + n \delta$. The effect on the order relation between $G_s$ and $A_s$ is not a substitution but a reversal: no code of the family is forced self-orthogonal, and instead every code of degree $s \geq M/2$ contains its dual.

%\subsection{The split hypotheses}\label{ssec-splitstanding}

Throughout this section we replace \textup{(H1)}--\textup{(H3)} by the following, and retain \textup{(H4)}.

\begin{itemize}
\item[\textup{(H1$'$)}] $p \nmid n$, the polynomial $f \in \F_q [x]$ is separable of degree $d \geq 3$, and $f$ splits into distinct linear factors over $\F_q$.
\item[\textup{(H2$'$)}] $\YY$ is the set of all $\F_q$-rational points of $\XX$ lying over $\A^1 ( \F_q )$ and not lying on $\Ram$, and $D = \sum_{P \in \YY} P$, $N = \deg D = | \YY |$.
\item[\textup{(H3$'$)}] $\Delta = \left\{ \, a \in \F_q \; : \; f (a) \neq 0, \; f (a) \notin ( \F_q^* )^n \, \right\}$, $\delta = | \Delta |$, and $E = \pi^* ( \Delta )$.
\end{itemize}

Write $R = \{ a \in \F_q : f (a) = 0 \}$, a set of $d$ elements by \textup{(H1$'$)}, so that $\Ram = \sum_{a \in R} Q_a$ and each $Q_a$ is $\F_q$-rational, being the unique point of $\XX$ over $a$.

\begin{lem}\label{lem-splitbasic}
Assume \textup{(H1$'$)}--\textup{(H3$'$)}. Then $\Ram$, $E$ and $D_\infty$ are reduced and have pairwise disjoint supports, all three are disjoint from $\Supp D$, and
\[
\deg \Ram = d, \qquad \deg D_\infty = c, \qquad \deg E = n \delta, \qquad N = n \left( q - d - \delta \right) .
\]
Moreover
\[
\dv \left( e_\Delta (x) \right) = E - \frac{n \delta}{c} D_\infty ,
\]
$e_\Delta (P) \neq 0$ and $y (P) \neq 0$ for every $P \in \YY$, and $g_\XX \geq 1$.
\end{lem}

\begin{proof}
For $a \in R$ the fibre is $\pi^* (a) = n \, Q_a$ with $Q_a$ rational, and $Q_a \notin \YY$ by \textup{(H2$'$)}; for $a \notin R$ the fibre is reduced of degree $n$ and consists of the pairs $(a,b)$ with $b^n = f (a) \neq 0$, all rational if $a \notin \Delta$ and none rational if $a \in \Delta$, exactly as in \cref{lem-pullback}, since \textup{(H4)} places $\zeta_n$ in $\F_q$. Hence $\Ram$ and $E$ are reduced with disjoint supports lying over $R$ and over $\Delta$ respectively, $\Supp D$ lies over $\F_q \setminus ( R \cup \Delta )$, and $\Supp D_\infty$ lies over $x = \infty$; the degrees and the count $N = n ( q - d - \delta )$ follow. For $a \in \Delta$ the point $a$ is not a root of $f$, so $\dv (x-a) = \pi^* (a) - \tfrac{n}{c} D_\infty$. Summing over $a \in \Delta$ gives $\dv ( e_\Delta (x) ) = E - \tfrac{n \delta}{c} D_\infty$. The function $e_\Delta$ vanishes only over $\Delta$ and, by \cref{eq-divisors}, $y$ vanishes only on $\Ram$, so neither vanishes on $\YY$. The last assertion is the argument of \cref{lem-standing}, which uses only \cref{eq-genus}.
\end{proof}

%\subsection{The canonical divisor and the dictionary}\label{ssec-splitdict}

\begin{prop}\label{prop-splitdual}
Assume \textup{(H1$'$)}--\textup{(H3$'$)} and \textup{(H4)}. Then the differential $\eta = - dx / ( x^q - x )$ satisfies $v_P ( \eta ) = -1$ and $\Res_P ( \eta ) = 1$ for every $P \in \Supp D$, and
\begin{equation}\label{eq-splitK}
K \; = \; \dv ( \eta ) \; = \; - \Ram - D - E + M \, D_\infty , \qquad M = \frac{(q-1)n}{c} - 1 .
\end{equation}
Consequently, for every $s \geq 1$,
\begin{equation}\label{eq-splitdual}
\CC_s^\perp \; = \; C_L ( D, A_s ) , \qquad A_s \; = \; - \Ram - E + ( M - s ) D_\infty .
\end{equation}
\end{prop}

\begin{proof}
A point $P \in \Supp D$ lies over some $a \in \F_q \setminus ( R \cup \Delta )$ and is unramified, so the computation of \cref{prop-dual} applies verbatim and gives $v_P ( \eta ) = -1$ and $\Res_P ( \eta ) = 1$. For \cref{eq-splitK}, the fibres computed in \cref{lem-splitbasic} give
\[
\dv ( x^q - x ) \; = \; \sum_{a \in \F_q} \pi^* (a) - \frac{qn}{c} D_\infty \; = \; D + E + n \Ram - \frac{qn}{c} D_\infty ,
\]
the term $n \Ram$ being the only change from \cref{prop-dual}, and subtracting this from $\dv (dx) = (n-1) \Ram - \left( \frac{n}{c} + 1 \right) D_\infty$ of \cref{eq-divisors} yields
\[
\dv ( \eta ) \; = \; (n-1) \Ram - n \Ram - D - E + \left( \frac{qn}{c} - \frac{n}{c} - 1 \right) D_\infty ,
\]
which is \cref{eq-splitK}. Then \cref{eq-splitdual} is \cref{eq-duality} with $G = G_s = s D_\infty$, since $D - G_s + \dv ( \eta ) = A_s$. As a check, $\deg K = -d - N - n \delta + cM = nd - n - d - c = 2 g_\XX - 2$ by \cref{eq-genus} and \cref{lem-splitbasic}.
\end{proof}

\begin{lem}\label{lem-splitmeetjoin}
Assume \textup{(H1$'$)}--\textup{(H3$'$)} and \textup{(H4)}. For every $s \geq 1$,
\begin{equation}\label{eq-splitmeetjoin}
\begin{split}
G_s \wedge A_s \; &= \; \mu (s) \, D_\infty - \Ram - E , \\
G_s \vee A_s \; &= \; \max \{ s, M-s \} \, D_\infty ,
\end{split}
\end{equation}
and therefore
\begin{equation}\label{eq-splitdegrees}
\begin{split}
\deg ( G_s \wedge A_s ) \; &= \; c \, \mu (s) - d - n \delta , \\
\deg ( G_s \vee A_s ) \; &= \; c \max \{ s, M - s \} .
\end{split}
\end{equation}
\end{lem}

\begin{proof}
By \cref{lem-splitbasic} the divisors $\Ram$, $E$ and $D_\infty$ have pairwise disjoint supports, so the coefficientwise minimum and maximum may be computed one support at a time. Along $\Supp \Ram$ the coefficients of $G_s$ and $A_s$ are $0$ and $-1$, along $\Supp E$ they are $0$ and $-1$, along $\Supp D_\infty$ they are $s$ and $M-s$, and elsewhere both vanish. This gives \cref{eq-splitmeetjoin}, and \cref{eq-splitdegrees} follows from \cref{lem-splitbasic}.
\end{proof}

%The comparison with \cref{eq-meetjoin} is the whole of this section. There $\Ram$ appeared in the join with coefficient $n-1$; here it appears in the meet with coefficient $-1$, alongside $E$ and with the same effect on every degree. Accordingly we redefine the shifted argument.

\begin{defn}\label{def-splitnu}
Under \textup{(H1$'$)}--\textup{(H3$'$)} and \textup{(H4)} put
\[
\nu (s) \; = \; \mu (s) - \frac{d + n \delta}{c} \; \in \; \Z ,
\]
an integer because $c$ divides $d$ and $c$ divides $n$.
\end{defn}

\begin{lem}\label{lem-splitshift}
Assume \textup{(H1$'$)}--\textup{(H3$'$)} and \textup{(H4)}, let $\mu \in \Z$ and put $\nu = \mu - \frac{d + n \delta}{c}$. Then multiplication by $y \, e_\Delta (x)$ is an isomorphism
\[
L ( \nu D_\infty ) \; \xrightarrow{\ \sim\ } \; L \left( \mu D_\infty - \Ram - E \right) ,
\]
and the diagonal matrix $\operatorname{diag} \left( y (P) e_\Delta (P) \right)_{P \in \YY}$ carries $C_L ( D, \nu D_\infty )$ onto $C_L \left( D, \mu D_\infty - \Ram - E \right)$. In particular the two codes are monomially equivalent and
\[
\ell \left( \mu D_\infty - \Ram - E \right) \; = \; \Phi ( \nu ) \; = \; \sum_{j=0}^{n-1} \max \left\{ 0, \; \left\lfloor \frac{c \nu - dj}{n} \right\rfloor + 1 \right\} .
\]
\end{lem}

\begin{proof}
By \cref{eq-divisors} and \cref{lem-splitbasic},
\[
\dv \left( y \, e_\Delta (x) \right) \; = \; \Ram + E - \frac{d + n \delta}{c} D_\infty ,
\]
so for $h \in \F_q ( \XX )^*$ the function $h' = h \, y \, e_\Delta (x)$ satisfies $\dv ( h' ) + \mu D_\infty - \Ram - E = \dv (h) + \nu D_\infty$, and $h \mapsto h'$ is an $\F_q$-linear bijection $L ( \nu D_\infty ) \to L ( \mu D_\infty - \Ram - E )$. It is compatible with evaluation up to the scalars $y (P) e_\Delta (P)$, which are nonzero by \cref{lem-splitbasic}. The last formula is \cref{eq-ell} applied to $\nu D_\infty$, with the convention that the right hand side vanishes for $\nu < 0$.
\end{proof}

\begin{thm}\label{thm-splitdict}
Assume \textup{(H1$'$)}--\textup{(H3$'$)} and \textup{(H4)}. Then every statement of \cref{sec-hull} from \cref{prop-dual} onwards that involves the pair $( G_s, A_s )$ only through the meet, the join and their degrees holds verbatim, with $n \delta$ replaced throughout by $d + n \delta$ and with $\nu (s)$ as in \cref{def-splitnu}. The statements that use the order relation between $G_s$ and $A_s$, or the expression of $\deg A_s$ in the data of $f$, do not, and they are the following. \cref{prop-dichotomy} and \cref{cor-asym} fail, the order relation being reversed, and are replaced by \cref{thm-dualcontaining} below. In \cref{cor-distances} the first expression becomes $\deg A_s = c ( M - s ) - d - n \delta$, the second expression and the two designed distance bounds surviving unchanged. In \cref{cor-thresholds} the threshold becomes $s^{+} = M - \frac{d + n \delta}{c} + 1$. Explicitly, writing
\[
t (s) \; = \; 2 g_\XX - 2 - c \, \mu (s) + d + n \delta \; = \; 2 g_\XX - 2 - c \, \nu (s) ,
\]
one has $\deg A_s = N + 2 g_\XX - 2 - cs$; the window and detection band are obtained from \cref{prop-window} after replacing $n \delta$ by $d + n \delta$, namely with $\mu_0 = \left\lfloor \frac{2 g_\XX - 2 + d + n \delta}{c} \right\rfloor + 1$ and $\mu_1 = \left\lceil \frac{g_\XX - 1 + d + n \delta}{c} \right\rceil$ and with the same truncation at $\lfloor M/2 \rfloor$ in the detection band; and for every $s$ with $c \, \nu (s) < N$
\begin{equation}\label{eq-splitquasi}
\dim \Hull ( \CC_s ) \; = \; \Phi \left( \nu (s) \right) + \varepsilon ( G_s, A_s ) .
\end{equation}
On the window one has $\varepsilon ( G_s, A_s ) = 0$ and
\begin{equation}\label{eq-splitmain}
\Hull ( \CC_s ) \; = \; C_L \left( D, \mu (s) D_\infty - \Ram - E \right) , \qquad \dim \Hull ( \CC_s ) \; = \; c \, \mu (s) - d - n \delta + 1 - g_\XX ,
\end{equation}
and $s \mapsto \dim \Hull ( \CC_s )$ is symmetric about $M/2$ whenever the excess vanishes identically, any failure of that symmetry being carried by the excess alone.
\end{thm}

\begin{proof}
\cref{lem-splitbasic} replaces \cref{lem-standing} and \cref{lem-pullback}, \cref{prop-splitdual} replaces \cref{prop-dual}, \cref{lem-splitmeetjoin} replaces \cref{lem-meetjoin} and \cref{lem-splitshift} replaces \cref{lem-shift}; in each case the statement is the same with $n \delta$ replaced by $d + n \delta$, because $\Ram$ and $E$ enter \cref{eq-splitmeetjoin} with the same coefficient and their degrees add. The identity $\deg A_s = N + 2 g_\XX - 2 - cs$ follows from $G_s + A_s = D + K$ and $\deg K = 2 g_\XX - 2$, and $t(s) = 2 g_\XX - 2 - \deg ( G_s \wedge A_s )$ is \cref{eq-splitdegrees}. Everything downstream in \cref{sec-hull} except \cref{prop-dichotomy} and \cref{cor-asym} is deduced from these four statements together with results of \cref{sec-window}, which are unconditional; those two use the sign of the coefficient of $A_s - G_s$ along $\Supp \Ram$, which is now negative, and \cref{thm-dualcontaining} treats them. The modified expressions are immediate from \cref{eq-splitdual}: $\deg A_s = c ( M - s ) - d - n \delta$, negative precisely when $s \geq M - \frac{d + n \delta}{c} + 1$, which is the new $s^{+}$. In particular \cref{eq-splitquasi} is \cref{thm-meet} with \cref{lem-splitshift}, the vanishing of the excess on the window is \cref{cor-window}, \cref{eq-splitmain} is \cref{thm-hull}, and the symmetry statement is \cref{prop-asym}, whose proof uses only the invariance of \cref{eq-splitmeetjoin} under $s \mapsto M-s$.
\end{proof}

%\subsection{The classes in play}\label{ssec-splitclasses}

The collapse of \cref{lem-classes} persists, with the same proof and the same consequence, so that \cref{cor-detect} is no more informative here than there.

\begin{lem}\label{lem-splitclasses}
Assume \textup{(H1$'$)}--\textup{(H3$'$)} and \textup{(H4)}, and put $\gamma = \frac{2 g_\XX - 2}{c}$. Then, in $\Pic ( \XX )$,
\[
\begin{split}
\Ram \; &\sim \; \frac{d}{c} \, D_\infty , \qquad E \; \sim \; \frac{n \delta}{c} \, D_\infty , \qquad D \; \sim \; \frac{N}{c} \, D_\infty , \qquad K \; \sim \; \gamma \, D_\infty , \\
G_s \wedge A_s \; &\sim \; \nu (s) \, D_\infty , \qquad K - G_s \wedge A_s \; \sim \; \frac{t (s)}{c} \, D_\infty ,
\end{split}
\]
and consequently $\varepsilon ( G_s, A_s ) \leq \Phi \left( \gamma - \nu (s) \right)$ for every $s \geq 1$.
\end{lem}

\begin{proof}
The first two are \cref{eq-divisors} and \cref{lem-splitbasic}. For the third, the proof of \cref{prop-splitdual} gives $D \sim \frac{qn}{c} D_\infty - E - n \Ram \sim \frac{n ( q - \delta - d )}{c} D_\infty = \frac{N}{c} D_\infty$. For the fourth, \cref{eq-splitK} makes $K$ equivalent to $\lambda D_\infty$ with
\[
\lambda \; = \; - \frac{d}{c} - \frac{n ( q - d - \delta )}{c} - \frac{n \delta}{c} + \frac{(q-1)n}{c} - 1 \; = \; \frac{nd - n - d}{c} - 1 \; = \; \gamma .
\]
The last two follow from \cref{eq-splitmeetjoin} and \cref{lem-splitshift}, and the bound on the excess is \cref{eq-eps-h1} together with \cref{eq-ell}, exactly as in \cref{cor-eps-bound}.
\end{proof}

%\subsection{Dual-containing codes}\label{ssec-dualcontaining}

\cref{prop-dichotomy} said that under \textup{(H1)} the relation $G_s \leq A_s$ holds precisely on the totally split locus and precisely in the lower half of the degree range, so that the self-orthogonal members are exactly there. Under \textup{(H1$'$)} the inequality is reversed, and it holds on the whole locus with no arithmetic condition whatever.

\begin{thm}\label{thm-dualcontaining}
Assume \textup{(H1$'$)}--\textup{(H3$'$)} and \textup{(H4)} and let $s \geq 1$. Then $G_s \not\leq A_s$ for every $s$, so no $\CC_s$ is self-orthogonal by the divisorial criterion, while
\[
A_s \; \leq \; G_s \quad \Longleftrightarrow \quad 2s \; \geq \; M .
\]
When $2s \geq M$ the code $\CC_s$ contains its dual, $\varepsilon ( G_s, A_s ) = 0$, and
\[
\Hull ( \CC_s ) \; = \; \CC_s^\perp , \qquad \dim \Hull ( \CC_s ) \; = \; N - \dim \CC_s .
\]
\end{thm}

\begin{proof}
By \cref{eq-splitdual}, $A_s - G_s = - \Ram - E + ( M - 2s ) D_\infty$. Its coefficient along $\Supp \Ram$ is $-1$, and $\Ram \neq 0$ because $d \geq 3$, so $A_s - G_s \geq 0$ fails for every $s$ and $G_s \not\leq A_s$. Its coefficients along $\Supp \Ram$ and $\Supp E$ are at most $0$, so $A_s - G_s \leq 0$ if and only if $M - 2s \leq 0$. In that case $G_s \wedge A_s = A_s$ and $G_s \vee A_s = G_s$, so \cref{prop-degenerate} applies with the two divisors exchanged and gives $\varepsilon ( G_s, A_s ) = 0$ together with $\Hull ( \CC_s ) = C_L ( D, A_s ) = \CC_s^\perp$; the dimension is then $N - \dim \CC_s$.
\end{proof}

This is the degenerate zone $\textup{(Z0)}$ of \cref{defn-zones}, entered from the other side. It occupies the upper half of the range unconditionally, so on this locus the hull is large for a reason that has nothing to do with the number of rational points, and by \cref{prop-degenerate} it carries no information about the moduli point there; what remains of interest is the lower half, where the excess is what \cref{eq-splitquasi} leaves undetermined. There it is again read from the profile alone: \cref{prop-asym} holds verbatim and $\varepsilon ( G_{M-s}, A_{M-s} ) = 0$ for $s \leq \lfloor M/2 \rfloor$ by \cref{thm-dualcontaining}, so $\varepsilon ( G_s, A_s ) = \dim \Hull ( \CC_s ) - \dim \Hull ( \CC_{M-s} )$ for every such $s$, the mirror of \cref{cor-asym}; on both curves of \cref{tab-hermitian} this gives $1, 1, 0$ at $s = 1, 2, 3$. The construction that the upper half does supply is the one that \cref{prop-dichotomy} supplied in the lower half under \textup{(H1)}, with the two ingredients exchanged.

\begin{cor}\label{cor-splitcss}
Assume \textup{(H1$'$)}--\textup{(H3$'$)} and \textup{(H4)} and let $s$ satisfy $M \leq 2s$ and $2 g_\XX - 2 < cs < N$. Then $\dim \CC_s = cs + 1 - g_\XX$, the code $\CC_s$ contains its dual, and the \textup{CSS} construction applied to $\CC_s^\perp \subseteq \CC_s$ returns a quantum code with parameters
\[
\left[\left[ \, N, \; 2cs + 2 - 2 g_\XX - N, \; \geq N - cs \, \right]\right]_q ,
\]
requiring no entanglement. Its designed distance decreases and its dimension increases with $s$, while the quantum dimension plus twice the designed distance is constant.
\end{cor}

\begin{proof}
Since $cs > 2 g_\XX - 2$ the divisor $s D_\infty$ is non-special and $\ell ( s D_\infty ) = cs + 1 - g_\XX$, while $cs < N$ gives $\ell ( s D_\infty - D ) = 0$, so $\dim \CC_s = cs + 1 - g_\XX$. The containment is \cref{thm-dualcontaining}. For a pair $\CC^\perp \subseteq \CC$ with $\dim \CC = k$ the \textup{CSS} construction returns $[[ N, 2k - N, \geq d ( \CC ) ]]_q$, and $d ( \CC_s ) \geq N - cs$ by the argument of \cref{cor-distances}. Substituting $k$ gives the parameters, and $( 2cs + 2 - 2 g_\XX - N ) + 2 ( N - cs ) = N + 2 - 2 g_\XX$ is independent of $s$.
\end{proof}

%\subsection{Hermitian curves and their quotients}\label{ssec-hermitian}

The hypotheses \textup{(H1$'$)}--\textup{(H3$'$)} are met, with $\delta = 0$, by the Hermitian curve in its Kummer model and by every quotient of it by a subgroup of $\langle \tau \rangle$.

\begin{thm}\label{thm-hermitian}
Let $q = q_0^2$, let $n \geq 2$ divide $q_0 + 1$, and let
\[
\XX_n \; : \; y^n \; = \; - \left( x^{q_0 + 1} + 1 \right)
\]
over $\F_q$. Then $d = q_0 + 1$, $c = n$, \textup{(H1$'$)}--\textup{(H3$'$)} and \textup{(H4)} hold with $\delta = 0$, and
\[
g_\XX = \frac{(n-1)(q_0-1)}{2} , \qquad N = n \left( q_0^2 - q_0 - 1 \right) , \qquad M = q_0^2 - 2 .
\]
The curve $\XX_n$ is the quotient of the Hermitian curve $\XX_{q_0+1}$ by the subgroup of order $\frac{q_0+1}{n}$ of $\langle \tau \rangle$, all $d$ points of $\Ram$ and all $c$ points of $D_\infty$ are $\F_q$-rational, and $\XX_n$ is $\F_q$-maximal. Moreover
\[
t (s) \; = \; n \left( q_0 - 1 - \mu (s) \right) , \qquad W ( \XX_n ) \; = \; \left\{ \, s \; : \; q_0 \leq s \leq q_0^2 - 2 - q_0 \, \right\} ,
\]
nonempty for $q_0 \geq 3$, and on the window
\[
\dim \Hull ( \CC_s ) \; = \; n \, \mu (s) - q_0 - \frac{(n-1)(q_0-1)}{2} .
\]
\end{thm}

\begin{proof}
Write $\mathrm{N}_{q_0} (x) = x^{q_0+1}$ for the norm $\F_{q}^* \to \F_{q_0}^*$, which is surjective with kernel of order $q_0+1$. The roots of $f (x) = - ( x^{q_0+1} + 1 )$ are the solutions of $\mathrm{N}_{q_0} (x) = -1$, of which there are exactly $q_0+1$ in $\F_q$ and all are distinct, so $f$ is separable of degree $d = q_0+1$ and splits over $\F_q$: this is \textup{(H1$'$)}, and $p \nmid n$ because $n \mid q_0 + 1$ and $q_0$ is a power of $p$. As $n \mid q_0+1 \mid q_0^2 - 1$ we have \textup{(H4)}, and $c = \gcd (n, q_0+1) = n$. For $a \in \F_q$ with $f (a) \neq 0$ one has $f (a) \in \F_{q_0}^* = ( \F_q^* )^{q_0+1} \subseteq ( \F_q^* )^n$, the containment because $n \mid q_0+1$; hence $\Delta = \emptyset$ and $\delta = 0$. The genus is \cref{eq-genus} with $c = n$, and $N$ and $M$ are \cref{lem-splitbasic} and the definition of $M$. The map $(x,y) \mapsto \left( x, y^{(q_0+1)/n} \right)$ carries $\XX_{q_0+1}$ onto $\XX_n$ and identifies the latter with the quotient by the subgroup of $\langle \tau \rangle$ of order $\frac{q_0+1}{n}$. The points of $\Ram$ are rational because the roots of $f$ are; the $c = n$ points of $D_\infty$ correspond to the $n$-th roots of $-1$, which lie in $\F_q$ because $-1 \in \F_{q_0}^* = ( \F_q^* )^{q_0+1} \subseteq ( \F_q^* )^n$. Hence
\[
\# \XX_n ( \F_q ) \; = \; N + d + c \; = \; n q_0^2 - n q_0 + q_0 + 1 \; = \; q + 1 + 2 g_\XX \, q_0 ,
\]
so $\XX_n$ is maximal. Finally $t (s) = 2 g_\XX - 2 - n \mu (s) + q_0 + 1 = n ( q_0 - 1 ) - n \mu (s)$ by \cref{thm-splitdict} and the value of $g_\XX$, so $t (s) < 0$ if and only if $\mu (s) \geq q_0$, which is the stated interval; it is nonempty exactly when $q_0 \leq q_0^2 - 2 - q_0$, that is $q_0 \geq 3$. The value of the hull on the window is \cref{eq-splitmain} with $d + n \delta = q_0 + 1$.
\end{proof}

For $q_0 = 3$ the theorem gives two curves, $n = 2$ and $n = 4$, the second being the Hermitian curve itself. \cref{tab-hermitian} records their hull profiles, computed by \cref{eq-hull-rank}, against \cref{thm-splitdict}.

\begin{table}[ht]
\caption{The maximal curves $\XX_n : y^n = - ( x^4 + 1 )$ over $\F_9$, $M = 7$. Degrees with $2s \geq M$ are dual-containing by \cref{thm-dualcontaining}, and there $\dim \Hull ( \CC_s ) = N - \dim \CC_s$.}\label{tab-hermitian}
\begin{tabular}{lrrrrrrr}
\toprule
\multicolumn{8}{l}{$\XX_4$ Hermitian, \; $g_\XX = 3$, \; $N = 20$, \; $W = \{ 3,4 \}$} \\
$s$ & $1$ & $2$ & $3$ & $4$ & $5$ & $6$ & $7$ \\
\midrule
$\nu (s)$                  & $0$ & $1$ & $2$ & $2$ & $1$ & $0$ & $-1$ \\
$t (s)$                    & $4$ & $0$ & $-4$ & $-4$ & $0$ & $4$ & $8$ \\
$\Phi ( \nu (s) )$         & $1$ & $3$ & $6$ & $6$ & $3$ & $1$ & $0$ \\
$\dim \CC_s$               & $3$ & $6$ & $10$ & $14$ & $17$ & $19$ & $20$ \\
$\dim \Hull ( \CC_s )$     & $2$ & $4$ & $6$ & $6$ & $3$ & $1$ & $0$ \\
$\varepsilon ( G_s, A_s )$ & $1$ & $1$ & $0$ & $0$ & $0$ & $0$ & $0$ \\
\midrule
\multicolumn{8}{l}{$\XX_2$, \; $g_\XX = 1$, \; $N = 10$, \; $W = \{ 3,4 \}$} \\
$s$ & $1$ & $2$ & $3$ & $4$ & $5$ & $6$ & $7$ \\
\midrule
$\nu (s)$                  & $-1$ & $0$ & $1$ & $1$ & $0$ & $-1$ & $-2$ \\
$t (s)$                    & $2$ & $0$ & $-2$ & $-2$ & $0$ & $2$ & $4$ \\
$\Phi ( \nu (s) )$         & $0$ & $1$ & $2$ & $2$ & $1$ & $0$ & $0$ \\
$\dim \CC_s$               & $2$ & $4$ & $6$ & $8$ & $9$ & $10$ & $10$ \\
$\dim \Hull ( \CC_s )$     & $1$ & $2$ & $2$ & $2$ & $1$ & $0$ & $0$ \\
$\varepsilon ( G_s, A_s )$ & $1$ & $1$ & $0$ & $0$ & $0$ & $0$ & $0$ \\
\bottomrule
\end{tabular}
\end{table}

Both profiles are asymmetric about $M/2 = 7/2$, and by \cref{thm-splitdict} the asymmetry is excess: the meet, the join and, by \cref{lem-splitclasses}, the class of $K - G_s \wedge A_s$ agree at $s$ and at $M-s$, while $\varepsilon = 1$ at $s = 1, 2$ and $\varepsilon = 0$ at $s = 5, 6$. The degrees $s = 2$ and $s = 5$ have $t (s) = 0$ and lie in the detection band; the class they present is trivial on both curves by \cref{lem-splitclasses}, and what separates the two degrees is again decomposability alone. The bound $\varepsilon ( G_s, A_s ) \leq \Phi ( \gamma - \nu (s) )$ of \cref{lem-splitclasses} has $\gamma = 1$ for $\XX_4$ and $\gamma = 0$ for $\XX_2$, and is attained at $s = 2$ on both.

By \cref{cor-splitcss} the useful dual-containing degrees for $\XX_4$ are those with $4 \leq s < 5$, giving the single code $[[ 20, 8, \geq 4 ]]_9$. The construction is more productive as $q_0$ grows: for $q_0 = 4$ and $n = 5$ one has $g_\XX = 6$, $N = 55$, $M = 14$, and the degrees $s = 7, \dots, 10$ return
\[
[[ 55, 5, \geq 20 ]]_{16} , \qquad [[ 55, 15, \geq 15 ]]_{16} , \qquad [[ 55, 25, \geq 10 ]]_{16} , \qquad [[ 55, 35, \geq 5 ]]_{16} ,
\]
a family on a maximal curve, of constant $k + 2d$, and with no entanglement. This is the mirror of the regime described at the end of \cref{sec-comput}, where the \textup{CSS} construction was available in the lower half of the range and only on the totally split locus.

%\subsection{The split locus}\label{ssec-splitlocus}

\cref{thm-hermitian} produces curves satisfying \textup{(H1$'$)} with $\delta = 0$, which is the analogue here of \cref{def-split} and is again the condition under which \cref{eq-splitmain} is sharpest. We do not know how far the list extends.

\begin{prob}\label{prob-splitlocus}
Determine the pairs $(n, f)$ over $\F_q$ satisfying \textup{(H1$'$)} and \textup{(H4)} for which $\delta = 0$, that is, for which $f$ splits into $d$ distinct linear factors over $\F_q$ and $f (a) \in ( \F_q^* )^n$ for every $a \in \F_q$ outside the roots of $f$. Are they all quotients of Hermitian curves, and are they all maximal?
\end{prob}

A necessary condition is immediate: the values of $f$ on $\F_q$ lie in $\{ 0 \} \cup ( \F_q^* )^n$, a set of $1 + \frac{q-1}{n}$ elements, and $f$ takes each value at most $d$ times, so $q \leq d \left( 1 + \frac{q-1}{n} \right)$, that is
\[
d \; \geq \; \frac{qn}{q + n - 1} .
\]
For the curves of \cref{thm-hermitian} this reads $q_0 + 1 \geq q_0^2 n / ( q_0^2 + n - 1 )$, and for the Hermitian curve itself, $n = q_0 + 1$, it reads $q_0 + 1 \geq q_0$, so that curve sits just above the bound; a curve meeting it with equality would have $f$ taking every value of $( \F_q^* )^n$ exactly $d$ times, and $f$ would be a planar-type map onto the $n$-th powers. Whether that forces the Hermitian shape is the content of \cref{prob-splitlocus}.

%--------------------------------------------------------------------------
\subsection*{Data availability}
The scripts used for the computations of \cref{sec-equiv,sec-comput,sec-splitram}, which compute $\dim \Hull ( \CC_s )$ by \cref{eq-hull-rank} over $\F_q$, are available as ancillary files with the arXiv version of this paper, \href{https://arxiv.org/abs/2608.02850}{arXiv:2608.02850}.

\bibliographystyle{amsplain}
\bibliography{sh-140}

\end{document}